\pdfoutput=1

\documentclass[
	bigMargin = false,
	font = palatino,
	final
	]{MyClass}
\usepackage{myMathHeader}
\RenewDocumentCommand\difFracAt{O{1} m m m}{
	\difFrac[#1]{#2}{#3}						
	\mathchoice							
		{\raisebox{0.15ex}{\kern 1pt \rule[-2.15ex]{0.7pt}{5.4ex} \kern -2pt}_{\raisebox{0.1ex}{{$\scriptstyle #4$}}}}   
		{\scalebox{0.9}{\Big|}_{\raisebox{1.5pt}{{$\scriptscriptstyle #4$}}}}
		{UNDEFINED-C}
		{UNDEFINED-D}
	}
\RenewDocumentCommand{\setSymbol}{o}{
	\nonscript \, \mid 
	\allowbreak
	\nonscript \,
	\mathopen{ }
}
\MyNewMathOperator{\GalileanGroup}					{command={\mathrm{Gal}}, sort={Gal}, display={$\GalileanGroup$}, description={Galilean group}}

\usepackage[arrow, matrix, curve]{xy}

\usepackage{framed}	
 
\Title{Norm-squared of the momentum map in infinite dimensions with 
applications to K\"ahler geometry and symplectic connections}

\Abstract{
We initiate the study of the norm-squared of the momentum map as a 
rigorous tool in infinite dimensions.
In particular, we calculate the Hessian at a critical point, 
show that it is positive semi-definite along the complexified orbit, 
and determine a decomposition of the stabilizer under the complexified
action. We apply these results to the action of the group of
symplectomorphisms on the spaces of compatible almost complex 
structures and of symplectic connections. In the former case, we 
extend results of Calabi to not 
necessarily integrable almost complex structures that are extremal in a 
relative sense. 
In both cases, the momentum map is not equivariant, which gives 
rise to new phenomena and opens up new avenues for interesting applications. 
For example, using the prequantization construction, we obtain new central extensions 
of the group of symplectomorphisms that are encoding geometric information 
of the underlying finite-dimensional manifold.
}



\Institution{sjtud}{School of Mathematical Sciences and
Shanghai Frontier Science Center of Modern Analysis, Shanghai 
Jiao Tong University, 800 Dongchuan Road, 200240 China}

\Institution{sjtur}{School of Mathematical Sciences,
Ministry of Education Laboratory of Scientific Computing (MOE-LSC),
and  Shanghai Frontier Science Center of Modern Analysis, Shanghai 
Jiao Tong University, 800 Dongchuan Road, 200240 China, and 
Section de Math\'ematiques, Ecole Polytechnique F\'ed\'erale de 
Lausanne, 1500 Lausanne, Switzerland.}

\Author[
	affiliation = sjtud,
	email = {diez@sjtu.edu.cn, research@tobiasdiez.de},
	funding = {Partially supported by National Natural Science 
	Foundation of China grant 12350410357.}
	]{diez}{Tobias Diez}
\Author[
	affiliation = sjtur,
	email = {ratiu@sjtu.edu.cn, tudor.ratiu@epfl.ch},
	funding = {Partially supported by National Natural Science 
	Foundation of China grant 12250710125 and NCCR SwissMAP grant of
	the Swiss National Science Foundation.}
	]{ratiu}{Tudor S. Ratiu}

\begin{document}
	
\MakeTitle

\tableofcontents

\listoftodos

\section{Introduction}

The interplay between Hamiltonian systems with symmetry and complex 
geometry is of paramount importance in symplectic geometry. A 
particularly powerful tool in connecting these areas is the 
norm-squared $\norm{J}^2$ of the momentum map $J$.
The study of $\norm{J}^2$ has seen a wide range of applications. 
For example, Kirwan \parencite{Kirwan1984} used the Morse theory 
of $\norm{J}^2$ to obtain strong results about the cohomology of 
the symplectic quotient, known as Kirwan surjectivity. Witten 
\parencite{Witten1992} demonstrated that certain integrals localize 
to a small neighborhood of the critical set of $\norm{J}^2$ (see 
\parencite{Paradan1999,Paradan2000,Woodward2005,HaradaKarshon2012} for 
similar localization results). Additionally, the Kempf-Ness theorem 
\parencite{KempfNess1979}, which describes the equivalence between 
notions of quotient in symplectic and algebraic geometry, uses 
$\norm{J}^2$ in an essential way.

While being landmarks in their own right, these rigorous results 
about finite-dimensional systems expand their full strength as a 
conceptual framework for the study of geometric partial differential 
equations. Many geometric PDEs can be formulated as a zero 
level-set constraint of a momentum map associated with an 
infinite-dimensional Lie group acting on an infinite-dimensional 
Hamiltonian system. When this is the case, the finite-dimensional 
techniques surrounding $\norm{J}^2$ serve as a blueprint to come 
up with fundamental conjectures about obstructions and stability 
of solutions to the original PDE. Examples include the work of 
Atiyah and Bott \parencite{AtiyahBott1983} on Yang--Mills 
connections on a Riemann surface, the Donaldson-Uhlenbeck-Yau 
correspondence \parencite{Donaldson1985, UhlenbeckYau1986} relating 
stable holomorphic vector bundles and Hermitian Yang--Mills 
connections, the Kobayashi-Hitchin correspondence 
\parencite{Kobayashi1982,Hitchin1979}, and the recent resolution 
of the Yau--Tian--Donaldson conjecture 
\parencite{ChenDonaldsonSun2015,Tian2015}.
All these examples are quite different in nature, but they all 
share the same abstract framework grounded in infinite-dimensional 
symplectic geometry.

In view of the wide success of this \emph{conceptual} picture, it 
is astonishing that no \emph{rigorous} infinite-dimensional framework 
is available yet. In this paper, we initiate the study of the 
norm-squared momentum map as a rigorous tool in infinite dimensions. 
The long term goal is the development of a rigorous theory of the 
Kempf-Ness theorem in infinite dimensions, which encompasses the
above-mentioned examples as specific cases. In particular, we calculate 
the Hessian of the norm-squared momentum map in infinite dimensions
(\cref{prop:normedsquared:hessian,prop:normedsquared:hessianSummary}).
As a direct corollary of an explicit formula for the Hessian, we 
conclude that the Hessian is non-negative definite along the 
infinitesimal complex orbit 
(\cref{prop:normedsquared:hessianPositiveDefAlongComplexOrbit}).
When applied to different PDEs, this provides a unified framework 
explaining various convexity results, such as the well known fact 
that the Calabi energy is locally convex near an extremal K\"ahler 
metric. Extrapolating from the finite dimensional theory, one expects 
that the stabilizer of a critical point of the norm-squared momentum 
map inherits additional structure. We prove that this is indeed the 
case and establish an eigenvalue decomposition (with respect to a 
certain operator derived from the momentum map) of the stabilizer 
of the complex Lie algebra action at a critical point of the 
norm-squared momentum map (\cref{prop:normedsquared:decompositionComplexStab,prop:decomposition:decompositionComplexStab}).
If a critical point is a zero of the momentum map, then this 
decomposition collapses and the stabilizer is reductive.
In this way, we obtain an obstruction for a point to lie in the 
zero set of the momentum map. As explained below, this recovers 
the reductiveness/decomposition of the automorphism algebra at 
an extremal K\"ahler metric due to \textcite{Calabi1982,Matsushima1957}. 
In fact, many similar obstructions are known in a wide range of 
geometric PDEs: Mabuchi solitons \parencite[Theorem~4.1]{Mabuchi2001}, 
K\"ahler--Ricci solitons \parencite[Theorem~A]{TianZhu2000}, coupled 
K\"ahler--Einstein metrics \parencite{Nakamura2023}, \( f \)-extremal 
K\"ahler metrics \parencite{FutakiOno2017,Lahdili2019}, extremal 
Sasakian metrics \parencite[Theorem~11.3.1]{Boyer2008}, Cahen--Gutt 
extremal K\"ahler metrics \parencite[Theorem~4.11]{FutakiOno2018}, and 
solutions of the K\"ahler--Yang--Mills--Higgs equations 
\parencite[Theorem~3.6]{AlvarezConsulGarciaFernandezGarciaPrada2019}.
For all these cases an infinite-dimensional symplectic framework is 
available, it is covered by our general framework, and our results 
provide a unified proof of these obstructions.

Besides providing a unified and rigorous framework, our setup also 
opens up new avenues. In K\"ahler geometry, 
\textcite{Fujiki1992,Donaldson1997} showed that the scalar 
curvature is the momentum map for the action of the group of 
Hamiltonian diffeomorphisms on the space of (almost) complex 
structures compatible with the symplectic form. The norm-squared 
momentum map yields the Calabi energy functional and critical points 
are extremal K\"ahler metrics, with constant scalar curvature and 
K\"ahler-Einstein metrics as important special cases.
From this perspective, it is natural to ask for a generalization 
to the action of the full group of symplectomorphisms.
In \cref{sec:kaehler}, we show that the action of the group of 
symplectomorphisms has a momentum map given by the Chern connection 
on the anti-canonical bundle relative to the Chern connection of a 
fixed compatible almost complex structure. The fixed reference 
complex structure is necessary to ensure that the momentum map 
is well-defined, but it also destroys the equivariance of the 
momentum map. We calculate the associated non-equivariance 
\( 2 \)-cocycle on the Lie algebra of symplectic vector fields 
and show that it vanishes for Calabi--Yau manifolds.
In fact, triviality of the non-equivariance cocycle is only 
slightly weaker than the vanishing of the first Chern class, so 
triviality can be seen as a natural generalization of the 
Calabi--Yau condition. Using an infinite-dimensional version of 
the prequantum construction, we obtain a central extension of 
the group of symplectomorphisms that integrates the non-equivariance 
cocycle. If the non-equivariance cocycle is trivial, then following 
the strategy of \parencite{Shelukhin2014} we show that the universal 
covering of the identity component \( \DiffGroup(M, \omega)_0 \) 
of \( \DiffGroup(M, \omega) \) admits 
a non-trivial quasimorphism and hence has infinite commutator length.
It is natural to conjecture that this quasimorphism is a natural 
generalization of the Entov quasimorphism \parencite{Entov2004} 
constructed under the stronger assumption that the first Chern 
class vanishes. Returning to the Calabi program, the norm-squared 
momentum map for the full group of symplectomorphisms yields a 
natural extension of the Calabi energy functional that contains, 
beside the scalar curvature, a term that measures the deviation 
of the (holonomy of the) Chern connection from the reference 
connection. Critical points thus yield a relative notion of 
extremal K\"ahler metrics and our general results provide a
 Matsushima-type decomposition for such metrics 
(\cref{prop:kaehler:decompositionComplexStabRelativeExtremal}).
Our proof does not rely on any integrability condition, so we 
also obtain a Matsushima/Calabi-type decomposition for extremal 
\emph{almost} K\"ahler metrics 
(\cref{prop:kaehler:decompositionComplexStabAndHol}).

One may consider K\"ahler geometry as a first-order example, in the 
sense that the action of a symplectic vector field on the space of 
almost complex structures compatible with the symplectic form only 
involves the first jet of the vector field. Consequently, upon 
dualizing, the momentum map for the full group of symplectomorphisms 
depends also only on the first jet of the almost complex structure 
(and the momentum map for Hamiltonian diffeomorphisms is second-order).
One can step up the ladder and consider the action of the group of 
symplectomorphisms on the space of symplectic connections, which 
involves the second jet of the symplectic vector field.
In \parencite{CahenGutt2005}, Cahen and Gutt showed that the space  
of symplectic connections is an infinite-dimensional symplectic 
manifold and the group of Hamiltonian diffeomorphisms possesses a 
momentum map (a certain \( 3 \)-rd order operator of the connection).
This momentum map is of importance in deformation quantization where 
it serves as an obstruction for the associated Fedosov star product 
to be closed \parencite{Fuente-Gravy2015}. 
Applied to this setting, our general results yield a decomposition 
of the stabilizer and an expression for the Hessian of a critical 
point of the norm-squared momentum map. This recovers and extends 
the results of 
\parencite{FutakiOno2018,FutakiFuenteGravy2019,Fuente-Gravy2016,Fuente-Gravy2015}.
Again passing from the group of Hamiltonian diffeomorphisms to the 
full group of symplectomorphisms, we calculate the momentum map for 
the full group and show that it is no longer equivariant.
The non-equivariance cocycle is this time related to the first 
Pontryagin class of the underlying manifold and we show that it 
can be integrated to a central extension of the group of 
symplectomorphisms. We work out the Calabi program for the group of 
symplectomorphisms which yields an extension of the notion of 
Cahen--Gutt extremal K\"ahler metrics. From our general results, 
we obtain a Matsushima-type decomposition for such metrics 
(\cref{prop:symplecticConnections:decompositionAut}).
Given the tight relation of the momentum map for Hamiltonian 
diffeomorphisms and the Fedosov star product, it is natural to 
ask for a similar relation between the momentum map for the full 
group of symplectomorphisms and  deformation quantization.

Finite dimensional analogues of our general results are known in the 
literature. For example, X.~Wang \parencite{Wang2004} and L.~Wang 
\parencite{Wang2006} have obtained similar decompositions of 
complex stabilizer algebras using the Hessian of the momentum map 
in the finite-dimensional K\"ahler setting.
One may initially hope that this finite-dimensional analysis can be 
straightforwardly extended to the infinite-dimensional setting.
However, this is not the case, mainly for two reasons.
First, as discussed above, the action of symplectomorphism groups 
often does not have an \emph{equivariant} momentum map. Moreover, 
we are not aware of an \( \AdAction \)-invariant pairing on the Lie 
algebra of symplectic vector fields. A discussion of the Hessian 
of the norm-squared momentum map in such a non-equivariant setting 
is missing in the literature, even in finite dimensions.
Note that it is not possible to circumvent this issue by passing to 
an appropriate central extension (as one does so often with 
non-equivariant momentum maps), as this would only shift the problem 
to a non-equivariant pairing on the Lie algebra, which then generates 
the same kind of difficulties for the norm-squared momentum map.
Second, the infinite-dimensional setting introduces additional technical 
complications. In the finite-dimensional treatment, one assumes 
integrability of the complex structure on the symplectic manifold and 
that the action is preserving the complex structure, hence it extends 
to a holomorphic action of the complexified group.
In the infinite-dimensional setting, the complex structure is not 
necessarily integrable (nor is it a priori clear what the right notion 
of integrability should be) and the construction (or non-existence) 
of complexifications of certain diffeomorphism groups is a 
notoriously difficult open problem.
In fact, the infinite-dimensional setting is so different from the 
finite-dimensional one that we must proceed in a completely different
way. The standard approach is to calculate first the Hessian of the 
norm-squared momentum map and then, invoking the Hessian, conclude that 
the Lichnerowicz--Calabi operators commute. Using the commutativity of 
these operators, one then proceeds to investigate the structure of 
the stabilizer.
In the infinite-dimensional setting, we proceed in the opposite 
direction and first establish the commutativity of the 
Lichnerowicz--Calabi operators by direct means, using some weak 
and natural invariance properties of the complex structure as the main 
ingredient. We then use the commutativity of the Lichnerowicz--Calabi 
operators to calculate the Hessian of the norm-squared momentum map
and to obtain the decomposition of the stabilizer.
Since the commutativity of the Lichnerowicz--Calabi operators is no 
longer coupled to the Hessian of the norm-squared momentum map, 
we also obtain a root-space like decomposition of the complex stabilizer 
with respect to an Abelian subalgebra of the real stabilizer; see 
\cref{prop:decomposition:decompositionComplexStab}.
Our invariance assumptions are so weak that the infinitesimally 
complexified action is not a Lie algebra action and its stabilizer 
is not a Lie algebra. So the existence of such a decomposition 
is quite surprising. In fact, already in the finite-dimensional 
example of the Galilean group acting on one of its coadjoint orbits, 
the complex stabilizer is not a Lie algebra; see 
\cref{ex:momentumMapSquared:galileanGroup}.

\paragraph*{Organization of the paper}
In \cref{sec:contractible}, we calculate the momentum map for 
a symplectic action on a contractible manifold (on abstract grounds, 
such an action must have a momentum map, but we are not aware of 
a reference for its explicit calculation). In general, the momentum 
map depends on a choice of a reference point and this choice renders 
the momentum map non-equivariant.
We use the prequantum bundle construction, specialized to the case 
of a contractible manifold, to obtain a central group extension 
that integrates the non-equivariance cocycle. These results are 
used in subsequent applications to calculate momentum maps 
and non-equivariance cocycles.
In \cref{sec:stabilizer_algebra_of_the complexified_action}, we 
introduce the Lichnerowicz--Calabi operators in the general 
setting of infinite-dimensional symplectic manifolds and 
investigate their properties.
The main conclusion is the general decomposition 
\cref{prop:decomposition:decompositionComplexStab}.
Using these results, in \cref{sec:momentumMapSquared}, we 
calculate the Hessian of the norm-squared momentum map in 
terms of the Lichnerowicz--Calabi operators, see 
\cref{prop:normedsquared:hessian,prop:normedsquared:hessianSummary}, 
and obtain a decomposition of the stabilizer, see 
\cref{prop:normedsquared:decompositionComplexStab}.
Then we apply these results to K\"ahler geometry in 
\cref{sec:kaehler}, to symplectic connections in 
\cref{sec:symplecticConnections}, and to Yang--Mills 
connections in \cref{sec_yang_mills}. 
The appendix contains a summary of notations and conventions, 
especially concerning the Penrose abstract index notation, 
used extensively in \cref{sec:symplecticConnections}.

\paragraph*{Acknowledgments} We are grateful to Simone Gutt
for advice and the challenge to start this project in order to
better understand the Cahen--Gutt momentum map. We   
thank Akito Futaki, Barbara Tumpach, Cornelia Vizman, and 
Fran\c{c}ois Ziegler for fruitful discussions.

\section{Momentum maps on contractible symplectic manifolds}
\label{sec:contractible}

In this section, we establish general results concerning symplectic group 
actions on manifolds that are contractible. As an important special case, we 
study affine actions on symplectic affine spaces. The momentum 
map for affine actions is similar to the quadratic momentum map for linear 
actions, with the important difference that there is an affine term that 
breaks equivariance.

\subsection{The momentum map and its non-equivariance cocycle}
\label{sec:nonequivariant_cocycle}

Let \( (M, \omega) \) be a symplectic manifold.
In the following, we assume that there exists a smooth contraction of \( M \);
that is, a smooth map \( \Lambda: M \times M \times [0,1] \to M \) such 
that \( \Lambda(m_0, m, 0) = m_0 \) and \( \Lambda(m_0, m, 1) = m \) 
for all $m_0, m \in  M$.
On abstract grounds, every symplectic action on a contractible manifold 
possesses a momentum map. The following gives an explicit construction of 
this momentum map under a natural equivariance assumption.
\begin{prop}
\label{prop:contractible:momentumMap}
Let \( (M, \omega) \) be a symplectic manifold, \( G \) a Lie group 
acting symplectically on \( M \), and \( \kappa: \LieA{g}^* \times 
\LieA{g} \to \R \) a non-degenerate pairing\footnotemark{}.
\footnotetext{Here, \( \LieA{g}^* \) is an abstract vector space whose 
role as the dual is embodied only through the pairing \( \kappa \). However, 
intuitively, we think of \( \LieA{g}^* \) as \textquote{the} dual of 
\( \LieA{g} \) even though it is not necessarily the functional analytic 
dual of \( \LieA{g} \).}
Assume that there exists a smooth contraction \( \Lambda: M \times M \times 
[0,1] \to M \) of \( M \) which is equivariant in the sense that 
\( \Lambda(g \cdot m_0, g \cdot m, t) = g \cdot \Lambda(m_0, m, t) \) for 
all \(g \in  G\). For every \( m_0 \in M \), a momentum map 
\( J: M \to \LieA{g}^* \) for the \( G \)-action on \( M \) is given by
	\begin{equation}
\label{eq:contractible:momentumMap}
\kappa\bigl(J(m), \xi\bigr) = \int_0^1 \Bigl( \bigl(\Lambda_{m_0}^* 
\omega\bigr)_{(m, t)} (\difp_t, \xi \ldot m) + \bigl(\bar{\Lambda}_m^* 
\omega\bigr)_{(m_0, t)} (\difp_t, \xi \ldot m_0) \Bigr) \dif t,
\end{equation}
where \( \Lambda_{m_0} = \Lambda(m_0, \cdot, \cdot) \) and 
\( \bar{\Lambda}_m = \Lambda(\cdot, m, \cdot) \).
In infinite dimensions, we need to additionally assume that the 
linear functional on \( \LieA{g} \) defined by the right-hand side can 
indeed be represented by an element of \( \LieA{g}^* \) with respect to 
the pairing \( \kappa \). Moreover, if \( \Lambda(m_0, m_0, t) = m_0 \), 
then \( J \) is the unique momentum map satisfying \( J(m_0) = 0 \).
\end{prop}

Roughly speaking, the \( 1 \)-form \( \int_0^1 \bigl(\Lambda_{m_0}^* 
\omega\bigr)_{(\cdot, t)} (\difp_t, \cdot) \dif t \) on \( M \) occurring 
in the first summand is a primitive of \( \omega \) and the second summand 
accounts for the fact that this primitive is not \( G \)-invariant, in general.
\begin{proof}
For every \( \xi \in \LieA{g} \), define the function
\begin{equation}
\label{equ:def_J_xi}
J_\xi = - \int_0^1 \difp_t \contr 
\bigl(\Lambda^*_{m_0} (\xi^* \contr \omega)\bigr) \dif t,
\end{equation}
where \( \xi^* \) is the fundamental vector field on \( M \) induced by 
the action of \( \xi \). Similar to the proof of the Poincar\'e lemma, 
we find
\begin{equation}
\begin{split}
\dif J_\xi&= - \bigl(\Lambda^*_{m_0} (\xi^* \contr \omega)\bigr)\big|_{t=1} 
+ \bigl(\Lambda^*_{m_0} (\xi^* \contr \omega)\bigr)\big|_{t=0} 
+ \int_0^1 \difp_t \contr \bigl(\Lambda^*_{m_0} 
\dif (\xi^* \contr \omega)\bigr) \dif t  \\
&= - \, \xi^* \contr \omega,
\end{split}
\end{equation}
where the first equality follows 
from~\parencite[Prop.~IV.2.IX, page 157]{GreubHalperinEtAl1972} and 
the second equality follows from \( \Lambda_{m_0}(\cdot, 0) = m_0 \), 
\( \Lambda_{m_0}(\cdot, 1) = \id_M \), and 
$\dif (\xi^* \contr \omega)=0$ since the action is symplectic.	
Thus, \( J: M \to \LieA{g}^* \) defined by \( \kappa\bigl(J(m), \xi\bigr) 
= J_\xi(m) \) is a momentum map.

Since \( \Lambda \) is equivariant, we find
\begin{equation}
\xi \ldot \Lambda(m_0, m, t) =
\tangent_{(m_0, t)} \bar{\Lambda}_m (\xi \ldot m_0) + 
\tangent_{(m, t)} \Lambda_{m_0} (\xi \ldot m)	
\end{equation}
for any $\xi \in \LieA{g}$. This identity implies
\begin{equation}
\begin{split}
\bigl(\Lambda^*_{m_0} (\xi^* \contr \omega)\bigr)_{m, t}(\difp_t)
&= \omega^{}_{\Lambda(m_0, m, t)} \bigl(\xi \ldot \Lambda(m_0, m, t), 
\tangent_{(m_0, m, t)} \Lambda (\difp_t)\bigr)   
                \\
&= \omega^{}_{\Lambda(m_0, m, t)} \bigl(
\tangent_{(m_0, t)} \bar{\Lambda}_m (\xi \ldot m_0), 
\tangent_{(m_0, t)} \bar{\Lambda}_m (\difp_t)\bigr)
				\\
&\qquad+ \omega^{}_{\Lambda(m_0, m, t)} \bigl(
\tangent_{(m, t)} \Lambda_{m_0} (\xi \ldot m), 
\tangent_{(m, t)} \Lambda_{m_0} (\difp_t)\bigr)
			\\
&= (\bar{\Lambda}_{m}^* \omega)_{(m_0, t)} \bigl(
\xi \ldot m_0, \difp_t\bigr) + 
(\Lambda_{m_0}^* \omega)_{(m, t)} \bigl(\xi \ldot m, \difp_t\bigr).
\end{split}
\end{equation}
Using this identity in the defining equation~\eqref{equ:def_J_xi} of 
\( J_\xi \) yields~\eqref{eq:contractible:momentumMap}.
	
Finally, since \( M \) is connected, the momentum map is uniquely defined 
up to an additive constant. If \( \Lambda(m_0, m_0, t) = m_0 \), then 
\( \tangent_{(m_0, m_0, t)} \Lambda (\difp_t) = 0 \) and 
so~\eqref{eq:contractible:momentumMap} implies \( J(m_0) = 0 \).
\end{proof}

The momentum map \( J \) defined in~\eqref{eq:contractible:momentumMap} 
does not need to be equivariant.
Recall that the \emphDef{non-equivariance one-cocycle} 
\( \sigma: G \to \LieA{g}^* \) associated with \( J \) is defined by
\begin{equation}
\label{group_one_cocycle}
\sigma(g) = J(g \cdot m) - \CoAdAction_{g^{-1}} J(m),
\end{equation}
where \( \CoAdAction \) denotes the coadjoint action, \ie, 
\(\sigma(gh)= \sigma (g) + \operatorname{Ad}^\ast_{g ^{-1}} \sigma(h)\)
for all \(g,h \in G\) and \(\sigma(e) = 0\).
Since \( M \) is connected, the cocycle \( \sigma \) is independent of 
\( m \in M \), see \parencite[Proposition~4.5.21]{OrtegaRatiu2003}.
The associated \textit{non-equivariance \( 2 \)-cocycle} 
\( \Sigma: \LieA{g} \times \LieA{g} \to \R \) defined by
\begin{equation}\begin{split}
\label{eq:normedsquared:nonequiv_first}
\Sigma(\xi, \eta) &\defeq
\kappa\left(\tangent_e \sigma (\xi), \eta \right) =
\kappa(J(m), \commutator{\xi}{\eta}) +
\omega_{m}(\xi \ldot m, \eta \ldot m) \\
&=J_{\commutator{\xi}{ \eta}}(m) + \poisson{J_\xi}{J_\eta}(m), 
\qquad \xi, \eta \in \LieA{g}, \quad m \in M,
\end{split}\end{equation}
where \(J_\xi \defeq \kappa(J(\cdot), \xi)\) for any \(\xi \in \LieA{g}\),
is also independent of \(m \in  M\) if \(M\) is connected;
the second equality follows from the definition of the momentum map,
namely \(\xi^\ast  = X_{J_\xi}\) for all \(\xi \in \LieA{g}\). Recall that
\( \Sigma \) is bilinear, skew-symmetric, and satisfies the \(2\)-cocycle 
identity \( \Sigma([\xi , \eta], \zeta) + \Sigma([\eta , \zeta], \xi) +
\Sigma([\zeta , \xi], \eta) =0 \) for all \(\xi, \eta, \zeta\in \mathfrak{g}\).
\medskip

Returning to our case, assume that 
\( \Lambda(m_0, m_0, t) = m_0 \) for all \(t \in [0,1]\).
\Cref{prop:contractible:momentumMap} guarantees 
that \( J(m_0) = 0 \) and we get from~\eqref{group_one_cocycle}
\begin{equation}
	\label{eq:contractible:oneCocycle}
	\sigma(g) = J(g \cdot m_0).
\end{equation}
Thus, the non-equivariance of \( J \) is a consequence of the fact that 
\( m_0 \) does not need to be fixed by the \( G \)-action. 
Thus~\eqref{eq:normedsquared:nonequiv_first} gives the corresponding infinitesimal 
non-equivariance two-cocycle \( \Sigma: \LieA{g} \times \LieA{g} \to \R \): 
\begin{equation}
	\label{eq:contractible:twoCocycle}
	\Sigma(\xi, \eta) 
		\defeq \kappa\bigl(\tangent_e \sigma (\xi), \eta\bigr)
		= \omega_{m}(\xi \ldot m, \eta \ldot m).
\end{equation}
We will now use the prequantum bundle construction to integrate the 
\( 2 \)-cocycle \( \Sigma \) to a central Lie group extension of \( G \).
For this purpose, recall the following geometric construction of Lie group 
extensions. Let \( (M, \omega) \) be a connected symplectic manifold and 
let \( P \to M \) be a \( \UGroup(1) \)-prequantum bundle with connection 
\( \vartheta \). The smooth identity component \( \AutGroup(P, \vartheta)_0 \) 
of the group \( \AutGroup(P, \vartheta) \) of connection-preserving 
automorphisms of \( P \) is a central \( \UGroup(1) \)-extension of the 
group \( \HamDiffGroup(M, \omega) \) of Hamiltonian diffeomorphisms of \( M \).
Here, \( \HamDiffGroup(M, \omega) \) consists of those symplectomorphisms that 
are endpoints of smooth curves in the kernel of the flux homomorphism
\begin{equation}
\DiffGroup(M, \omega)_0 \to \clDiffFormSpace^1(M) \slash 
\clZDiffFormSpace^1(M), \qquad \phi \mapsto 
\equivClass*{\int_0^1 (\difLog_t \phi_t) \contr \omega \, \dif t},
\end{equation}
where $\DiffGroup(M, \omega)_0$ denotes the smooth identity
component of $\DiffGroup(M, \omega)$, \( \phi_t \) is a smooth curve 
in \( \DiffGroup(M, \omega) \) from \( \id_M \) to \( \phi \), 
\( \difLog_t \phi_t \in \VectorFieldSpace(M) \) is its left logarithmic 
derivative, and \( \clZDiffFormSpace^1(M) \) denotes the space of 
closed \( 1 \)-forms on \( M \) with periods contained in 
\( \Z \subseteq \R \); see  
\parencites{NeebVizman2003}[Section~3]{DiezJanssensNeebVizmannHolPres} 
for details.

Given a Hamiltonian \( G \)-action, the pull-back of this central extension 
along the action \( G \to \HamDiffGroup(M, \omega) \) yields a central 
\( \UGroup(1) \)-extension \( \hat{G} \) of \( G \):
\begin{equationcd}[label={eq:contractible:groupExt}]
	\UGroup(1) 
		\to[r]
		& \AutGroup(P, \vartheta)_0
		\to[r]
		& \HamDiffGroup(M, \omega)
	\\
	\UGroup(1) 
		\to[r]
		\to[u]
		& \hat{G}
		\to[r]
		\to[u]
		& G.
		\to[u]
\end{equationcd}
The Lie algebra \( 2\)-cocycle underlying the infinitesimal central 
extension \( \hat{\LieA{g}} \) of \( \LieA{g} \) is cohomologous to 
the non-equivariance 
\( 2 \)-cocycle of the \( G \)-action on \( M \), see \eg 
\parencite[Remark~3.5]{NeebVizman2003}.
If the manifold \( M \) is infinite-dimensional, then the groups 
\( \AutGroup(P, \vartheta)_0 \) and \( \HamDiffGroup(M, \omega) \) are not 
Lie groups in general. Nonetheless, the pull-back \( \hat{G} \) turns out to 
be a Lie group even in the infinite-dimensional setting, see 
\parencite[Theorem~3.4]{NeebVizman2003}.

Clearly, this construction applies, in particular, to a symplectic group 
action on a contractible manifold.
The prequantum bundle is trivial in this case because the base manifold 
is contractible. This allows to describe the resulting Lie group 
extension \( \hat{G} \) explicitly. Below we write
the group operation in \( \UGroup(1) = \mathbb{R}/\mathbb{Z}\) as 
addition, modulo \(\mathbb{Z}\) being tacitly understood.

\begin{prop}
	\label{prop:contractible:groupExt}
Let \( (M, \omega) \) be a symplectic manifold, \( G \) a Lie group 
acting symplectically on \( M \), and \( \kappa: \LieA{g}^* \times 
\LieA{g} \to \R \) a non-degenerate pairing.
Assume that there exists a smooth contraction \( \Lambda: M 
\times M \times [0,1] \to M \) of \( M \) which is equivariant in 
the sense that \( \Lambda(g \cdot m_0, g \cdot m, t) = 
g \cdot \Lambda(m_0, m, t) \) for all \(g \in  G\) and which 
satisfies \( \Lambda(m_0, m_0, t) = m_0 \) for all \(t \in [0,1]\). 
For every \( m_0 \in M \) and \( g_1, g_2 \in G \), let 
\( \chi_{g_1, g_2}: [0,1] \times [0,1] \to M \) be defined by 
\( \chi_{g_1, g_2}(s, t) = \Lambda_{g_1^{-1} \cdot m_0} 
\bigl(\Lambda_{g_2^{-1} \cdot m_0}(m_0, s), t \bigr) \). Then 
the Lie group \( \hat{G} = G \times \UGroup(1) \) with group 
multiplication
\begin{equation}
	\label{eq:contractible:groupExt:groupMult}
(g_1, z_1) \cdot (g_2, z_2) = \left(g_1 g_2, z_1 + z_2 - 
\int_0^1 \dif s \int_0^1 \dif t \difp_s \contr \difp_t 
\contr \bigl(\chi^*_{g_1g_2, g_2} \omega - 
\chi^*_{g_2, g_2} \omega \bigr) \right)
\end{equation}
is a central Lie group \( \UGroup(1) \)-extension of \( G \) whose 
associated Lie algebra \( 2 \)-cocycle is the non-equivariance 
\( 2 \)-cocycle \( \Sigma \).
\end{prop}

\begin{proof}
If the prequantum bundle \( P \to M \) is trivial, a bundle automorphism 
\( \phi: P \to P \) is necessarily of the form \( \phi(m, z) = 
\bigl(\check{\phi}(m), \tilde{\phi}(m) + z \bigr) \), \(z \in  
\UGroup(1) \), for some 
diffeomorphism \( \check{\phi}: M \to M \) and a smooth map 
\( \tilde{\phi}: M \to \UGroup(1) \).
Moreover, a \( 1 \)-form \( \theta \) on \( M \) with 
\( \dif \theta = \omega \) gives rise to a connection \( 1 \)-form 
\( \theta + \dif \vartheta \) on \( P = M \times \UGroup(1) \), whose 
curvature is \( \omega \). Here \( \dif \vartheta \) is the natural 
form on \( \UGroup(1) \). Clearly, the bundle 
automorphism \( \phi \) 
preserves \( \theta + \dif \vartheta \) if and only if 
\( \check{\phi}^* \theta + \dif \tilde{\phi} = \theta \).
For every \( m_0 \in M \), we thus obtain a section 
\(\HamDiffGroup(M, \omega)\to \AutGroup(P, \vartheta)_0 \)
in the top row
of~\eqref{eq:contractible:groupExt} by assigning to a 
 Hamiltonian diffeomorphism \( \psi: M \to M \) 
the pair \( (\psi, \tilde{\psi}) \) with \( \tilde{\psi} \) being 
the unique solution of \( \theta - \psi^* \theta = \dif \tilde{\psi} \) 
satisfying \( \tilde{\psi}(m_0) = 0 \).

For every \( g \in G \), let \( \Upsilon_g: M \to M \) be its action 
diffeomorphism. Then the above discussion shows that \( \hat{G} \) is 
identified with tuples \( (g, \tilde{\Upsilon}_g) \), where \( g \in G \) 
and \( \tilde{\Upsilon}_g: M \to \UGroup(1) \) satisfies \( \theta - 
\Upsilon_g^* \theta = \dif  \tilde{\Upsilon}_g \).
Since such a map \( \tilde{\Upsilon}_g \) is unique up to addition 
of a constant, we may identify \( \hat{G} \), as a manifold not as a 
Lie group, with \( G \times \UGroup(1) \) by sending 
\( (g, \tilde{\Upsilon}_g) \) to \( \bigl(g, 
\tilde{\Upsilon}_g(m_0)\bigr) \). A section \( s \) of 
\( \hat{G} \to G \) is given by choosing the unique solution 
\( \tilde{\Upsilon}_g \) satisfying \( \tilde{\Upsilon}_g(m_0) = 0 \).
From the general theory in \parencite{NeebVizman2003}, 
we know that every such section of a central extension defines a 
\( \UGroup(1) \)-valued \( 2 \)-cocycle \( c: G \times G \to \UGroup(1) \)
by \( c(g_1, g_2) = s(g_1) \cdot s(g_2) \cdot s(g_1g_2)^{-1} \).
In the present case, we find
\begin{equation}
c(g_1, g_2) = \tilde{\Upsilon}_{g_1} \circ \Upsilon_{g_2} 
+ \tilde{\Upsilon}_{g_2} - \tilde{\Upsilon}_{g_1g_2}.
\end{equation}
A priori, the right-hand side is a \( \UGroup(1) \)-valued map 
on \( M \). But from the general setting we know it has to define 
an element of the center \( \UGroup(1) \), \ie, it has to be a constant map on \(M\).
This can also be verified by a direct computation of its derivative.
In particular, we may evaluate the right-hand side at \( m_0 \) and 
get \( c(g_1, g_2) = \tilde{\Upsilon}_{g_1} \bigl(g_2 \cdot m_0\bigr) \).
To obtain the claimed group multiplication, we use the Poincar\'e lemma 
to construct a globally defined \( 1 \)-form \( \theta \) such that 
\( \dif \theta = \omega \) in terms of the contraction \( \Lambda \), 
and then solve the defining equation for \( \tilde{\Upsilon}_g \).

Thus, recall that the Poincar\'e lemma entails that for every closed 
\( k \)-form \( \beta \) on \( M \) the \( (k-1) \)-form defined by
\begin{equation}
\label{equ_Poincare_lemma}
\alpha = \int_0^1 \difp_t \contr \left(\Lambda^*_{m_0} \beta\right) \dif t
\end{equation}
satisfies \( \dif \alpha = \beta \). 
We apply this to \( \beta = \omega \) and obtain, using equivariance 
of \( \Lambda \), that the primitive \( \theta \) of \( \omega \) 
satisfies
\begin{equation}
\label{equ_upsilon_g_star_theta}
\begin{split}
	\Upsilon_g^* \theta 
&= \Upsilon_g^* \int_0^1 \difp_t \contr \left(
\Lambda^*_{m_0} \omega\right) \dif t
= \int_0^1 \difp_t \contr \left(\bigl(\Upsilon_g \circ 
\Lambda_{g^{-1} \cdot m_0}\bigr)^* \omega\right) \dif t
			\\
&= \int_0^1 \difp_t \contr \left(\Lambda^*_{g^{-1} \cdot m_0} 
\omega\right) \dif t \, .
\end{split}\end{equation}
On the other hand, we can use the Poincar\'e lemma
\eqref{equ_Poincare_lemma} again to calculate the primitive 
\( \tilde{\Upsilon}_g \) of the closed \( 1\)-form 
\( \theta - \Upsilon_g^* \theta \).
When inserting formula \eqref{equ_upsilon_g_star_theta} for 
\( \Upsilon_g^* \theta \), we encounter terms involving the following 
map
\begin{equation}
\Lambda_m \Bigl(\Lambda_{m_0}(g_2 \cdot m_0, s), t\Bigr) 
= g_2 \cdot \Lambda_{g_2^{-1} \cdot m} 
\Bigl(\Lambda_{g_2^{-1} \cdot m_0}(m_0, s), t \Bigr),
\end{equation}
where we have used the equivariance of \( \Lambda \).
Note that for \( m = m_0 \) the right-hand side is equal to 
\( g_2 \cdot \chi_{g_2, g_2}(s, t) \), while for 
\( m = g^{-1}_1 \cdot m_0 \) it is equal to 
\( g_2 \cdot \chi_{g_1g_2, g_2}(s, t) \).
Equipped with this observation, we can now compute
\begin{equation}\label{eq:contractible:groupExt:groupCocycle}\begin{split}
	c(g_1, g_2) 
	&= \tilde{\Upsilon}_{g_1}(g_2 \cdot m_0)
	= \Upsilon_{g_2}^* \tilde{\Upsilon}_{g_1}
	\Big|_{m=m_0}
			\\
	&= \Upsilon_{g_2}^* \int_0^1 \dif s \difp_s \contr 
	\Lambda_{m_0}^* \bigl( \theta - \Upsilon_{g_1}^* \theta\bigr) \, 
	\Big|_{m=m_0}
			\\
	&= \Upsilon_{g_2}^* \int_0^1 \dif s \difp_s \contr 
	\Lambda_{m_0}^*  \int_0^1 \dif t \difp_t \contr 
	\left(\Lambda_{m_0}^* \omega - \Lambda_{g^{-1}_1 \cdot m_0}^*
	 \omega\right) \, \Big|_{m=m_0}
			\\
	&= \int_0^1 \dif s \int_0^1 \dif t \difp_s \contr \difp_t \contr
	 \Bigl(\bigl(\Upsilon_{g_2} \circ \chi_{g_2, g_2}\bigr)^* \omega -
	  \bigl(\Upsilon_{g_2} \circ \chi_{g_1g_2, g_2}\bigr)^* 
	  \omega\Bigr)
			\\
	&= \int_0^1 \dif s \int_0^1 \dif t \difp_s \contr\difp_t \contr
	 \Bigl(\chi^*_{g_2, g_2} \omega - \chi^*_{g_1g_2, g_2} \omega\Bigr).
	\end{split}\end{equation}
This shows that the group multiplication~\eqref{eq:contractible:groupExt:groupMult} is indeed 
the one induced by the prequantum bundle construction.
\end{proof}

\begin{remark}
It is surprisingly difficult to establish the cocycle identity 
for \( c \) directly. In the following, we sketch a proof. Let 
\begin{equation}
	\nu_{g_1, g_2} = \int_0^1 \dif s \int_0^1 \dif t \difp_s 
	\contr \difp_t \contr \chi^*_{g_1g_2, g_2} \omega
	\end{equation}
be the symplectic volume of the triangle with vertices \( m_0 \), 
\( g_2^{-1}g_1^{-1} \cdot m_0 \), and \( g_2^{-1} \cdot m_0 \).
Note that \( c(g_1, g_2) = \nu_{e, g_2} - \nu_{g_1, g_2} \), so 
that the cocycle identity \( c(g_1, g_2) + c(g_1g_2, g_3) = 
c(g_2, g_3) + c(g_1, g_2g_3) \) is equivalent to
\begin{equation}
\label{eq:contractible:groupExt:cocycleIdentity}
\nu_{g_1, g_2} + \nu_{g_1g_2, g_3} + \nu_{e, g_2g_3} = 
\nu_{g_2, g_3} + \nu_{g_1, g_2g_3} + \nu_{e, g_2}.
\end{equation}
A \textquote{visual proof} of this identity is given in \cref{fig:contractible:groupExt:cocycleIdentity}.
\end{remark}

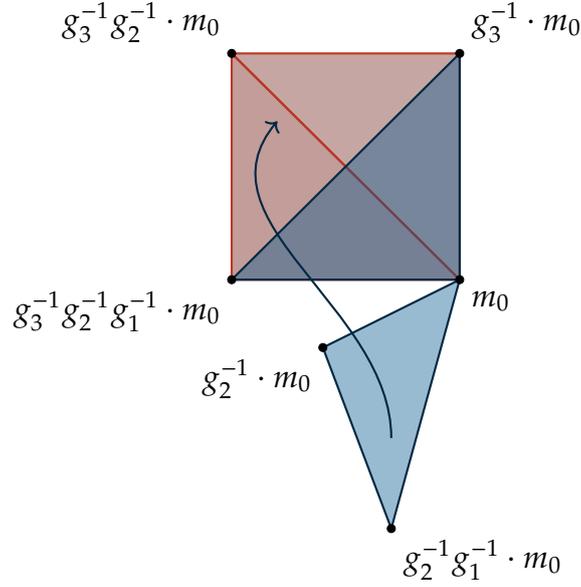
\begin{figure}
	\centering
	\begin{tikzpicture}[scale=3]
		// Triangles
		\filldraw[fill=red3, fill opacity=0.4, draw=red4, thick] (0,0) -- (1,0) -- (0,1) -- (0,0);
		\filldraw[fill=red2, fill opacity=0.4, draw=red4, thick] (1,1) -- (1,0) -- (0,1) -- (1,1);
		\filldraw[fill=blue2, fill opacity=0.4, draw=blue4, thick] (0,0) -- (1,1) -- (1,0) -- (0,0);		
		\filldraw[fill=blue2, fill opacity=0.4, draw=blue4, thick] (1,0) -- (0.7, -1.1) -- (0.4,-0.3) -- (1,0);

		// Points
		\filldraw (0,0) circle (0.5pt) node[anchor=north east] {$g_3^{-1}g_2^{-1}g_1^{-1} \cdot m_0$};
		\filldraw (1,0) circle (0.5pt) node[anchor=north west] {$m_0$};
		\filldraw (0,1) circle (0.5pt) node[anchor=south east] {$g_3^{-1}g_2^{-1} \cdot m_0$};
		\filldraw (1,1) circle (0.5pt) node[anchor=south west] {$g_3^{-1} \cdot m_0$};
		\filldraw (0.4,-0.3) circle (0.5pt) node[anchor=north east] {$g_2^{-1} \cdot m_0$};
		\filldraw (0.7,-1.1) circle (0.5pt) node[anchor=north west] {$g_2^{-1}g_1^{-1} \cdot m_0$};

		// Arrow from lowest triangle to upper left triangle
		\draw[->, thick, blue4] (0.7,-0.7) to[out=90, in=230] (0.2,0.7) ;
	\end{tikzpicture}
	\caption{Visual proof of the cocycle 
	identity~\eqref{eq:contractible:groupExt:cocycleIdentity} in the 
	special case \( \nu(e, g) = 0 \). The red triangles represent 
	\( \nu_{g_2, g_3} \) and \( \nu_{g_1, g_2g_3} \), while the blue 
	triangles represent \( \nu_{g_1, g_2} \) and \( \nu_{g_1g_2, g_3} \). 
	The arrow indicates that upon moving \( \nu_{g_1, g_2} \) by 
	\( g^{-1}_3 \) the blue triangles fill out the same area as the 
	red triangles.}
	\label{fig:contractible:groupExt:cocycleIdentity}
\end{figure}

\begin{remark}
\label{rem:contractible:quasimorhpism}
Assume that \( (M, \omega) \) is a Domic-Toledo space 
\parencite{DomicToledo1987}, \ie the Gromov norm of \( \omega \) 
is bounded. By this we mean that there exists a constant \( C > 0 \) 
such that for all \( m_1, m_2, m_3 \in M \) we have
\begin{equation}
	\abs*{\, \, \, \int_{\triangle(m_1, m_2, m_3)} \omega \,} \leq C,
\end{equation}
where \( \triangle(m_1, m_2, m_3) \) is any disk with boundary given 
by the curves \( m_1 \rightsquigarrow m_2 \), 
\( m_2 \rightsquigarrow m_3 \), \( m_3 \rightsquigarrow m_1 \) 
defined by the contraction \( \Lambda \). Then 
 formula~\eqref{eq:contractible:groupExt:groupCocycle}
implies that the group \( 2 \)-cocycle \( c \) is bounded in the 
sense that there exists a lift of \( c \) to a bounded map \( G \times G \to \R \).

If, in addition, the momentum map is equivariant, then the derivative 
of the group 2-cocycle \( c \) is cohomologous to \( 0 \).
Hence, upon passing to the universal covering \( \tilde{G} \) of 
\( G \), we conclude that \( c \) is the coboundary of a map 
\( \phi: \tilde{G} \to \UGroup(1) \).
The boundedness of \( c \) then implies that \( \phi \) is a 
quasimorphism, \ie, the \( \UGroup(1) \)-valued cocycle 
\( (g,h) \mapsto \phi(gh) - \phi(g) - \phi(h) \) is bounded 
in the sense above. In this way, we recover the construction of 
\textcite{Shelukhin2014} of quasimorphisms from equivariant momentum 
maps\footnotemark{}.
\footnotetext{Note that \textcite{Shelukhin2014} only considers a 
class of paths for which the path from \( m \) to itself is the 
constant path. We do not assume this for \( \Lambda \) and thus get 
the additional integral over \( \chi_{g_2,g_2}([0,1], [0,1]) \), \cf 
\parencite[Equation~(3)]{Shelukhin2014}.}
\end{remark}

\subsection{Momentum maps for affine actions}
\label{sec:affineSymplectic}
An important special case are actions on affine symplectic spaces.

Recall that an \emphDef{affine space modeled on the vector space} \( V \) is a 
set $X$ together with a free transitive right action $X \times V
\ni (x, v) \mapsto x + v \in X$ of the additive Abelian Lie group $(V, +)$ 
underlying the vector space $V$. In particular, given $x, y \in  X$, there 
exists a unique vector, denoted $y-x \in  V$, such that $x+(y-x) = y$.
In the following, we assume that \( V \) is a locally convex space. 
For any $x_0 \in  X$, the map $V\ni v \mapsto x_0+v \in X$ is a bijection.
Put on $X$ the locally convex manifold structure by declaring this bijection 
to be a diffeomorphism. This manifold structure is independent of $x_0$ 
since, for $y_0\in  X$, we have $y_0 + v = x_0 + v + (y_0-x_0)$, \ie, the 
inverse of the diffeomorphisms with base point $x_0$ composed with 
the diffeomorphism with base point $y_0$ is the action of $y_0-x_0 \in V$ 
on $X$. We will shortly refer to \( X \) as a 
\emphDef{locally convex affine space}.
Note, in particular, that \(X\) is connected.
Let $X$ be an affine space modeled on the locally convex space $V$ and $Y$ an 
affine space modeled on the locally convex space $W$. An \emphDef{affine map} 
from $X$ to $Y$ is a pair of maps $f: X \to Y$, $\bar{f}:V \to W$ such that 
$f$ is smooth, $\bar{f}$ is linear and continuous, and $f(x+v)= 
f(x) + \bar{f}(v)$ for all $x \in X$ and $v \in V$. Thus, by simple 
transitivity of the vector space actions, the map $f$ is determined by its 
value on a single point in $X$ and the linear map $\bar{f}$. An affine map 
is an \emphDef{isomorphism} if $\bar{f}$ is a linear isomorphism.
It follows that $f$ is a diffeomorphism.

Let \( (X, \omega) \) be a locally convex affine symplectic space.
That is, the symplectic form \( \omega \) on \( X \) is constant, \ie 
invariant under the $V$-action on $X$, and is 
thereby canonically induced, via any of the standard diffeomorphisms 
associated with a point $x_0 \in X$, by a constant weak symplectic form 
\( \bar{\omega} \) on the linear model space \( V \) of \( X \). 
In particular, \( (V, \bar{\omega}) \) is a weak
symplectic locally convex space.
Here, \textquote{weak symplectic} refers to the fact that the linear 
continuous map $V\ni v \mapsto \omega(v, \cdot ) 
\in V^*$ is injective (but not necessarily an isomorphism of locally convex 
spaces). From now
on, unless otherwise specified, \textquote{symplectic} will always mean
\textquote{weak symplectic}.

A Lie group \(G\) \emphDef{acts affinely} on the affine 
space \(X\) if every \(g \in  G\) acts as an affine isomorphism on \(X\), 
\ie, there is a \(G\)-action on \(X\) and a linear \(G\)-action on \(V\)
such that \(g \cdot (x+v) = g \cdot x + g \cdot v\) for any \(g \in G\) 
and \(v \in  V\).
An affine group action of a Lie group \( G \) on an affine symplectic 
space \( (X, \omega) \) is \emph{symplectic} 
if and only if the associated linear action on \( V \) preserves the 
symplectic form \( \bar{\omega} \).

It is well known that every symplectic 
linear action on a symplectic vector space has a quadratic momentum map.
The following result is the corresponding affine version.

\begin{lemma}
\label{prop:affineSymplectic:momentumMap}
Let \( (X, \omega) \) be an affine symplectic vector space modeled on the 
locally convex space \( V \). Assume that a Lie group \( G \) acts both 
affinely and symplectically on \( X \). Let \( \kappa: \LieA{g}^* 
\times \LieA{g} \to \R \) be a non-degenerate pairing.
For every \( x_0 \in X \), the unique 
momentum map \( J: X \to \LieA{g}^* \) for the \( G \)-action on 
\( X \) vanishing at \( x_0 \) is given by
\begin{equation}
\label{eq:affineSymplectic:momentumMap}
\kappa\bigl(J(x_0 + v), \xi\bigr) = \omega(v, \xi \ldot x_0) + 
\frac{1}{2} \omega(v, \xi \ldot v),
\end{equation}
where \( v \in V \) and \( \xi \in \LieA{g} \). 
In the infinite-dimensional case, one needs to assume that functional 
on \( \LieA{g} \) defined by the right-hand side 
of~\eqref{eq:affineSymplectic:momentumMap} can be represented by 
an element of \( \LieA{g}^* \).
\end{lemma}

The term linear in \( v \) is special to the affine setting: in the case 
when \( X \) is a vector space, one has the natural choice \( x_0 = 0 \) 
which is a fixed point of every linear action and so \( \xi \ldot x_0 =0\).

\begin{proof}
This follows directly from \cref{prop:contractible:momentumMap} using 
the equivariant contraction \( \Lambda(x_0, x, t) = x_0 + t (x-x_0) \).
In fact, we find
\begin{equation}
\bigl(\Lambda^*_{x_0} \omega\bigr)_{(x,t)}(\difp_t, v) = 
t \omega\bigl((x-x_0), v\bigr)
\end{equation}
and
\begin{equation}
\bigl(\bar{\Lambda}^*_{x} \omega\bigr)_{(x_0,t)}(\difp_t, w) = 
(1-t) \omega\bigl((x-x_0), w\bigr).
\end{equation}
Thus, \cref{eq:contractible:momentumMap} simplifies to
\begin{equation}
\kappa\bigl(J(x_0 + v), \xi\bigr) = 
\frac{1}{2}\omega\bigl(v, \xi \ldot (x_0+v)\bigr) 
+ \frac{1}{2} \omega(v, \xi \ldot x_0),
\end{equation}
from which~\eqref{eq:affineSymplectic:momentumMap} follows immediately.
\end{proof}

In contrast to the linear case, the momentum map \( J \) for affine 
actions does not need to be equivariant. As \( J(x_0) = 0 \), the 
\emphDef{non-equivariance one-cocycle} \( \sigma: G \to \LieA{g}^* \) 
associated with \( J \) is 
\begin{equation}
	\label{eq:affineSymplectic:oneCocycle}
	\sigma(g) = J(g \cdot x_0) = J\bigl(x_0 + (g \cdot x_0 - x_0)\bigr).
\end{equation}
Thus, the non-equivariance of \( J \) is a consequence of the fact that 
\( x_0 \) does not need to be a fixed point of the \( G \)-action.
As in the general case~\eqref{eq:contractible:twoCocycle}, the corresponding 
infinitesimal non-equivariance two-cocycle \( \Sigma: \LieA{g} \times 
\LieA{g} \to \R \) is given by
\begin{equation}
	\label{eq:affineSymplectic:twoCocycle}
	\Sigma(\xi, \eta) 
		= \dualPair{\tangent_e \sigma (\xi)}{\eta}
		= \omega(\xi \ldot x_0, \eta \ldot x_0).
\end{equation}

\begin{prop}
	\label{prop:affineSymplectic:groupExt}
The Lie group \( \hat{G} = G \times \UGroup(1) \) with group multiplication
\begin{equation}
\label{prop_group_two_cocycle_affine}
	(g_1, z_1) \cdot (g_2, z_2) = \left(g_1 g_2, z_1 + z_2 - \frac{1}{2} 
	\omega(x_0 - g_1^{-1} \cdot x_0, g_2 \cdot x_0 - x_0)\right)
	\end{equation}
is a central Lie group \( \UGroup(1) \)-extension of \( G \) whose
associated Lie algebra \( 2 \)-cocycle is the non-equivariance 
\( 2 \)-cocycle \( \Sigma \) given by~\eqref{eq:affineSymplectic:twoCocycle}.
\end{prop}
\begin{proof}
This follows directly from \cref{prop:contractible:groupExt} using 
the contraction \( \Lambda(x_0, x, t) = x_0 + t (x-x_0) \).
In fact, then \( \chi_{g_1, g_2}(s, t) = g_1^{-1} \cdot x_0 + 
t g_2^{-1} \cdot x_0 - tg_1^{-1} \cdot x_0 + ts x_0 - ts g_2^{-1}x_0 \) 
and so
\begin{equation}
\left(\chi^*_{g_1g_2, g_2} \omega - \chi^*_{g_2, g_2} \omega\right)
(\difp_t, \difp_s) = t \omega(g_2^{-1} \cdot x_0 - 
g_2^{-1}g_1^{-1} \cdot x_0, x_0 - g_2^{-1}\cdot x_0).
\end{equation}
The integration over \( t \) and \( s \) is now straightforward 
and yields the desired result
\begin{equation}
\label{prop:affineSymplectic:group2cocycle}
c(g_1, g_2) = - \frac{1}{2}\omega(x_0 - g_1^{-1} \cdot x_0, 
g_2 \cdot x_0 - x_0)
\end{equation}
for the group \( 2 \)-cocycle.
\end{proof}

An equally easy proof is to check directly
the cocycle identity for~\eqref{prop_group_two_cocycle_affine}.

A particular important special case is given by affine actions on a 
symplectic vector space. In fact, the general setting considered 
above can be reduced to this special case by choosing a reference 
point. Let \( (V, \omega) \) be a symplectic vector space and let 
\( \rho: G \to \SpGroup(V, \omega) \) be a linear action of \( G \)
 on \( V \) preserving the symplectic form \( \omega \). 
Let \(\rho' \defeq \tangent_e \rho: \LieA{g} \to 
\operatorname{sp}(V, \omega) \) be the induced linear Lie algebra
action. Every 
 affine action of \( G \) on \( V \) 
with linear part \( \rho \) is of the form
\begin{equation}
	g \cdot v = \rho(g) \, v + \tau(g)
\end{equation}
for some \( V \)-valued group one-cocycle \( \tau: G \to V \), \ie, 
\( \tau(g_1g_2) = \tau(g_1) + \rho(g_1) \tau(g_2) \), hence 
\( \tau(e) = 0 \) 
and \( \tau(g^{-1}) = - \rho(g^{-1}) \tau(g) \). Choosing \( x_0 = 0 \) 
as reference point and using these identities, the group 
two-cocycle~\eqref{prop:affineSymplectic:group2cocycle} takes the form
\begin{equation}
\label{eq:affineSymplectic:linear:group2Cocycle}
\begin{split}
c(g_1, g_2) 
	&= \frac{1}{2}\omega\bigl(-g_1^{-1} \cdot 0, g_2 \cdot 0\bigr) \\
	&= \frac{1}{2}\omega\bigl(- \tau(g_1^{-1}), \tau(g_2)\bigr) \\
	&= \frac{1}{2}\omega\bigl(- \tau(g_1^{-1}), - \rho(g_2) \tau(g_2^{-1})
	   \bigr) \\
	&= \frac{1}{2}\omega\bigl(\tau(g_2^{-1}g_1^{-1}), \tau(g_2^{-1})\bigr)
	= \frac{1}{2}\omega\bigl(\tau(g_1), \tau(g_1 g_2)\bigr).
\end{split}
\end{equation}
A straightforward calculation shows that the associated Lie algebra 
cocycle~\eqref{eq:affineSymplectic:twoCocycle} is given by
\begin{equation}
	\label{eq:affineSymplectic:linear:twoCocycle}
	\Sigma = \omega \circ \tau',
\end{equation}
where \( \tau' \defeq \tangent_e \tau: \LieA{g} \to V \).
We note that the cocycles \( c \) and \( \Sigma \) depend only on the 
symplectic form \( \omega \) and on the cocycle \( \tau \), but not 
(directly) on the representation \( \rho \). In contrast, the 
non-equivariance one-cocycle \( \sigma = J \circ \tau: G \to 
\LieA{g}^* \) defined 
in~\eqref{eq:affineSymplectic:oneCocycle} does depend on \( \rho \) since, 
according to~\eqref{eq:affineSymplectic:momentumMap}, the momentum map is 
given by
\begin{equation}
\label{eq:affineSymplectic:linear:momentumMap}
\kappa\bigl(J(v), \xi\bigr) 
= \omega\bigl(v, \tau'(\xi)\bigr) + 
\frac{1}{2} \omega\bigl(v, \rho'(\xi)\,v\bigr).
\end{equation}

In summary, \textit{this symplectic framework yields a construction of the 
cocycles \( c, \Sigma \), and \( \sigma \) using the symplectic form 
\( \omega \), the symplectic representation \( \rho \), and the 
one-cocycle \( \tau \) as ingredients.}

Among others, we recover the following two fundamental group extensions 
as special cases of \cref{prop:affineSymplectic:groupExt}. 
\begin{example}[Heisenberg group]
	\label{ex:affineSymplectic:heisenbergGroup}
Let \( (V, \omega) \) be a symplectic vector space. The natural action 
of \( V \) on itself by translation is an affine symplectic action. 
In this case, the Lie group extension constructed in 
\cref{prop:affineSymplectic:groupExt} coincides with the Heisenberg 
group of \( (V, \omega) \), where the latter is 
considered as an extension by \( \UGroup(1) \).
\end{example}

\begin{example}[Galilean group]
	\label{ex:affineSymplectic:galileanGroup}
Let \( \GalileanGroup = (\SOGroup(3) \lSemiProduct \R^3)\lSemiProduct \R^4 
\) denote the group of Galilean transformations, \ie, the semi-direct 
product of the group \( \SOGroup(3) \) of rotations, the Abelian group 
\( \R^3 \) of boosts and the Abelian group \( \R^4 \) of spacetime 
translations.
We will write elements of \( \GalileanGroup \) as 4-tuples 
\( (R, \vec{v}, \vec{a}, \tau) \) where \( R \in \SOGroup(3) \), 
\( \vec{v} \in \R^3 \), \( \vec{a} \in \R^3 \) and \( \tau \in \R \).
We write elements of the Lie algebra \(\mathfrak{gal}
= (\mathfrak{so}(3)\lSemiProduct \R^3)\lSemiProduct \R^4\)
of  \( \GalileanGroup \) as 4-tuples \((\vec{\alpha}, \vec{\beta},
\vec{\gamma}, \delta)\). The Lie bracket is given by
\begin{equation}
\label{Lie_bracket_on_gal}
\begin{aligned}
&\left[ (\vec{\alpha}_1, \vec{\beta}_1,\vec{\gamma}_1, \delta_1), 
(\vec{\alpha}_2, \vec{\beta}_2,\vec{\gamma}_2, \delta_2)\right] 
\\ 
&\quad = \left(\vec{\alpha}_1\times\vec{\alpha}_2 ,\,
\vec{\alpha}_1\times\vec{\beta}_2-
\vec{\alpha}_2\times\vec{\beta}_1 ,\,
\vec{\alpha}_1\times\vec{\gamma}_2 -
\vec{\alpha}_2\times\vec{\gamma}_1 - 
\delta_1\vec{\beta}_2 + \delta_2\vec{\beta}_1,\,0
\right) .
\end{aligned}
\end{equation}

Consider a non-relativistic particle with mass \( m \neq 0 \) and 
spin \( s > 0 \) moving in three-dimensional Euclidean space 
\( \R^3 \).	This corresponds to the action of \( \GalileanGroup \) 
on \( \R^3 \times \R^3 \times S^2 \) given by 
\begin{equation}
(R, \vec{v}, \vec{a}, \tau) \cdot \bigl(\vec{q}, \vec{p}, \vec{x}\bigr) 
= \left(R \left(\vec{q} - \frac{\tau}{m} \vec{p}\right) - 
\vec{v} \tau + \vec{a}, R\vec{p} + m \vec{v}, R \vec{x}\right).
\end{equation}
The action is symplectic with respect to the symplectic form 
\( \omega = \dif \vec{p} \wedge \dif \vec{q} + s \vol_{S^2} \). 
The action on \( S^2 \) factors through the 
standard action of \( \SOGroup(3) \) on \( S^2 \) by rotations, which 
is symplectic with equivariant momentum map \( \vec{x} \mapsto 
-\frac{s}{2} \vec{x} \) relative to the pairing 
\( (\rho, \alpha) \mapsto 2 \rho \cdot \alpha \) on 
\( \SOAlgebra(3) \isomorph \R^3 \).
Moreover, the action on the first factor is clearly affine, with 
linear action and \( \R^6 \)-valued cocycle given by
\begin{equation}\begin{aligned}
(R, \vec{v}, \vec{a}, \tau) \cdot (\difp_{\vec{q}}, \difp_{\vec{p}}) 
&= \left(R \left(\difp_{\vec{q}} - 
\frac{\tau}{m} \difp_{\vec{p}}\right), R\difp_{\vec{p}}\right),
		\\
(R, \vec{v}, \vec{a}, \tau) &\mapsto \bigl(-\vec{v} \tau 
+ \vec{a}, m\vec{v}\bigr).
\end{aligned}\end{equation} 
Thus, using~\eqref{eq:affineSymplectic:linear:momentumMap} or by 
direct calculation, the momentum map \( J \) for the action on 
\( \R^3 \times \R^3 \times S^2 \) is given by
\begin{equation}
\label{eq:affineSymplectic:galileanGroup:momentumMap}
J(\vec{q}, \vec{p}, \vec{x}) = 
\left(\frac{1}{2} \vec{q} \times \vec{p} - 
\frac{s}{2} \vec{x}, -m\vec{q}, \vec{p}, 
-\frac{1}{2m} \norm{\vec{p}}^2\right),
\end{equation}
relative to the pairing 
\begin{equation}
\label{eq:affineSymplectic:galileanGroup:pairing}
\kappa\bigl((\alpha_1, \beta_1, \gamma_1, \delta_1), 
(\alpha_2, \beta_2, \gamma_2, \delta_2)\bigr) = 
2 \alpha_1 \cdot \alpha_2 + \beta_1 \cdot \beta_2 + 
\gamma_1 \cdot \gamma_2 + \delta_1 \delta_2.
\end{equation}
Since the part of the momentum map corresponding to the action 
on \( S^2 \) is equivariant, the non-equivariance group two-cocycle 
\( c \) can be calculated, 
using~\eqref{eq:affineSymplectic:linear:group2Cocycle}, to be
\begin{equation}
c\bigl((R_1, \vec{v}_1, \vec{a}_1, \tau_1), 
(R_2, \vec{v}_2, \vec{a}_2, \tau_2)\bigr) 
= \frac{m}{2} \Bigl( \vec{v}_1 \cdot 
(R_1 \vec{a}_2) - \vec{a}_1 \cdot (R_1 \vec{v}_2) 
- \tau_2 \, \vec{v}_1 \cdot (R_1 \vec{v}_2) \Bigr).
\end{equation}
This is the Bargmann cocycle \parencite[Equation~(6.28)]{Bargmann1954}, 
and hence the central \( \UGroup(1) \)-extension of \( \GalileanGroup \) 
constructed in \cref{prop:affineSymplectic:groupExt} is the 
Bargmann group.
\end{example}

\begin{example}[Virasoro group]
	\label{ex:affineSymplectic:virasoroGroup}
Let \( V \) denote the vector space of smooth functions \( f: S^1 \to \R \) 
(\ie, smooth functions $\R \to \R$ of period 1) modulo 
constants. The skew-symmetric form on \( V \) defined by
	\begin{equation}
		\omega(\equivClass{f}, \equivClass{g}) = \int_{S^1} f \dif g
	\end{equation}
	is non-degenerate, and thus defines a linear symplectic structure on \( V \).
	Consider the affine symplectic action
	\begin{equation}
		\phi \cdot \equivClass*{f} = \equivClass*{f \circ \phi^{-1} 
		+ \log  \bigl( (\phi^{-1})'\bigr) }
	\end{equation}
on \( V \) of the group \( \DiffGroup_+(S^1) \) of orientation-preserving 
diffeomorphisms of the circle. Here, \( \phi' \in \sFunctionSpace(S^1, \R) \) 
denotes the strictly positive function uniquely determined by 
\( \phi^* \dif \varphi = \phi' \dif \varphi \) for the natural line element 
\( \dif \varphi \) on \( S^1 \). Thus, in the 
notation above, the representation and one-cocycle are given by 
\( \rho(\phi) \equivClass*{f} = \equivClass*{f \circ \phi^{-1}} \) and 
\( \tau(\phi) = \equivClass*{\log\left((\phi^{-1})'\right)} \), respectively. 
According to~\eqref{eq:affineSymplectic:linear:group2Cocycle}, the group 
\( 2 \)-cocycle on \( \DiffGroup_+(S^1) \) is given by
	\begin{equation}
	c(\phi_1, \phi_2) 
	= \frac{1}{2}\omega\bigl(\tau(\phi_2^{-1}\phi_1^{-1}), \tau(\phi_2^{-1})\bigr)
	= \frac{1}{2}\int_{S^1} \log\left((\phi_1\phi_2)'\right) 
	\dif \left(\log\left(\phi_2'\right)\right),
	\end{equation}
which is the Bott--Thurston cocycle, see \parencite[Equation~2]{Bott1977} 
and \eg \parencite[Definition~4.5.1]{GuieuRoger2007}. Thus, the central 
\( \UGroup(1) \)-extension of \( \DiffGroup_+(S^1) \) constructed in 
\cref{prop:affineSymplectic:groupExt} yields the Virasoro group.
Since \( \tau'(X \difp_\varphi) = - \equivClass{X'} \) for \( X \in 
\sFunctionSpace(S^1) \) with \( X' \defeq \dif X (\difp_\varphi) \), the 
associated Lie algebra two-cocycle~\eqref{eq:affineSymplectic:linear:twoCocycle} 
is given by
	\begin{equation}
		\Sigma(X \difp_\varphi, Y \difp_\varphi)
			= \omega\bigl(\equivClass{X'}, \equivClass{Y'}\bigr)
			= \int_{S^1} X' \dif Y'
			= - \int_{S^1} X Y''' \dif \varphi ,
	\end{equation}
which is the Gelfand--Fuchs cocycle (in an appropriate normalization); 
see \eg \parencite[Equation~4.9]{GuieuRoger2007}. Finally, the momentum map 
is given by
	\begin{equation}
		\label{eq:affineSymplectic:virasoroGroup:momentumMap}
		\SectionMapAbb{J}\bigl(\equivClass{f}\bigr) = 
		\Bigl(-f'' + \frac{1}{2} (f')^2\Bigr) \dif \varphi^2
	\end{equation}
where we have identified the (regular) dual of \( \VectorFieldSpace(S^1) \) 
with the space of quadratic differential forms using the pairing
	\begin{equation}
		\kappa(\alpha \dif \varphi^2, X \difp_\varphi) 
		= \int_{S^1} \alpha X \dif \varphi, \qquad 
		\alpha, X \in \sFunctionSpace(S^1).
	\end{equation}
Indeed, using~\eqref{eq:affineSymplectic:linear:momentumMap}, we find by partial 
integration
\begin{equation}\begin{split}
\kappa\Bigl(\SectionMapAbb{J}\bigl(\equivClass{f}\bigr), X \difp_\varphi\Bigr)
&= \omega\bigl(\equivClass{f}, -\equivClass{X'}\bigr) + \frac{1}{2} 
\omega\bigl(\equivClass{f}, -\equivClass{X f'}\bigr) \\
&= - \int_{S^1} f X'' \dif \varphi - 
\frac{1}{2} \int_{S^1} f (X' f' + X f'') \dif \varphi \\
&= - \int_{S^1} f'' \, X \dif \varphi + 
\frac{1}{2} \int_{S^1} (f')^2 \, X \dif \varphi .
\end{split}\end{equation}
Thus, the non-equivariance one-cocycle \( \sigma = \SectionMapAbb{J} \circ \tau: 
\DiffGroup_+(S^1) \to \VectorFieldSpace(S^1)^* \) satisfies
	\begin{equation}
		\sigma\bigl(\phi^{-1}\bigr) = - \left(\frac{\phi'''}{\phi'} - 
		\frac{3}{2} \biggl(\frac{\phi''}{\phi'}\biggr)^2\right) \dif \varphi^2,
	\end{equation}
where the expression between the brackets is the Schwarzian derivative 
of \( \phi \). 

In summary, \textit{the Bott--Thurston cocycle, the Gelfand--Fuchs 
cocycle, the Schwarzian derivative, and the Virasoro group are directly 
and intrinsically derived from the affine symplectic action of 
\( \DiffGroup_+(S^1) \) on the space of smooth functions.}
\end{example}

\section{Structure of the stabilizer algebra of the complexified action}
\label{sec:stabilizer_algebra_of_the complexified_action}

The goal of this section is to derive a structure theorem about the 
stabilizer Lie algebra of the complexified action.  
Similar results in the finite-dimensional K\"ahler 
setting have been obtained in
\parencite{GeorgoulasRobbinEtAl2018,Wang2004,Wang2006} 
at critical points of the norm-squared momentum map.
However, it is quite remarkable that such a decomposition 
theorem can be obtained at an arbitrary point under a rather 
mild compatibility assumptions of the symplectic action 
with the complex structure.

Throughout this section, \( (M, \omega) \) is a connected (weak)
symplectic Fr\'echet manifold endowed with a symplectic action of a 
Fr\'echet Lie group \( G \).
We assume that the action has a Lie algebra-valued momentum map 
\( J: M \to \LieA{g} \) relative to a non-degenerate, symmetric, 
not necessarily \( \AdAction_G \)-invariant, pairing 
\( \kappa: \LieA{g} \times \LieA{g} \to \R \); that is, \( J \) 
satisfies
\begin{equation}
\label{tangent_momentum}
\omega_m (\xi \ldot m, X) + 
\kappa\bigl(\tangent_m J (X), \xi \bigr) = 0
\end{equation}
for all \( m \in M \), \( X \in \TBundle_m M \), and 
\( \xi \in \LieA{g} \).
Recall the notation $J_\xi \defeq \kappa(J(\cdot), \xi)$ 
for any $\xi \in \LieA{g}$ and hence the definition of the 
momentum map is equivalent to $\xi^\ast  = X_{J_\xi}$.
We do not assume \( J \) to be equivariant with respect to 
the coadjoint action (relative to $\kappa$).
So the non-equivariance 
one-cocycle~\eqref{group_one_cocycle} and the non-equivariance 
\( 2 \)-cocycle~\eqref{eq:normedsquared:nonequiv_first}, namely,
\begin{equation}\begin{split}
\label{eq:normedsquared:nonequiv}
\Sigma(\xi, \eta) &\defeq
\kappa\left(\tangent_e \sigma (\xi), \eta \right) =
\kappa(J(m_0), \commutator{\xi}{\eta}) +
\omega_{m_0}(\xi \ldot m_0, \eta \ldot m_0) \\
&=J_{\commutator{\xi}{ \eta}}(m_0) + \poisson{J_\xi}{J_\eta}(m_0), 
\qquad \xi, \eta \in \LieA{g}
\end{split}\end{equation}
need not vanish. Since \(M\) is connected, \(\sigma \) and
\( \Sigma \) do not depend on the reference point \( m_0 \in M \) 
used in their definition (see, \eg, 
\parencite[Theorem~4.5.25]{OrtegaRatiu2003}).

As in finite dimensions, an \emphDef{almost complex structure} 
on a Fr\'echet manifold \( M \) is a collection of linear 
maps \( j_m: \TBundle_m M \to \TBundle_m M \) satisfying 
\( j_m^2 = - \id \). Moreover, we require \( j_m \) to be smooth 
in \( m \in M \), that is, relative to every chart 
\( M \supseteq U \to V \subseteq E \) on \( M \), where \(E\) 
is the model Fr\'echet space 
of \(M\),  the induced map \( V \times E \to E	\) is 
smooth. If the base point is clear from the 
context, then we simply write \( j \) in place of \( j_m \).
An almost complex structure \( j \) on \( M \) is said to 
be \emphDef{compatible} with the symplectic structure 
\( \omega \) if \( \omega(j \,\cdot, j \,\cdot) = 
\omega(\cdot, \cdot)  \) and \( \omega(X, j X) > 0 \) for all 
non-zero \( X \in \TBundle M \).
If a Lie group \( G \) acts on \( M \), it is naturally to assume 
that \( j \) is invariant under the action. However, this is not 
the case in the example of symplectic connections studied in 
\cref{sec:symplecticConnections}. In this example, the action 
is compatible with the almost complex structure only in the 
weak sense that the stabilizer of a point \( m \) leaves \( j_m \) 
invariant. It turns out that this is enough to obtain the structure 
theorem of the complex stabilizer and the Hessian.
Let \( \tau_j: G \to \EndBundle(\TBundle M) \) be the one-cocycle
\begin{equation}
\label{group_cocyle_in_terms_of_j}
	\tau_j(g) = \tangent \Upsilon_g \circ j \circ \tangent \Upsilon_{g^{-1}} - j
\end{equation}
measuring the non-equivariance of \( j \),
\ie, it satisfies $\tau_j (gh) = \tau_j (g) + g \cdot \tau_j(h)$ 
(which is easily checked) and $\tau_j(e) = 0$, where
$G$ acts on \(\EndBundle(\TBundle M)\) by \(g \cdot  S \defeq 
\tangent \Upsilon_g \circ S \circ \tangent \Upsilon_{g^{-1}}\).
Note that \(\tau_j (g)_m: \TBundle_m M \to \TBundle_m M\)
is a linear map for every \(m \in  M\).   
Let \( \tau_j': \LieA{g} \to \EndBundle(\TBundle M) \) be the 
associated Lie algebra  one-cocycle, \ie, 
\begin{equation}
\label{LA_one_cocycle_j}
\tau_j'(\xi)(v_m) \defeq \left.\frac{d}{dt}\right|_{t=0} \tau_j(\exp t \xi)(v_m)
\in  \TBundle_m M,
\end{equation} 
for all \(v_m \in \TBundle_m M\).
The Lie algebra one-cocycle identity is 
\begin{equation}
\label{LA_one_cocycle_identity_j}
\xi \ldot \tau_j'(\eta) - \eta\ldot \tau_j'(\xi) = 
\tau_j'([\xi, \eta]) \quad \text{for all} \quad \xi, \eta \in \mathfrak{g},  
\end{equation} 
where \(\mathfrak{g}\) acts on \(\EndBundle(\TBundle M) \) by 
\((\xi\ldot S)(v_m)\defeq \xi \ldot S(v_m) - S(\xi \ldot  v_m)\)
and the action of \(\mathfrak{g}\) on \( \TBundle_m M \) is given
by the tangent lift of the original \(G\)-action on \(M\), \ie, 
\[
\xi \ldot  v_m \defeq \left.\frac{d}{dt}\right|_{t=0} 
\tangent_m \Upsilon_{\exp t \xi}v_m \in  \TBundle_m M,
 \quad \text{for all} 
\quad v_m \in  \TBundle_m M.
\]
For a point \( m \in M \), we say that \( j_m \) is 
\emphDef{\( \LieA{g}_m \)-invariant} if 
\( \tau_j'(\xi)_m: \TBundle_m M \to \TBundle_m M \) vanishes 
for every \( \xi \in \LieA{g}_m \). Since \( \xi \in \LieA{g}_m \), 
this property indeed only depends on \( j_m \) and not on the 
equivariance behavior of \( j \) at other points.

We need one more notational convention: the \emphDef{adjoint} 
\(A^*\) of a linear continuous operator \(A:\mathfrak{g} \rightarrow  
\mathfrak{g}\), if it exists, is always taken relative to 
\( \kappa \) and is uniquely determined, \ie, 
\(\kappa (A \xi , \eta) = \kappa(\xi, A^* \eta)\) for all 
\(\xi , \eta \in \mathfrak{g}\). We assume that the adjoints 
\( \CoadAction_\xi \) of the adjoint operators exist for 
all \( \xi \in \LieA{g} \) (this is automatic in finite dimensions). 
For \(\sigma \in  \mathfrak{g}\), we say that $\kappa$ is 
\( \operatorname{ad}_\sigma \)-invariant if
\( \operatorname{ad}_\sigma ^* = - \operatorname{ad}_\sigma \). 
If \(\mathfrak{k} \subset \mathfrak{g}\) is a subset, we 
say that \( \kappa \) is \emphDef{
\( \operatorname{ad}_\mathfrak{k} \)-invariant} if
\( \operatorname{ad}_\sigma ^* = - \operatorname{ad}_\sigma \) 
for all \(\sigma \in  \mathfrak{k}\).

The following operators will play an essential role:
\begin{equation}\begin{split}
	\label{L_m_L_mu}
	L_m \xi &\defeq \TBundle_m J\bigl(j \, (\xi \ldot m)\bigr),\\
	Z_m \xi &\defeq \TBundle_m J\bigl(\xi \ldot m\bigr) = \Sigma_\kappa(\xi) - 
	\CoadAction_\xi J(m),
\end{split}\end{equation}
for \( m \in M \) and \( \xi \in \LieA{g} \).
Here the map \( \Sigma_\kappa: \LieA{g} \to \LieA{g} \) is defined 
by \( \kappa\bigl(\Sigma_\kappa(\xi), \eta\bigr) = \Sigma(\xi, \eta) \).
As we will see in \cref{sec:kaehler}, in the K\"ahler example, the 
operator \( L_m \) coincides with the operator introduced by 
Lichnerowicz. For this reason, we will also refer to it as 
the \emphDef{Lichnerowicz operator}.
The following summarizes some important properties of these operators.
\begin{prop}
\label{prop:normedsquared:propertiesL}
The following holds:
\begin{thmenumerate}
\item
\label{prop:normedsquared:propertiesL:symmetry}
\( L_m \) is symmetric and \( Z_m \) is skew-symmetric with respect 
to \( \kappa \).
\item
\label{prop:normedsquared:propertiesL:derivations}
For all \( \xi, \eta, \rho \in \LieA{g} \),
\begin{equation}
\label{eq:normedsquared:propertiesL:L1derivation}
\begin{split}
\kappa\bigl(L_m \commutator{\xi}{\eta}, \rho\bigr)
&= \kappa\bigl(\CoadAction_{\eta}{L_m\xi}, \rho\bigr) + 
\kappa\bigl(\CoadAction_{\xi} L_m\eta, \rho\bigr) \\
&\quad+ (\difLie_{j \xi^*} \omega)_m \bigl(\eta \ldot m, 
\rho \ldot m\bigr) \\
&\quad+ \kappa\bigl(\tangent_m J \bigl(\tau_j'(\eta) \, 
\xi \ldot m\bigr), \rho\bigr) - 
\kappa\bigl(\eta, \tangent_m J \bigl(\tau_j'(\rho) \, 
\xi \ldot m\bigr)\bigr).
\end{split}\end{equation}
\item
\label{prop:normedsquared:propertiesL:commute}
If \( \mu \in \LieA{g}_m \), then
\begin{subequations}
	\label{eq:normedsquared:propertiesL:commute}
	\begin{align}
		L_m^{} \adAction_\mu^{} \xi &= 
		- \adAction_\mu^* L_m^{} \xi 
		- \tangent_m J \Bigl(\tau_j'\bigl(\mu
		\bigr) \, \xi \ldot m\Bigr), \\
		Z_m^{} \adAction_\mu^{} \xi 
		&= - \adAction_\mu^* Z_m^{} \xi.		
	\end{align}
\end{subequations}
In particular, if \( j_m \) is \( \LieA{g}_m \)-invariant 
and \( \kappa \) is \( \adAction_{\LieA{g}_m} \)-invariant, 
then \( \adAction_\mu \) commutes with \( L_m \) and \( Z_m \).
\qedhere
\end{thmenumerate}
\end{prop}
\begin{proof}
	(i): For all \( \xi, \eta \in \LieA{g} \), we have
\begin{equation}
\label{L_m_j}
\kappa(L_m \xi, \eta) 
= \kappa\left(\tangent_m J \bigl(j \, 
(\xi \ldot m) \bigr), \eta\right)
\stackrel{\eqref{tangent_momentum}}= 
\omega_m\bigl(j \, (\xi \ldot m), \eta \ldot m\bigr). 
\end{equation}
Thus,
\begin{equation}
	\kappa(L_m \xi, \eta) 
		= \omega_m\bigl(j \, (\xi \ldot m), \eta \ldot m\bigr)
		= \omega_m\bigl(j \, (\eta \ldot m), \xi \ldot m\bigr)
		= \kappa(L_m \eta, \xi),
\end{equation}
showing that \(L_m\) is \(\kappa \)-symmetric.

On the one hand, symmetry of \(\kappa \) implies
\begin{equation}
\begin{split}
	 - \kappa\bigl(\CoadAction_\xi J(m), \eta\bigr)
		= \kappa\bigl(J(m), \commutator{\eta}{\xi}\bigr)
		= \kappa\bigl(\CoadAction_\eta {J(m)}, \xi \bigr)
		= \kappa\bigl(\xi, \CoadAction_\eta {J(m)}\bigr).
\end{split}
\end{equation} 
On the other hand, \( \Sigma_\kappa \) is 
\(\kappa \)-skew-symmetric because \( \Sigma \) is 
skew-symmetric. Together these facts show that \(Z_m\) 
is \(\kappa \)-skew-symmetric.

(ii): For all \( \xi, \eta, \rho \in \LieA{g} \), the 
Leibniz rule for the Lie derivative implies
\begin{equation}
\label{equ_Lie_derivative_almost_complex}
\begin{split}
(\difLie_{j \xi^*} \omega)_m \bigl(\eta \ldot m, 
\rho \ldot m\bigr)
&= j \xi^* \bigl(\omega(\eta^*, \rho^*)\bigr) (m) \\
&\quad- \omega_m\bigl(\commutator{j \xi^*}{\eta^*}_m, 
\rho \ldot m\bigr)
- \omega_m \bigl(\eta \ldot m, 
\commutator{j \xi^*}{\rho^*}_m\bigr).
\end{split}
\end{equation}
By~\eqref{eq:normedsquared:nonequiv}, we have 
\( \omega_m \bigl(\eta \ldot m, \rho \ldot m \bigr) = 
\Sigma\bigl(\eta, \rho\bigr) - 
\kappa\bigl(J(m), \commutator{\eta}{\rho}\bigr) \).
Thus, the first term equals
\begin{equation}
\label{equ_Sec3_first_term}
j \xi^* \bigl(\omega(\eta^*, \rho^*)\bigr) (m)
= - \kappa\Bigl(\tangent_m J\bigl(j(\xi \ldot m)\bigr), 
\commutator{\eta}{\rho}\Bigr)
= - \kappa\bigl(\CoadAction_{\eta}{\tangent_m J(j(\xi \ldot m))}, 
\rho\bigr).
\end{equation}
For the other terms, we need the identity
\begin{equation}
\label{j_commutator}
\commutator{\eta^*}{j X} = j \, \commutator{\eta^*}{X} 
- \tau_j'(\eta) \, X
\end{equation}  
for any vector field \(X \in \VectorFieldSpace(M)\) and 
\( \eta \in \LieA{g} \). Indeed, by the 
definition~\eqref{group_cocyle_in_terms_of_j} of \( \tau_j \), 
we have 
\begin{equation}
\label{eq:normedsquared:commutatorActionjSec3}
\tangent_{g^{-1} \cdot m} \Upsilon_{g} \circ j_{g^{-1} \cdot m}
= j_m \circ \tangent_{g^{-1} \cdot m} \Upsilon_g + 
\tau_j(g) \circ \tangent_{g^{-1} \cdot m} \Upsilon_g
\end{equation}
for all \( g \in G \). Put here \(g=\exp(-t \eta)\) and apply
the resulting identity to the vector \(X(\exp(t \eta) \cdot m)\) 
to get
\begin{align*}
\left( \Upsilon_{\exp (t \eta)}^{\,*} (jX)\right) (m) &=
\left(\tangent\Upsilon_{\exp(-t \eta)} \circ jX \circ  
\Upsilon_{\exp (t \eta)} \right)(m) \\
&=\tangent_{\exp(t \eta) \cdot m} \Upsilon_{\exp(-t \eta)} 
\left( j_{\exp(t \eta)\cdot m}X(\exp(t \eta) \cdot m) \right)\\
&= j_m \left( \tangent_{\exp(t \eta) \cdot m} \Upsilon_{\exp(-t \eta)}
X(\exp(t \eta) \cdot m) \right) \\
&\quad + \tau_j(\exp(-t \eta)) \left( 
\tangent_{\exp(t \eta) \cdot m} \Upsilon_{\exp(-t \eta)}
X(\exp(t \eta) \cdot m) \right) \\
&=j_m \bigl(\Upsilon_{\exp(t \eta)}^* X \bigr) (m) +
\tau_j(\exp(-t \eta))_m \left(\bigl( 
\Upsilon_{\exp(t \eta)}^* X \bigr) (m)\right),   
\end{align*} 
that is,
\begin{equation}
	\Upsilon_{\exp (t \eta)}^{\,*} (jX) = j \, \Upsilon_{\exp (t \eta)}^{\,*} X 
	+ \tau_j\bigl(\exp (-t \eta)\bigr) \, \Upsilon_{\exp (t \eta)}^{\,*} X.
\end{equation}
Taking the \(t\)-derivative of this relation and recalling that 
\(j: \TBundle M \to \TBundle M\) is linear on the fibers and \(\tau(e) = 0\), 
yields~\eqref{j_commutator}. 

Summarizing, using~\eqref{equ_Lie_derivative_almost_complex},~\eqref{equ_Sec3_first_term},~\eqref{j_commutator}, the identity 
\( \commutator{\xi^*}{\eta^*} = - \commutator{\xi}{\eta}^*\) 
for all \( \xi, \eta \in \LieA{g} \),~\eqref{L_m_j},~\eqref{tangent_momentum}, 
and (i), we obtain
\begin{equation}
\begin{split}
(\difLie_{j \xi^*} \omega)_m &\bigl(\eta \ldot m, \rho \ldot m\bigr)\\
	&= - \kappa\bigl(\CoadAction_{\eta}{\tangent_m J(j(\xi \ldot m))}, \rho\bigr) \\
		&\qquad+ \omega_m\bigl(j \, \commutator{\xi}{\eta} \ldot m - 
		  \tau_j'(\eta) \, \xi \ldot m, \rho \ldot m\bigr) \\
		&\qquad+ \omega_m\bigl(\eta \ldot m, j \, \commutator{\xi}{\rho} \ldot m 
		- \tau_j'(\rho) \, \xi \ldot m\bigr)   \\
		&= - \kappa\bigl(\CoadAction_{\eta}{L_m\xi}, \rho\bigr)
			+ \kappa\bigl(L_m \commutator{\xi}{\eta}, \rho\bigr)
			- \kappa\bigl(L_m \commutator{\xi}{\rho}, \eta\bigr) \\
			&\qquad- \kappa\bigl(\tangent_m J \bigl(\tau_j'(\eta) \, 
			\xi \ldot m\bigr), \rho\bigr)
			+ \kappa\bigl(\eta, \tangent_m J \bigl(\tau_j'(\rho) \, 
			\xi \ldot m\bigr)\bigr)   \\
		&= - \kappa\bigl(\CoadAction_{\eta}{L_m\xi}, \rho\bigr)
			+ \kappa\bigl(L_m \commutator{\xi}{\eta}, \rho\bigr)
			- \kappa\bigl(\rho, \CoadAction_{\xi} L_m\eta\bigr) \\
			&\qquad- \kappa\bigl(\tangent_m J \bigl(\tau_j'(\eta) \, 
			\xi \ldot m\bigr), \rho\bigr)
			+ \kappa\bigl(\eta, \tangent_m J \bigl(\tau_j'(\rho) \, 
			\xi \ldot m\bigr)\bigr).
\end{split}
\end{equation}
This establishes~\eqref{eq:normedsquared:propertiesL:L1derivation}.

(iii): Since \( \mu \in \LieA{g}_m \), 
using~\eqref{eq:normedsquared:propertiesL:L1derivation},~\eqref{tangent_momentum}, 
and the definition of \(L_m\) 
in~\eqref{L_m_L_mu} with 
\( \eta = \mu \), yields the first identity 
in~\eqref{eq:normedsquared:propertiesL:commute}. For the second 
identity, first observe that the definition of \(Z_m\) in~\eqref{L_m_L_mu}, 
the Jacobi identity and the 2-cocycle identity 
of \( \Sigma \) imply
\begin{equation}
\kappa(Z_m \adAction_\mu \xi, \eta) = 
- \kappa(Z_m \adAction_\eta \mu, \xi) 
- \kappa(Z_m \adAction_\xi \eta, \mu)
\end{equation}
for all \( \xi, \eta, \mu \in \LieA{g} \).
If \( \mu \in \LieA{g}_m \), then the second term vanishes 
(use~\eqref{tangent_momentum}).
Hence, in this case, the skew-symmetry of \( Z_m \) yields
\begin{equation}
\kappa(Z_m \adAction_\mu \xi, \eta) = 
- \kappa(Z_m \adAction_\eta \mu, \xi) = 
- \kappa(\eta, \CoadAction_\mu Z_m \xi).
\end{equation}
Since \( \eta \) was arbitrary, we conclude \( Z_m \adAction_\mu \xi 
= - \CoadAction_\mu Z_m \xi \).
\end{proof}

\begin{remark}
If \( J \) is \( \adAction \)-equivariant, then 
\( Z_m = - \adAction_{J(m)} \) and 
\cref{prop:normedsquared:propertiesL:commute} 
implies that \( L_m \) and \( Z_m \) commute 
if \( J(m) \in \LieA{g}_m \).

However, without this additional assumption on the 
equivariance of \( J \), there is no hope that 
\( L_m \) and \( Z_m \) still commute even if \( j_m \) 
is \( \LieA{g}_m \)-invariant and \( \kappa \) is 
\( \AdAction \)-invariant. In fact, without the 
equivariance one cannot control the relation between 
the non-equivariance cocycle \( \Sigma \) and the 
almost-complex structure \( j \). For example, in the 
affine setting considered at the end of 
\cref{sec:affineSymplectic}, an affine action preserves 
a constant almost-complex structure if and only if its 
linear part does so; at the same time, the non-equivariance 
cocycle is completely controlled by the affine part of the action.
\end{remark}

\begin{remark}
Often one can cast a problem involving a non-equivariant 
momentum map into questions in an equivariant setting by 
passing to a central extension. In fact, if 
\( J: M \to \LieA{g} \) is a non-equivariant momentum map 
with 2-cocycle \( \Sigma \), then the map 
\begin{equation}
\hat{J}: M \to \hat{\LieA{g}}, \qquad 
\hat{J}(m) \defeq \bigl(J(m), -1 \bigr)
\end{equation} 
is an equivariant momentum map for the natural action 
of the centrally extended Lie algebra \( \hat{\LieA{g}} 
= \LieA{g} \oplus_{\Sigma} \R \) associated with the 
cocycle \( \Sigma \) (\ie, the bracket
on \(\hat{\LieA{g}}\) is given by \([(\xi, s),(\eta , t)]
\defeq([\xi, \eta], \Sigma(\xi , \eta))\)). However, the pairing
\begin{equation}
\hat{\kappa}: \hat{\LieA{g}} \times \hat{\LieA{g}} \to \R, 
\qquad \hat{\kappa}\bigl((\xi, s), (\eta, t)\bigr) 
= \kappa(\xi, \eta) + st
\end{equation}
is not \( \AdAction \)-invariant even if \( \kappa \) is.
For this reason, the associated operators 
\( \hat{L}_m, \hat{Z}_m \) do not commute with 
\( \hat{\adAction}_\mu \)	on \( \hat{\LieA{g}} \).
Thus, perhaps somewhat surprisingly, the strategy of 
passing to central extensions is not helpful in this 
case, as one only trades non-equivariance of the momentum 
map with non-invariance of the inner product.
\end{remark}

At this point, it is convenient to turn to the complex picture.
Thus, let \( \LieA{g}_\C = \LieA{g} \oplus \I \LieA{g} \) be the 
complexification of \( \LieA{g} \).
We will also encounter \( \R \)-linear, but not complex-linear, 
operators \( \LieA{g}_\C \to \LieA{g}_\C \), and then write 
them in matrix form as
\begin{equation}
\label{T_operator}
\Matrix{T_{11} & T_{12} \\ T_{21} & T_{22}} (\xi_1 + \I \xi_2) 
= T_{11} \xi_1 + T_{12} \xi_2 + \I T_{21} \xi_1 + \I T_{22} \xi_2
\end{equation}
for \( \mathbb{R} \)-linear operators \( T_{11}, T_{12}, T_{21}, 
T_{22}: \LieA{g} \to \LieA{g} \). In particular, the 
\(\mathbb{C}\)-linear operator \( S + \I T \) is 
written as \( \smallMatrix{S & -T \\ T & S} \) in matrix form;
the \( \mathbb{C} \)-linear operator given by multiplication 
by \( \I \) is hence given by the matrix 
\( \smallMatrix{0 & -I \\ I & 0} \).

We introduce the complex-linear operators 
\( C^\pm_m: \LieA{g}_\C \to \LieA{g}_\C \) for every \( m \in M \) by
\begin{equation}
	\label{eq:decomposition:calabiOperatorDef}
	C^\pm_m \defeq L_m \pm \I Z_m.
\end{equation}
As we shall see in \cref{sec:kaehler}, in the K\"ahler 
example, these operators are generalizations of the 
operators introduced by Calabi. For this reason, we will 
also refer to them as the \emphDef{Calabi operators}.
The Lie algebra action of \(\LieA{g}\) on \( M \) extends to 
an \( \I \)-\( j \)-complex-linear map
\begin{equation}
\label{extended_complex_action}
\Upsilon_m: \LieA{g}_\C \to \TBundle_m M, \qquad 
(\xi_1 + \I \xi_2) \mapsto (\xi_1 + \I \xi_2) \ldot m 
\defeq \xi_1 \ldot m + j \, (\xi_2 \ldot m).
\end{equation}
Note that~\eqref{extended_complex_action} is 
equivalent to defining the infinitesimal generator of an 
imaginary Lie algebra element  \( \I \xi \), \( \xi \in  
\LieA{g} \), by \( (\I \xi)^\ast \defeq j \xi^\ast \).
Since we do not assume that the \( G \)-action leaves \( j \) 
invariant nor that \( j \) is integrable (\ie, 
the Nijenhuis tensor of \(j\) vanishes), the resulting 
map \( \LieA{g}_\C \to \VectorFieldSpace(M) \) is not 
necessarily a Lie algebra homomorphism.
In particular, the kernel of \( \Upsilon_m \), denoted by 
\( (\LieA{g}_\C)_m \), is not necessarily a complex Lie subalgebra
of \( \LieA{g}_\C \).

Let \( \kappa_\C: \LieA{g}_\C \times \LieA{g}_\C \to \C \) be the extension 
of \( \kappa \) to a Hermitian inner 
product on \( \LieA{g}_\C \), 
\begin{equation}
\label{kappa_C}
\kappa_\C (\xi_1 + \I \xi_2, \eta_1 + \I \eta_2) \defeq 
\kappa (\xi_1, \eta_1) + \I \kappa(\xi_2, \eta_1) - 
\I \kappa(\xi_1, \eta_2) + \kappa(\xi_2, \eta_2),
\end{equation}
which is complex-linear in the first argument and 
complex-antilinear in the second argument.

Finally, let \( h\defeq g - \I \omega \) denote the Hermitian metric 
associated with \( \omega \) and \( j \); recall, 
\( g(\cdot  , \cdot )\defeq \omega(\cdot, j \cdot) \).

In terms of these operators, we obtain the following result.

\begin{prop}
\label{prop_C_m}
	Let \( m \in M \)  such that \( j_m \) is invariant under \( \LieA{g}_m \).
	Then the following statements hold:
	\begin{thmenumerate}
\item
	\label{i:normedsquared:calabiSelfAdjoint}
	\( C^\pm_m \) are Hermitian with respect to \( \kappa_\C \).
	\item
	\label{i:normedsquared:calabiCommuteAdjoint}
	For every \( \mu \in \LieA{g}_m \), 
	\begin{equation}
		C^\pm_m \adAction_\mu = -\adAction_\mu^* C^\pm_m \, .
	\end{equation}
	In particular, if \( \kappa \) is \( \adAction_{\mu} \)-invariant, 
	then \( \commutator{C^\pm_m}{\adAction_\mu} = 0 \).
\item 
\label{i:normedsquared:calabiCommuteEachOther}
If \( J \) is \( \adAction \)-invariant, \( J(m) \in \LieA{g}_m \), 
and \( \kappa \) is \( \adAction_{J(m)} \)-invariant, then 
\( \commutator{C^+_m}{C^-_m} = 0 \).
\item
\label{i:normedsquared:kernelsIdentified}
For all \( \zeta, \gamma \in \LieA{g}_\C \),
\begin{subequations}\label{eq:normedsquared:dualL}
	\begin{align}
		\label{eq:normedsquared:dualLPlus}
		\kappa_\C (C^+_m \zeta, \gamma) &= 
		- h(\zeta \ldot m, \gamma \ldot m),
		\\
		\label{eq:normedsquared:dualLMinus}
		\kappa_\C (C^-_m \zeta, \gamma) &= 
		- h(\bar{\gamma} \ldot m, \bar{\zeta} \ldot m).
	\end{align}
\end{subequations}
In particular, the operators \( C^\pm_m \) on \( \LieA{g}_\C \) 
are negative\footnotemark{}.\footnotetext{An operator 
\( L: \LieA{g}_\C \to \LieA{g}_\C \) is called negative 
if \( \kappa_\C(L \zeta, \zeta) \leq 0 \) for all 
\( \zeta \in \LieA{g}_\C \).}
Moreover,
\begin{equation}
\label{eq:kernel_relations_L_plus_L_minus}
	\ker C^+_m  = (\LieA{g}_\C)_m\, , \qquad 
	\ker C^+_m \intersect \ker C^-_m
	= (\LieA{g}_m)_\C \, .
\end{equation}

\item
\label{i:normedsquared:LPlusAlternativeDef}
If \( \Upsilon_m^*: \TBundle_m M \to \LieA{g}_\C \) denotes the adjoint of 
\( \Upsilon_m: \LieA{g}_\C \to \TBundle_m M \) with respect to \( \kappa_\C \) 
and the Hermitian metric \( h \), then
	\begin{equation}
		\label{eq:normedsquared:LPlusAlternativeDef}
		C^+_m = - \Upsilon_m^* \Upsilon_m \, .
		\qedhere
	\end{equation}
\item 
\label{i:normedsquared:calabiImaginary}
\( \Im C^+_m = \tangent_m J \circ \Upsilon_m \).
\end{thmenumerate}
\end{prop}

\begin{proof}
Using \cref{prop:normedsquared:propertiesL:symmetry}, a
direct verification shows that \( C^\pm_m \) is Hermitian.
Moreover, by \cref{prop:normedsquared:propertiesL:commute}, we have
\begin{equation}
	C^\pm_m \adAction_\mu = (L_m \pm \I Z_m) \adAction_\mu
	= - \adAction_\mu^* (L_m \pm \I Z_m)
	= - \adAction_\mu^* C^\pm_m \, .
\end{equation}
This proves point~\labeliref{i:normedsquared:calabiCommuteAdjoint}.

Concerning~\labeliref{i:normedsquared:calabiCommuteEachOther}, 
if \( J \) is \( \adAction \)-invariant, then \( Z_m = - \adAction_{J(m)} \).
Thus,
\begin{equation}
	\commutator{C^+_m}{C^-_m}
		= \commutator{C^+_m}{C^+_m - 2 \I Z_m}
		= 2 \I \commutator{C^+_m}{\adAction_{J(m)}}.
\end{equation}
Since \( \adAction_{J(m)} \) is \( \adAction \)-invariant 
and \( J(m) \in \LieA{g}_m \), point 
\labeliref{i:normedsquared:calabiCommuteAdjoint} implies 
that \( C^+_m \) commutes with \( \adAction_{J(m)} \), and 
thus with \( C^-_m \).
 
We now prove \labeliref{i:normedsquared:kernelsIdentified} and 
\labeliref{i:normedsquared:LPlusAlternativeDef}.
Since both sides of~\eqref{eq:normedsquared:dualL} are complex-linear 
in \( \zeta \) and complex-antilinear in \( \gamma \), it suffices to 
consider the case where \( \zeta, \gamma \in \LieA{g} \).
Using the Riemannian metric \( g(\cdot, \cdot) = \omega(\cdot, j \cdot) \), 
we obtain
\begin{equation}
\kappa(L_m \zeta, \gamma)
\stackrel{\eqref{L_m_j}}= \omega_m\bigl(j \, (\zeta \ldot m), \gamma \ldot m\bigr)
= - g_m\bigl(\zeta \ldot m, \gamma \ldot m\bigr). 
\end{equation}
Moreover,
\begin{equation}
\kappa(Z_m \zeta, \gamma)
\stackrel{\eqref{tangent_momentum}}=
\omega_m\bigl(\zeta \ldot m, \gamma \ldot m\bigr).
\end{equation}
These identities directly imply
\begin{equation}\begin{split}
\kappa_\C\bigl((L_m \pm \I Z_m) \zeta, \gamma\bigr)
&= - g_m\bigl(\zeta \ldot m, \gamma \ldot m\bigr) \pm \I \omega_m\bigl(\zeta 
\ldot m, \gamma \ldot m\bigr).
\end{split}\end{equation}
Rewriting this in terms of the Hermitian metric \( h = g - \I \omega \) 
yields~\eqref{eq:normedsquared:dualL}. Negativity of 
\( C^\pm_m \) and the expression for the kernels 
follows directly from~\eqref{eq:normedsquared:dualL} by considering the 
case \( \gamma = \zeta \). Moreover, expressed using the operator 
\( \Upsilon_m: \LieA{g}_\C \to \TBundle_m M \), \cref{eq:normedsquared:dualLPlus} 
reads \( \kappa_\C (C^+_m \zeta, \gamma) = 
- h(\Upsilon_m \zeta, \Upsilon_m \gamma) = 
- \kappa_\C(\Upsilon_m^* \Upsilon_m \zeta, \gamma) \), which 
verifies~\eqref{eq:normedsquared:LPlusAlternativeDef}.

Finally, for every \( \xi_1, \xi_2 \in \LieA{g} \), we have
\begin{equation}
	\Im C^+_m (\xi_1 + \I \xi_2) = Z_m \xi_1 + L_m \xi_2 
	= \tangent_m J\bigl(\xi_1 \ldot m + j (\xi_2 \ldot m)\bigr) \, ,
\end{equation}
which proves~\labeliref{i:normedsquared:calabiImaginary}.
\end{proof}

The upshot of the next theorem is that one obtains a 
root-space-like decomposition of \( (\LieA{g}_\C)_m \) 
with respect to an Abelian subalgebra of \( \LieA{g}_m \).
This is quite surprising since, under our weak assumptions, 
\( (\LieA{g}_\C)_m \) is not even a Lie subalgebra of 
\( \LieA{g}_\C \), in general.

\begin{thm}
\label{prop:decomposition:decompositionComplexStab}
Let \( (M, \omega) \) be a connected symplectic Fr\'echet 
manifold endowed with a symplectic action of a Fr\'echet 
Lie group \( G \) and let \( j \) be an almost complex 
structure on \( M \) compatible with \( \omega \). Assume 
that the action has a momentum map \( J: M \to \LieA{g} \) 
with non-equivariance cocycle \( \Sigma \) relative to a 
non-degenerate symmetric pairing \( \kappa: \LieA{g} 
\times \LieA{g} \to \R \). For \( m \in M \), let 
\( \LieA{t} \subseteq \LieA{g}_m \) be an Abelian 
subalgebra such that \( j_m \) and \( \kappa \) are 
invariant under \( \LieA{t} \). In infinite dimensions, 
additionally assume the following:
\begin{enumerate}
\item The adjoints of \( \adAction_\xi: \LieA{g} \to \LieA{g} \) 
exist for all \( \xi \in \LieA{g} \).
\item The map \( \Sigma_\kappa: \LieA{g} \to \LieA{g} \) 
defined by \( \kappa\bigl(\Sigma_\kappa(\xi), \eta\bigr) 
= \Sigma(\xi, \eta) \) exists.
\item The stabilizer \( (\LieA{g}_\C)_m \) is finite-dimensional.
\end{enumerate} 
Then the following decomposition holds:
\begin{equation}
\label{eq:decomposition:decompositionComplexStab}
(\LieA{g}_\C)_m = \aCentralizer(\LieA{t}) \oplus 
\bigoplus_{\lambda \neq 0} \LieA{k}_\lambda,
\end{equation}
where:
\begin{thmenumerate}
\item
\( \aCentralizer(\LieA{t}) \) is the centralizer of 
\( \LieA{t} \), \ie the subspace of \( (\LieA{g}_\C)_m \) 
consisting of elements that commute with all \( \mu \in \LieA{t} \);
\item \( \LieA{t}_\C \subseteq \aCentralizer(\LieA{t}) \);
\item \( \lambda \) ranges over those real-valued linear functionals 
on \( \LieA{t} \) for which \( \LieA{k}_\lambda = 
\set{\zeta \in (\LieA{g}_\C)_m \given \I \adAction_\mu \zeta 
= \lambda(\mu) \zeta \text{ for all } \mu \in \LieA{t}} \) 
is non-trivial; in particular, $\aCentralizer(\LieA{t}) = \LieA{k}_0$;
\item
\( \commutator{\LieA{k}_\lambda}{\LieA{k}_\nu} 
\intersect (\LieA{g}_\C)_m \subseteq 
\LieA{k}_{\lambda + \nu} \) for all 
\( \lambda, \nu \in \LieA{t}^* \).
\item If \( \lambda \neq \nu \), then \( \LieA{k}_\lambda \) and \( \LieA{k}_\nu \) are orthogonal with respect to \( \kappa_\C \).\qedhere
\end{thmenumerate}
\end{thm}
\begin{proof}
For every \( \mu \in \LieA{t} \), \( \kappa \) is 
\( \adAction_\mu \)-invariant by assumption.
Thus, \cref{i:normedsquared:calabiCommuteAdjoint} implies 
that \( \adAction_\mu \) commutes with \( C^+_m \) and so 
it restricts to an operator on \( \ker C^+_m = (\LieA{g}_\C)_m \), 
\cf~\eqref{eq:kernel_relations_L_plus_L_minus}.
Since \( (\LieA{g}_\C)_m \) is a finite-dimensional complex 
vector space, the Hermitian operators \( \I \adAction_\mu: 
(\LieA{g}_\C)_m \to (\LieA{g}_\C)_m \) are simultaneously 
diagonalizable (for \( \mu \in \LieA{t} \)).
In other words, \( (\LieA{g}_\C)_m \) is a direct sum 
of the subspaces 
\begin{equation}
\LieA{k}_\lambda = \set{\zeta \in (\LieA{g}_\C)_m 
\given \I \adAction_\mu \zeta = \lambda(\mu) \zeta 
\text{ for all } \mu \in \LieA{t}},
\end{equation}
where \( \lambda \) ranges over real-valued functionals 
in \( \LieA{t}^* \). Define $\aCentralizer(\LieA{t}) = 
\LieA{k}_0$ and note that it equals the set of elements
in \( (\LieA{g}_\C)_m \) that commute with $\mu \in \LieA{t}$.
Since \( \LieA{t} \) is Abelian, its complexification is 
clearly contained in \( \aCentralizer(\LieA{t}) \). 
The inclusion \( \commutator{\LieA{k}_\lambda}{\LieA{k}_\nu}
\intersect (\LieA{g}_\C)_m \subseteq \LieA{k}_{\lambda + \nu}\) 
follows from the Jacobi identity.
Finally, if \( \lambda \neq \nu \), then there exists 
\( \mu \in \LieA{t} \) such that \( (\lambda - \nu)(\mu) \neq 0 \). 
But by \( \LieA{t} \)-invariance of \( \kappa \), we have for 
every \( \zeta \in \LieA{k}_\lambda \) and \( \eta \in \LieA{k}_\nu \):
\begin{equation}
\bigl(\lambda(\mu) - \nu(\mu)\bigr) \kappa_\C\bigl(\zeta, \eta\bigr) 
= \kappa_\C\bigl( \I \adAction_\mu \zeta, \eta \bigr) 
- \kappa_\C\bigl(\zeta, \I \adAction_\mu \eta \bigr)
= 0.
\end{equation}
Thus, \( \kappa_\C(\zeta, \eta) = 0 \).
\end{proof}

\begin{remark}
	\label{rem:decomposition:decompositionComplexStab:nonInvariantKappa}
If \( \kappa \) is not invariant under \( \LieA{t} \), then by 
\cref{i:normedsquared:calabiCommuteAdjoint} the operators 
\( \adAction_\mu \) still restrict to operators on 
\( \ker C^+_m = (\LieA{g}_\C)_m \), but they are no 
longer necessarily diagonalizable. Thus, a similar result 
holds in this case by replacing \( \LieA{k}_\lambda \) 
by generalized eigenspaces.
\end{remark}

\section{Norm-squared of the momentum map}
\label{sec:momentumMapSquared}

In this section we investigate the norm-squared of the momentum map 
and calculate its Hessian.
We thereby expand upon the results in 
\parencite{GeorgoulasRobbinEtAl2018,Wang2004,Wang2006}
which studied similar problems in the finite-dimensional K\"ahler setting.
Our approach differs in the following notable points.
First, having the application 
to momentum maps on contractible spaces in mind, we give up the assumption that the 
momentum map is equivariant. 
Second, we work in the framework of 
Fr\'echet manifolds and address the 
functional analytical problems arising from the transition to the 
infinite-dimensional setting. We refer to 
\parencite{DiezThesis,DiezSingularReduction} for background information 
concerning symplectic geometry on Fr\'echet manifolds.
Finally, we generalize the 
treatment in the aforementioned works to the more general case where the 
complex structure is not integrable.
This is mainly to circumvent the technical difficulties arising from 
the infinite-dimensional setting, where the correct notion of 
integrability is no longer clear.

We continue to work in the general setting of the previous section.
Thus, \( (M, \omega) \) is a connected symplectic Fr\'echet manifold 
with a symplectic action of a Fr\'echet Lie group \( G \), and we 
assume that the action has a Lie algebra-valued momentum map 
\( J: M \to \LieA{g} \) relative to a non-degenerate, symmetric, 
not necessarily \( \AdAction_G \)-invariant, pairing 
\( \kappa: \LieA{g} \times \LieA{g} \to \R \).
 We are interested in the norm-squared 
\( \norm{J}^2_\kappa: M \to \R \):
\begin{equation}
	\norm{J}^2_\kappa(m) \defeq \kappa\bigl(J(m), J(m)\bigr).
\end{equation}
The following gives a first hint that the behavior of 
\( \norm{J}^2_\kappa \) at a critical point is tightly connected 
to the structure of the stabilizer Lie algebra at this point.
\begin{prop}
\label{prop:normedsquared:criticalPoints}
A point \( m \in M \) is a critical point of \( \norm{J}^2_\kappa \) 
if and only if \( J(m) \) is an element of the stabilizer algebra 
\( \LieA{g}_m \) of \( m \).
In particular, every point fixed by the \( G \)-action is a 
critical point of \( \norm{J}^2_\kappa \).
\end{prop}
\begin{proof}
We have
\begin{equation*}
	\tangent_m \norm{J}^2_\kappa \, (X)
		= 2 \, \kappa\bigl(\tangent_m J (X), J(m)\bigr)
		\stackrel{\eqref{tangent_momentum}}
		= - 2 \, \omega_m \bigl(J(m) \ldot m, X\bigr)
\end{equation*}
for all \( m \in M \) and \( X \in \TBundle_m M \). By (weak)
non-degeneracy of \( \omega \), the point \( m \) is a critical point of 
\( \norm{J}^2_\kappa \) if and only if \( J(m) \ldot m \) vanishes, 
\ie, \( J(m) \in \LieA{g}_m \).
\end{proof}

The upshot of the next theorem is that, at a critical point 
\( m \), the operator \( \I \adAction_{J(m)} \) defines a 
grading of the stabilizer of the complexified \textquote{action}.
Note that under our weak assumptions on the equivariance of 
\( j \) under the \( G \)-action, the complex-linear extension 
\( \Upsilon_m: \LieA{g}_\C \to \TBundle_m M \) does not 
necessarily give rise to an action of \( \LieA{g}_\C \) and 
the stabilizer \( (\LieA{g}_\C)_m = \ker \Upsilon_m \) is not 
necessarily a Lie algebra.
\begin{thm}
\label{prop:normedsquared:decompositionComplexStab}
Let \( (M, \omega) \) be a connected symplectic Fr\'echet manifold endowed with 
a symplectic action of a Fr\'echet Lie group \( G \) and let \( j \) be an almost 
complex structure on \( M \) compatible with \( \omega \). Assume that the action 
has a momentum map \( J: M \to \LieA{g} \) with non-equivariance cocycle 
\( \Sigma \) relative to a non-degenerate symmetric pairing \( \kappa: \LieA{g} 
\times \LieA{g} \to \R \). Let \( m \in M \) be a critical point of 
\( \norm{J}^2_\kappa \) such that \( j_m \) and \( \kappa \) are invariant 
under \( J(m) \). In infinite dimensions, additionally assume the following:
\begin{enumerate}
	\item The adjoints of \( \adAction_\xi: \LieA{g} \to \LieA{g} \) 
	exist for all \( \xi \in \LieA{g} \).
	\item The map \( \Sigma_\kappa: \LieA{g} \to \LieA{g} \) 
	defined by \( \kappa\bigl(\Sigma_\kappa(\xi), \eta\bigr) 
	= \Sigma(\xi, \eta) \) exists.
	\item The stabilizer \( (\LieA{g}_\C)_m \) is finite-dimensional.
\end{enumerate} 
Then the following decomposition holds:
\begin{equation}
\label{eq:normedsquared:decompositionComplexStab}
(\LieA{g}_\C)_m = \LieA{c}_m \oplus \bigoplus_{\lambda \neq 0} \LieA{k}_\lambda,
\end{equation}
where:
\begin{thmenumerate}
\item
\( \LieA{c}_m \) is the subspace of \( (\LieA{g}_\C)_m \) 
consisting of all elements that commute with \( J(m) \);
\item \( \C J(m) \subseteq \LieA{c}_m \);
\item
\( \LieA{k}_\lambda \) are eigenspaces of \( \I \adAction_{J(m)} \) 
with eigenvalue \( \lambda \in \R \) {\rm(}with the convention that 
\( \LieA{k}_\lambda = \set{0} \) if \( \lambda \) is not an 
eigenvalue{\rm)}; in particular, $\LieA{c}_m = \LieA{k}_0$;
\item
\( \commutator{\LieA{k}_\lambda}{\LieA{k}_\mu} \intersect 
(\LieA{g}_\C)_m \subseteq \LieA{k}_{\lambda + \mu} \) if 
$\lambda + \mu$ is an eigenvalue 
of \( \I \adAction_{J(m)} \); otherwise 
\( \commutator{\LieA{k}_\lambda}{\LieA{k}_\mu} \intersect 
(\LieA{g}_\C)_m = 0 \).
\item If \( \lambda \neq \nu \), then \( \LieA{k}_\lambda \) 
and \( \LieA{k}_\nu \) are orthogonal with respect to \( \kappa_\C \).
\end{thmenumerate}
If \( J \) is \( \adAction \)-equivariant, then the 
decomposition~\eqref{eq:normedsquared:decompositionComplexStab} 
is refined as:
\begin{equation}
	\label{eq:normedsquared:decompositionComplexStabEquiv}
(\LieA{g}_\C)_m = (\LieA{g}_m)_\C \oplus 
\bigoplus_{\lambda < 0} \LieA{k}_\lambda.
\end{equation}
If, additionally, \( \LieA{g}_m \) is a compact subalgebra 
of \( \LieA{g} \), then \( (\LieA{g}_m)_\C \) is a maximal 
reductive complex Lie algebra contained in \( (\LieA{g}_\C)_m \).
\end{thm}

\begin{proof}
By \cref{prop:normedsquared:criticalPoints}, the point \( m \) 
is a critical point of \( \norm{J}^2_\kappa \) if and only 
if \( J(m) \in \LieA{g}_m \).
Thus, the first part of the theorem concerning the 
decomposition~\eqref{eq:normedsquared:decompositionComplexStab} 
follows from \cref{prop:decomposition:decompositionComplexStab} 
applied to the \( 1 \)-dimensional subalgebra 
\( \LieA{t} \subseteq \LieA{g_m} \) spanned by \( J(m) \).

If \( J \) is \( \adAction \)-equivariant, then 
\( C^-_m = C^+_m + 2 \I \adAction_{J(m)} \).
Thus, on \( \ker C^+_m = (\LieA{g}_\C)_m \), the 
operator \( C^-_m \) acts as \( 2 \I \adAction_{J(m)} \).
In particular, the eigenspace \( \LieA{k}_\lambda \) is 
an eigenspace of \( C^-_m \) with eigenvalue \( 2 \lambda \).
By \cref{i:normedsquared:kernelsIdentified}, 
the eigenvalues of \( C^-_m \) are non-positive and the 
\( 0 \)-eigenspace is \( \ker C^+_m \intersect \ker C^-_m 
= (\LieA{g}_m)_\C \). This establishes the 
decomposition~\eqref{eq:normedsquared:decompositionComplexStabEquiv}.
Moreover, if \( \LieA{g}_m \) is a compact subalgebra of 
\( \LieA{g} \), then it is reductive and so is its 
complexification \( (\LieA{g}_m)_\C \).
For the sake of contradiction, assume that \( (\LieA{g}_m)_\C \) 
is not a maximal reductive algebra contained in \( (\LieA{g}_\C)_m \).
Then there exists a reductive complex subalgebra 
\( \LieA{a} \subseteq (\LieA{g}_\C)_m \) of \( \LieA{g}_\C \) 
that properly contains \( (\LieA{g}_m)_\C \).
Let \( \xi \in \LieA{a} \), but \(\xi \notin (\LieA{g}_m)_\C \), 
and decompose it relative 
to~\eqref{eq:normedsquared:decompositionComplexStabEquiv} as 
\( \xi = \sum_{\lambda \leq 0} \xi_\lambda \) with 
\( \xi_\lambda \in \LieA{k}_\lambda \).
By \cref{prop:normedsquared:criticalPoints}, \( J(m) \) is an 
element of \( \LieA{g}_m \subseteq \LieA{a} \) and thus 
\( \commutator{J(m)}{\xi} \) lies in \( \LieA{a} \) again.
Moreover, evaluating 
\begin{equation}
\AdAction_{\exp\bigl(t J(m)\bigr)} \xi 
= e^{\adAction_{tJ(m)}} \xi 
= \sum_{\lambda \leq 0} e^{\adAction_{tJ(m)}} \xi_\lambda
= \sum_{\lambda \leq 0} e^{- \I t \lambda} 
	\xi_\lambda \in \LieA{a}
\end{equation}
at conveniently chosen values of \( t \), 
we conclude that \( \xi_\lambda \in \LieA{a} \) for 
all \( \lambda \leq 0 \); see \cref{lemma_decompositionComplexStab}. 
Since \( \xi \) is not an element of 
\( (\LieA{g}_\C)_m = \LieA{k}_0 \), there exists 
a \( \lambda < 0 \) with \( \xi_\lambda \neq 0 \).
But then \( \I \commutator{J(m)}{\xi_\lambda} = 
\lambda \xi_\lambda \) implies that the complex 
subalgebra of \( \LieA{a} \) generated by \( J(m) \) 
and \( \xi_\lambda \) is solvable and non-Abelian, 
contracting the reductiveness of \( \LieA{a} \).
This completes the proof that \( (\LieA{g}_m)_\C \) 
is maximal reductive in \( (\LieA{g}_\C)_m \).
\end{proof}

\begin{lemma}\label{lemma_decompositionComplexStab}
In the notations of \cref{prop:normedsquared:decompositionComplexStab}
and its proof assume that 
\(\sum_{\lambda \leq 0} e^{- \I t \lambda} 
	\xi_\lambda \in \LieA{a}\)
for all \(t \in  \mathbb{R}\), where \( \xi_\lambda \in 
\mathfrak{k}_\lambda \). Then \( \xi_\lambda 
\in  \mathfrak{a} \) for all eigenvalues \( \lambda \) 
{\rm(}all are \( \leq 0\){\rm)}.
\end{lemma}

\begin{proof}
Let \(\lambda_0 \leq 0\) be the largest eigenvalue
occurring in the sum \( \sum_{\lambda \leq 0} \xi_\lambda \)
and denote the corresponding eigenvector by \(\xi_0\).
Denote the other eigenvalues occurring in this sum
by \( \lambda_1, \ldots,\lambda_N \). Divide 
\( \sum_{\lambda \leq 0} e^{- \I t \lambda}\xi_\lambda \) 
by \(e^{- \I t \lambda_0}\) and conclude that 
\[
\xi_0 + \sum_{k=1}^N e^{- \I t \mu_k}\xi_{\lambda_k} \in  \mathfrak{a},
\quad \text{for all} \quad t \in \mathbb{R}, \quad
\text{where} \quad \mu_k \defeq \lambda_k - \lambda_0. 
\]
For \(t=0\) we get
\[
\xi_0 + \sum_{k=1}^N \xi_k \in  \mathfrak{a} .
\]
Next, put \(t= \frac{\pi}{\mu_k}\) for all \(k=1, \ldots, N\)
and add the resulting relations to get
\[
N \xi_0 - \sum_{k=1}^N \xi_k + \sum_{k\neq l} 
e^{- \pi \I \frac{\mu_k}{\mu_l}}\xi_k
\in  \mathfrak{a}; 
\] 
the last summand is a double sum.
Now put \(t=\frac{\pi}{\mu_k} + \frac{\pi}{\mu_l}\) for all
\(k,l=1, \ldots , N\) and
take the double sum over all pairs \((k,l)\) satisfying
\(k<l\) to get
\[
N_2 \xi_0 - \sum_{k<l} 
\left( e^{- \pi \I  \frac{\mu_k}{\mu_l}}\xi_k +
e^{- \pi \I  \frac{\mu_l}{\mu_k}}\xi_l \right) +
\sum_{r\neq k,l;\, k<l}e^{- \pi \I \frac{\mu_r}{\mu_k}}
e^{- \pi \I \frac{\mu_r}{\mu_l}}\xi_r
\in  \mathfrak{a}; 
\]
the last summand is a triple sum and \(N_2\) denotes the number
of choices of pairs \((k,l)\) satisfying \(k<l\) as \(k,l=1,
\ldots, N\). Continue in this way by taking \(t=\frac{\pi}{\mu_k} 
+ \frac{\pi}{\mu_l}+\frac{\pi}{\mu_m}\) for \(k<l<m\) and summing all
resulting relations, etc. The general term, when taking \(t = 
\frac{\pi}{\mu_{k_1}} + \cdots + \frac{\pi}{\mu_{k_p}}\) for all
\(p\)-tuples \((k_1, \ldots, k_p)\) satisfying \(k_1< \cdots< k_p\)
and adding the resulting relations, is 
\begin{align*} 
N_p \xi_0 &- \sum_{\substack{l_1 \neq \cdots \neq l_{p-1}\\ k \neq l_q;\,
q=1, \ldots, p-1}}
e^{- ~\pi \I \frac{\mu_k}{\mu_{l_1}}}\cdots
e^{- \pi \I \frac{\mu_k}{\mu_{l_{p-1}}}}\xi_k
+
\sum_{\substack{l_1 <\cdots <l_p;\\ k \neq l_r;\,
r=1, \ldots, p}}
e^{- \pi \I \frac{\mu_k}{\mu_{l_1}}}\cdots
e^{- \pi \I \frac{\mu_k}{\mu_{l_p}}}\xi_k \in  \mathfrak{a}; 
\end{align*}
the second summand is over \(p\) indices, the third 
summand is over \(p+1\) indices, and \(N_p\) denotes the
number of choices of \(p\)-tuples \((k_1, \ldots, k_p)\) 
satisfying \(k_1< \cdots< k_p\) as \(k_1, \ldots, k_p = 1,
\ldots, N\). Adding all the displayed relations yields 
\(M\xi_0 \in  \mathfrak{a}\) for some \(M \in \mathbb{N},
\, M\geq 1\), whence \(\xi_0 \in  \mathfrak{a}\).

Hence in the hypothesis \(\sum_{\lambda \leq 0} e^{- \I t \lambda} 
\xi_\lambda \in \LieA{a}\) we can eliminate the term for
the largest \( \lambda \). Now repeat the procedure for the
resulting sum having one less term. Inductively we conclude that
each \( \xi_\lambda \in  \mathfrak{a} \).
\end{proof}

\begin{remark}
In the equivariant case,~\eqref{eq:normedsquared:decompositionComplexStabEquiv} 
implies that \( (\LieA{g}_\C)_m \) is reductive when 
\( m \) is a critical point such that \( J(m) \) lies 
in the center of \( \LieA{g} \) and such that 
\( \mathfrak{g}_m \) is a compact subalgebra of 
\( \LieA{g} \). Such a conclusion does not seem to 
hold in the non-equivariant case. Indeed, the proof of 
\cref{prop:normedsquared:decompositionComplexStab} shows 
that \( \LieA{c}_m \) is the joint kernel of \( C^+_m \) 
and \( C^-_m + 2 \I (Z_m + \adAction_{J(m)}) \).
Thus,  \( \LieA{c}_m \) is invariant under complex 
conjugation only if \( Z_m + \adAction_{J(m)} \) vanishes.
Hence, it is unlikely that \( \LieA{c}_m \) is reductive 
(it is certainly not always a complexification of a compact algebra).
In fact, the example of the Galilean group 
(\cref{ex:momentumMapSquared:galileanGroup} below) shows 
that \( (\LieA{g}_C)_m \) and \( \LieA{c}_m \) may not even 
be Lie algebras in the non-equivariant case. 
\end{remark}

\begin{example}[Heisenberg group]
We continue the example of the Heisenberg group; see 
\cref{ex:affineSymplectic:heisenbergGroup}.
Let \( (V, \omega) \) be a symplectic vector space 
and consider the affine action of \( V \) on itself 
by translation. Moreover, assume \( V \) is 
endowed with a constant complex structure \( j \) 
compatible with \( \omega \) and denote the associated 
Riemannian metric by \( g \). So, in this case, we choose,
$\kappa =g$. Therefore, the	norm-squared 
of the momentum map is \( \norm{J}^2(v) = g(v, v) \) and the
only critical point is \( 0 \). A direct calculation shows 
that \( L_m = - \id_V \) and \( \Sigma_\kappa = j \). 
The stabilizer of \( 0 \) under the complexified action is
\begin{equation}
(V_\C)_{0} = \set{\xi + \I \, j \xi \given \xi \in V}.
\end{equation}
This is clearly the kernel of \( C^+_m = - \id_V + \I j \), 
in agreement with \cref{i:normedsquared:kernelsIdentified}. 
Note that the decomposition~\eqref{eq:normedsquared:decompositionComplexStab} of this 
stabilizer collapses to the first summand since \( V \) 
is Abelian.
\end{example}

\begin{example}[Galilean group]
	\label{ex:momentumMapSquared:galileanGroup}
We continue the example of the Galilean group discussed in 
\cref{ex:affineSymplectic:galileanGroup}.
The norm-squared of the momentum map 
(\cref{eq:affineSymplectic:galileanGroup:momentumMap}) with 
respect to the pairing in 
\cref{eq:affineSymplectic:galileanGroup:pairing} is
\begin{equation}
\norm{J}^2_\kappa\bigl(\vec{q}, \vec{p}, \vec{x}\bigr) 
= \frac{1}{2} \norm{\vec{q} \times \vec{p} - s \vec{x}}^2 
+ m^2 \norm{\vec{q}}^2 + \norm{\vec{p}}^2 
+ \frac{1}{4m^2} \norm{\vec{p}}^4.
\end{equation}
The critical points are given by the solutions of the equations
\begin{equation}
\label{equ_critical_points_gal}
\begin{aligned}
0 &= \vec{p} \times (\vec{q} \times \vec{p} 
- s \vec{x}) + 2 m^2 \vec{q}, \\
0 &= (\vec{q} \times \vec{p} - s \vec{x}) \times \vec{q} 
+ 2 \vec{p} + \frac{1}{m^2} \norm{\vec{p}}^2 \vec{p}, \\
0 &= s (\vec{q} \times \vec{p} - s \vec{x}) \times \vec{x}.
\end{aligned}\end{equation}
This system  has been obtained by imposing the condition
\( J(m) \in \mathfrak{g}_m\) for a critical point \(m\) in 
\cref{prop:normedsquared:criticalPoints}.
Of course, the calculations can be done directly, but then the 
third equation becomes \(s\vec{v} \cdot  
(\vec{q} \times \vec{p} - s \vec{x}) = 0\) for all \(\vec{v}\perp 
\vec{x}\) which again implies that \(\vec{q} \times \vec{p} - s \vec{x}\)
is parallel to \(\vec{x}\).

Clearly, \(\vec{q} =\vec{0}\), \(\vec{p} = \vec{0}\), and 
\(\vec{x} \in S^2\) arbitrary, is a solution of this system.
In fact, it is the only solution with \( \vec{q} \times \vec{p} = 0 \).
This is the first set of critical points.

The second set of critical points is obtained by assuming that
 \(\vec{q} \times \vec{p} \neq \vec{0}\). The
last equation implies that \( \vec{x} \) 
is a multiple of \( \vec{q} \times \vec{p} \). 
So, let \(\vec{x} = c \vec{q} \times \vec{p}\) for some \( 0 \neq c \in  
\mathbb{R}\). The first two equations in~\eqref{equ_critical_points_gal} imply then
\(\vec{q}\cdot\vec{p} = 0\),
\(\vec{q}\cdot\vec{x} = 0\), \(\vec{p}\cdot\vec{x} = 0\), and
\begin{align*}
\|\vec{p}\|^2 = \frac{2m^2}{cs-1} >0, \qquad  
\|\vec{q}\|^2 = \frac{2cs}{(cs-1)^2} >0.  
\end{align*}
So, we must have \(c>1/s >0\). 
The factor $c$ is determined from the condition
\(\vec{x} = c \vec{q} \times \vec{p}\) which, together
with \(\vec{q} \cdot \vec{p} = 0\), gives
\[
\begin{array}{rccccl}
c&=& \displaystyle{\frac{1}{s-(4m^2s)^{1/3}}\,,} &
\qquad cs-1 &=&\displaystyle{\frac{(4m^2s)^{1/3}}{s-(4m^2s)^{1/3}}\,,} \\ \\
\|\vec{p}\|^2 &=& \displaystyle{
\frac{2m^2\left( s-(4m^2s)^{1/3} \right)}{(4m^2s)^{1/3}}\,,}  
&\qquad \|\vec{q}\|^2 &=& \displaystyle{
\frac{2s\left( s-(4m^2s)^{1/3} \right)}{(4m^2s)^{2/3}}\,.}
\end{array} 
\] 
Taking the cross product of the second
equation in~\eqref{equ_critical_points_gal}
with \( \vec{x} \) and recalling that \( \norm{\vec{x}}=1\), 
\(\vec{q} \cdot \vec{x} =0\) yields
\((2m^2 + \|\vec{p}\|^2) \vec{p} \times \vec{x} 
= m^2(4m^2s)^{1/3}\vec{q}\).
Since $2m^2 + \|\vec{p}\|^2 = 2m^2s/(4m^2s)^{1/3}$, we get
the second set of critical points:
\begin{align*}
\vec{q} = \frac{(4m^2s)^{1/3}}{2m^2} \vec{p} \times \vec{x},
\quad \norm{\vec{p}}^2 = \displaystyle{
\frac{2m^2\left( s-(4m^2s)^{1/3} \right)}{(4m^2s)^{1/3}}\,,}  
\quad \vec{p}\cdot\vec{x}=0, \quad \|\vec{x}\| = 1.
\end{align*}

The images of these critical points under the momentum 
map~\eqref{eq:affineSymplectic:galileanGroup:momentumMap} are
\begin{equation}\begin{aligned}
J(0, 0, x) &= \left(-\frac{s}{2} \vec{x}, 0, 0, 0\right),\\
J(\vec{q}, \vec{p}, \vec{x}) &= \left(- \sqrt[3]{\frac{m^2 s}{2}}\vec{x},\,
-\sqrt[3]{\frac{s}{2m}} \vec{p} \times \vec{x},\, \vec{p},\, 
- \frac{m\left( s-(4m^2s)^{1/3} \right)}{(4m^2s)^{1/3}}
\right),
\end{aligned}\end{equation}
	respectively.

On \( \R^3 \times  \R^3 \times S^2 \) we introduce the complex structure
\begin{equation}
j_{(\vec{q}, \vec{p}, \vec{x})}\bigl(\diF\vec{q}, \diF\vec{p}, 
\diF\vec{x}\bigr) = 
\bigl(\diF\vec{p}, -\diF\vec{q}, \vec{x} \times \diF\vec{x}\bigr).
\end{equation}
A direct calculation shows that this complex structure has 
non-equivariance cocycle, \cf \cref{group_cocyle_in_terms_of_j},
\begin{equation}
	(R, \vec{v}, \vec{a}, \tau) \mapsto \frac{\tau}{m} \bigl(\id,\id,0\bigr)
	\in  \EndBundle(\TBundle(\R^3 \times  \R^3 \times S^2)).
\end{equation}
Hence, our standing assumption that the complex structure 
is \( J(m) \)-invariant only holds at the critical point 
\( m=(0, 0, \vec{x}) \).
Moreover, a direct calculation shows that (see~\eqref{Lie_bracket_on_gal})
\begin{equation}
\CoadAction_{J(0,0,\vec{x})} \bigl(\vec{\alpha}, \vec{\beta}, 
\vec{\gamma}, \delta\bigr) 
= \frac{s}{2} \left(\vec{x} \times \vec{\alpha}, 
 \vec{x} \times \vec{\beta}, 
 \vec{x} \times \vec{\gamma}, 0\right) 
= -\adAction_{J(0,0,\vec{x})} \bigl(\vec{\alpha}, \vec{\beta}, 
\vec{\gamma}, \delta\bigr),
\end{equation}
which implies that \( \kappa \) is invariant under \( J(0,0,\vec{x}) \).
Hence, all assumptions of 
\cref{prop:normedsquared:decompositionComplexStab} are satisfied 
at the critical point \( (0, 0, \vec{x}) \). In this case, the 
decomposition~\eqref{eq:normedsquared:decompositionComplexStab} 
takes the form
\begin{equation}
(\LieA{gal}_\C)_{(0,0,\vec{x})} = 
\LieA{c}_{(0,0,\vec{x})} \oplus \LieA{k}_{\frac{s}{2}} \oplus 
\LieA{k}_{-\frac{s}{2}},
\end{equation}
where the stabilizer \( (\LieA{gal}_\C)_{(0,0,\vec{x})} \) is 
\( 12 \)-dimensional and consists of points of the form
\begin{equation}
\bigl(\vec{\alpha} \times \vec{x} + \I \vec{\alpha} + a\vec{x}, 
\vec{\beta}, \I m \vec{\beta}, \theta\bigr), \quad 
\vec{\alpha} \in \R^3, \;\;
\vec{x} \cdot \vec{\alpha} = 0, \;\; a \in \C,\;\; \vec{\beta} \in \R^3,
\;\; \theta \in \C.
\end{equation}
The summands are given by\footnotemark{}
\footnotetext{Let \( \norm{\vec{x}} = 1 \). Note 
that the eigenvalue equation \( i\vec{x} \times 
\varepsilon = \lambda \varepsilon \) for \( \varepsilon \in \C^3 \) 
has solutions \( \varepsilon =  \vec{e} \times \vec{x} \pm \I \vec{e}\) 
with \( \vec{e} \in \R^3 \) and \( \vec{x} \cdot \vec{e} = 0 \) 
for \( \lambda = \pm 1 \), and \( \varepsilon = e \vec{x} \) with 
\( e \in \C \) for \( \lambda = 0 \).}
\begin{equation}\begin{aligned}
\LieA{c}_{(0,0,\vec{x})} &= \set*{(a \vec{x}, b \vec{x}, 
\I m b \vec{x}, \theta) \given a, b, \theta \in \C},\\
\LieA{k}_{\frac{s}{2}} &= \set*{\bigl(\vec{\alpha} \times \vec{x} + 
\I \vec{\alpha}, \vec{\beta} \times \vec{x} + \I \vec{\beta}, 
\I m (\vec{\beta} \times \vec{x} 
+ \I \vec{\beta}), 0\bigr) \given \vec{\alpha}, \vec{\beta} \in \R^3, 
\vec{x} \cdot \vec{\alpha} = 0 = \vec{x} \cdot \vec{\beta}},\\
\LieA{k}_{-\frac{s}{2}} &= \set*{\bigl(0, \vec{\beta} \times \vec{x} 
- \I \vec{\beta}, \I m (\vec{\beta} \times \vec{x} - \I \vec{\beta}), 
0\bigr) \given \vec{\beta} \in \R^3, \vec{x} \cdot \vec{\beta} = 0}.
\end{aligned}\end{equation}
Note that (the complexification of) the real stabilizer subalgebra 
\( \LieA{gal}_{(0,0,\vec{x})} = \set*{\bigl(a\vec{x},0,0,\theta\bigr) 
\given a, \theta \in \R} \) is contained in 
\( \LieA{c}_{(0,0,\vec{x})} \), but we no longer have 
\( \bigl(\LieA{gal}_{(0,0,\vec{x})}\bigr)_\C = 
\LieA{c}_{(0,0,\vec{x})} \) as in the equivariant case.
Moreover, by~\eqref{Lie_bracket_on_gal}, we have 
\begin{equation}
\commutator*{\bigl(a_1 \vec{x}, b_1 \vec{x}, \I m b_1 \vec{x}, \theta_1\bigr)}
{\bigl(a_2 \vec{x}, b_2 \vec{x}, \I m b_2 \vec{x}, \theta_2\bigr)}
= \bigl(0, 0, b_1 \theta_2 \vec{x} - b_2 \theta_1 \vec{x}, 0\bigr).
\end{equation}
Hence \( \LieA{c}_{(0,0,\vec{x})} \) and 
\( (\LieA{gal}_\C)_{(0,0,\vec{x})} \) are not even Lie 
subalgebras of \( \LieA{gal}_\C \). 
\end{example}
	
\begin{example}[Virasoro group]
Continuing \cref{ex:affineSymplectic:virasoroGroup} of
the Virasoro group, the norm-squared of the momentum map is 
by~\eqref{eq:affineSymplectic:virasoroGroup:momentumMap}
\begin{equation}
\norm{\SectionMapAbb{J}}^2_\kappa\bigl(\equivClass{f}\bigr) = 
\int_{S^1}\Bigl(f'' - \frac{1}{2} (f')^2\Bigr)^2 \dif \varphi
\end{equation}
with respect to the pairing
\begin{equation}
\kappa\bigl(X \difp_\varphi, Y \difp_\varphi\bigr) = 
\int_{S^1} X Y \dif \varphi, \qquad X,Y \in \sFunctionSpace(S^1).
\end{equation}
The critical points \( \equivClass{f} \) are solutions of 
the equation
\begin{equation}
f''' - \frac{1}{2} (f')^3 = c,
\end{equation}
for some constant \( c \in \R \). For \( c = 0 \), 
the solutions can be given explicitly in terms of 
the Jacobi elliptic sine function. In this setting, the 
formula~\eqref{eq:normedsquared:hessianOrbit} for 
the Hessian can be checked through a straightforward 
but lengthy calculation using integration by parts\footnote{In 
fact, the calculation is so tedious that we used the open-source 
computer algebra system SageMath.}.
Moreover, the Hilbert transform yields an almost complex 
structure compatible with the symplectic form, see 
\parencite{Pressley1982,AtiyahPressley1983,BlochMorrisonRatiu2012}.
However, \cref{prop:normedsquared:decompositionComplexStab} 
does not apply since neither the pairing \( \kappa \) nor the 
Hilbert transform are equivariant with respect to the natural 
actions of the diffeomorphism group. Indeed, we have
\begin{equation}
\kappa\bigl(\adAction_{Z \difp_\varphi} (X \difp_\varphi), 
(Y \difp_\varphi)\bigr)
+ \kappa\bigl(X \difp_\varphi, \adAction_{Z \difp_\varphi} 
(Y \difp_\varphi)\bigr)
= 3 \int_{S^1} Z' XY  \dif \varphi.
\end{equation}
Thus, \( \kappa \) is only invariant under rotations, \( Z' = 0 \).
The rotation \( S^1 \)-action has the energy functional
\begin{equation}
\SectionMapAbb{J}_{S^1}\bigl(\equivClass{f}\bigr) = 
\frac{1}{2}\int_{S^1}{(f')}^2 \dif \varphi
\end{equation}
as its momentum map, and the Hessian of 
\( \SectionMapAbb{J}_{S^1} \) has been investigated 
in detail in \parencite{Pressley1982}.
\end{example}

Our next aim is to calculate the Hessian of \( \norm{J}^2_\kappa \) 
at a critical point. To avoid serious pathological behavior in 
infinite dimensions, we \emph{assume that \( G \) has a smooth 
exponential map}. Moreover, we also \emph{assume that the stabilizer 
\( G_m \) of every point \( m \in M \) is a Lie subgroup of \( G \).} 
Since \(G_m\) is a closed subset of \(G\), this means that 
\( G_m \) is a submanifold, not just injectively immersed. This is, 
for example, the case when \( G_m \) is compact (\eg, the action 
is proper) and \( G \) is locally exponential 
\parencite[Theorem~7.3.14]{GloecknerNeeb2013}. 
Under these assumptions, \( \sigma \in \LieA{g} \) is an 
element of \( \LieA{g}_m \) if and only if 
\( \exp(t \sigma) \cdot m = m \) for all \( t \in [0,1] \), 
see \parencite[Proposition~II.6.3]{Neeb2006}. 
Moreover, \( m \) is in the vanishing locus of the fundamental 
vector field \( \sigma^*: M \to \TBundle M, \sigma^*(m) 
\defeq \sigma \ldot m \) for every \( \sigma \in \LieA{g}_m \).
The \emphDef{linearization} \( \tau_m\sigma^*: \TBundle_m M \to 
\TBundle_m M \) of \( \sigma^* \) at the point \( m \) is the 
linear operator defined by
\begin{equation}
\label{equ_linearization_fundamental_vector_field}
\tau_m \sigma^* (X) \defeq 
\difFracAt{}{t}{t=0} \tangent_m 
\Upsilon_{\exp(t \sigma)} (X), \qquad X \in \TBundle_m M,
\end{equation}
where \( \Upsilon_g: M \to M \) denotes the action of \( g \in G \).
As \( \Upsilon_{\exp(t \sigma)}(m) = m \) for 
\( \sigma \in \LieA{g}_m \) and  \( t \in [0,1] \), the assignment 
\( t \mapsto \tangent_m \Upsilon_{\exp(t \sigma)} (X) \) defines a 
curve in \( \TBundle_m M \), which shows that \( \tau_m \sigma^* \) 
takes indeed values in \( \TBundle_m M \). In finite dimensions, 
the resulting representation of \( \LieA{g}_m \) 
on \( \TBundle_m M \), called the \textit{isotropy representation}, 
is Hamiltonian with momentum map
\begin{equation}
\label{j_hat_momentum}
\hat{J}: \TBundle_m M \to \LieA{g}_m, \qquad 
\kappa\bigl(\hat{J}(X), \sigma\bigr) = 
\frac{1}{2} \omega_m (X, \tau_m \sigma^* (X)).
\end{equation}
In infinite dimensions, we have to assume that the 
functional on \( \LieA{g}_m \) defined by the right-hand 
side can indeed be represented by an element 
\( \hat{J}(X) \in \LieA{g}_m \). 

Finally, recall that the Hessian 
\( \Hessian_m f: \TBundle_m M \to \R \) of 
a function \( f: M \to \R \) at a critical point 
\( m \in M \) is the intrinsically defined quadratic 
form given by
\begin{equation}
\label{Hessian_definition}
\Hessian_m f \, (X) = \difFracAt[2]{}{t}{t=0} f(\gamma(t)),
\end{equation}
where \( \gamma: (-\varepsilon, \varepsilon) \to M \), 
\(\varepsilon >0\), is a smooth map with \( \gamma(0) = m \) 
and \( \dot{\gamma}(0) = X \in \TBundle_m M\). Since \( m \) 
is a critical point of \( f \), the right-hand side does 
not depend on the chosen curve \( \gamma \).

With this preparation, we can state the first important 
result of this section.
\begin{prop}
\label{prop:normedsquared:hessian}
Let \( (M, \omega) \) be a connected symplectic Fr\'echet 
manifold endowed with a symplectic action of a Fr\'echet 
Lie group \( G \). Assume that the action has a momentum 
map \( J: M \to \LieA{g} \) relative 
to a non-degenerate symmetric pairing 
\( \kappa: \LieA{g} \times \LieA{g} \to \R \).
Let \( m \in M \) be a critical point of \( \norm{J}^2_\kappa \).
In infinite dimensions, assume the following conditions {\rm(}which
always hold in finite dimensions{\rm)}:
\begin{enumerate}
\item
\( G \) has a smooth exponential map.
\item
\( G_m \) is a Lie subgroup of \( G \).
\item
The isotropy representation of \( \LieA{g}_m \) on \( \TBundle_m M \) 
has a momentum map \( \hat{J}: \TBundle_m M \to \LieA{g}_m \). 
\end{enumerate}
Then the Hessian of \( \norm{J}^2_\kappa \) at \( m \) is given by
\begin{equation}
\label{eq:normedsquared:hessian}
\frac{1}{2} \, \Hessian_m \norm{J}^2_\kappa \, (X) = 
\norm{\TBundle_m J(X)}^2_\kappa + 
2 \, \kappa\bigl(\hat{J}(X), J(m)\bigr)
\end{equation}
and the associated symmetric bilinear form is given by
\begin{equation}
\frac{1}{2} \, \Hessian_m \norm{J}^2_\kappa \, (X, Y) 
=	\kappa\bigl(\tangent_m J(X), \tangent_m J(Y)\bigr) 
+ \omega_m\bigl(X, \tau_m \bigl(J(m)^*\bigr) \, Y\bigr).
\qedhere
\end{equation}
\end{prop}

\begin{proof}
Let \( X \in \TBundle_m M \) and let 
\( \gamma: (-\varepsilon, \varepsilon) \to M \) be 
a smooth map with \( \gamma(0) = m \) and \( \dot{\gamma}(0) = X \). 
Using the definition of a momentum map, we obtain
\begin{equation}\begin{split}
\frac{1}{2} \, \Hessian_m \norm{J}^2_\kappa \, (X)
&= \frac{1}{2} \, \difFracAt[2]{}{t}{t=0}\kappa\Bigl(J\bigl(\gamma(t)\bigr), 
J\bigl(\gamma(t)\bigr)\Bigr)
			\\
&= \difFracAt{}{t}{t=0}
 \kappa\Bigl(\tangent_{\gamma(t)} 
J \bigl(\dot{\gamma}(t)\bigr), J\bigl(\gamma(t)\bigr)\Bigr)
			\\
&= \difFracAt{}{t}{t=0}
 \omega_{\gamma(t)}\Bigl(\dot{\gamma}(t), 
J\bigl(\gamma(t)\bigr) \ldot \gamma(t)\Bigr).
\end{split}\end{equation}
Using \( \dif \omega = 0 \), \( J(m) \ldot m = 0 \) 
(see
\cref{prop:normedsquared:criticalPoints}), and the identity
\begin{equation}
\difFracAt{}{t}{t=0}
 J\bigl(\gamma(t)\bigr) \ldot \gamma(t)
= \bigl(\tangent_m J (X)\bigr) \ldot m + \tau_m \bigl(J(m)^*\bigr) \, X
\end{equation}
(for the second summand use the definition of the
fundamental vector field and~\eqref{equ_linearization_fundamental_vector_field})
we continue
\begin{equation}\begin{split}
\frac{1}{2} \, \Hessian_m \norm{J}^2_\kappa \, (X)
&= \omega_m \Bigl(X, \bigl(\tangent_m J (X)\bigr) \ldot m\Bigr) + 
\omega_m \Bigl(X, \tau_m \bigl(J(m)^*\bigr) \, X\Bigr)
	\\
&= \kappa\Bigl(\tangent_m J(X), \tangent_m J(X)\Bigr) + 
2 \, \kappa\Bigl(\hat{J}(X), J(m)\Bigr)
\end{split}
\end{equation}
by~\eqref{j_hat_momentum}. This
proves~\eqref{eq:normedsquared:hessian}.
The associated bilinear form is given by the polarization identity
\begin{equation}\begin{split}
\frac{1}{2} \, \Hessian_m \norm{J}^2_\kappa \, (X, Y)
&=\frac{1}{4} \Bigl(\Hessian_m \norm{J}^2_\kappa \, (X + Y) \\
&\qquad- \Hessian_m \norm{J}^2_\kappa \, (X)
- \Hessian_m \norm{J}^2_\kappa \, (Y)\Bigr)\\
&=\kappa\bigl(\tangent_m J(X), \tangent_m J(Y)\bigr) 
+ \omega_m\bigl(X, \tau_m \bigl(J(m)^*\bigr) \, Y\bigr).
	\qedhere
\end{split}\end{equation}
\end{proof}

\begin{remark}
A different, and perhaps more conceptional, proof of the 
identity~\eqref{eq:normedsquared:hessian} for the Hessian can be obtained 
by working with the Marle--Guillemin--Sternberg normal form of \( J \).
Note, however, that in infinite dimensions the construction of a normal 
form for the momentum map is a complex endeavor (see
\parencite{DiezThesis,DiezSingularReduction}) which is why we preferred the
direct proof given above.
\end{remark}

In the presence of an almost complex structure on \( M \), the Hessian 
of \( \norm{J}^2_\kappa \) in the direction of the complexified orbit 
is of particular interest.
\begin{lemma}
\label{prop:normedsquared:hessianExpanded}
In the setting of \cref{prop:normedsquared:hessian}, let \( j \) be an almost 
complex structure on \( M \) compatible with \( \omega \).
Let \(\Sigma : \LieA{g} \times  \LieA{g} 
\to  \R \) be the non-equivariance 2-cocycle of the 
momentum map \(J: M \to \LieA{g} \) given by~\eqref{eq:normedsquared:nonequiv}.
Assume that the adjoint \( \CoadAction_\xi: \LieA{g} \to \LieA{g} \) 
of \( \adAction_\xi: \LieA{g} \to \LieA{g} \) and the map \( \Sigma_\kappa: 
\LieA{g} \to \LieA{g} \) defined by \( \kappa\bigl(\Sigma_\kappa(\xi), \eta\bigr)
\defeq \Sigma(\xi, \eta) \) exist {\rm(}this is automatic in finite dimensions{\rm)}.
For every critical point \( m \in M \) of \( \norm{J}^2_\kappa \),
\begin{subequations}
	\label{eq:normedsquared:hessianExpanded}
	\begin{align}
\begin{split}
\label{eq:normedsquared:hessianOrbit}
\frac{1}{2} \, \Hessian_m \norm{J}^2_\kappa \, 
&\bigl(\xi \ldot m, \eta \ldot m\bigr)\\
&= \kappa\bigl(\Sigma_\kappa(\xi), \Sigma_\kappa(\eta)\bigr) - 
\kappa\bigl(\Sigma_\kappa(\xi), \CoadAction_\eta J(m)\bigr)
			\\
&\quad- \kappa\bigl(\CoadAction_\xi J(m), \Sigma_\kappa(\eta)\bigr)
+ \kappa\bigl(\CoadAction_\xi J(m), \CoadAction_\eta J(m)\bigr)
			\\
&\quad+ \kappa\bigl(\Sigma_\kappa \adAction_\eta J(m), \xi\bigr) - 
\kappa\bigl(\CoadAction_{\commutator*{\eta}{J(m)}} J(m),\xi \bigr)
\end{split}\\
\begin{split}
			\label{eq:normedsquared:hessianOrbitMixed}
\frac{1}{2} \, \Hessian_m \norm{J}^2_\kappa \, 
&\bigl(\xi \ldot m, j \, (\eta \ldot m)\bigr)\\
&= \kappa\bigl(\Sigma_\kappa(\xi), \
\tangent_m J(j \, (\eta \ldot m))\bigr) \\
&\quad- \kappa\bigl(\CoadAction_\xi J(m), 
\tangent_m J(j \, (\eta \ldot m))\bigr)
			\\
&\quad+\kappa\bigl(\CoadAction_{J(m)} 
\tangent_m J(j \, (\eta \ldot m)),\xi \bigr)
\end{split}\\
		\intertext{and}
\begin{split}
			\label{eq:normedsquared:hessianComplex}
\frac{1}{2} \, \Hessian_m \norm{J}^2_\kappa \, 
&\bigl(j \, (\xi \ldot m), j \, (\eta \ldot m)\bigr) \\
&= \kappa\bigl(\tangent_m J\bigl(j \, (\xi \ldot m)\bigr), 
\tangent_m J\bigl(j \, (\eta \ldot m)\bigr)\bigr)
			\\ &\qquad+  
\Sigma(\xi, \commutator*{J(m)}{\eta}) - 
\kappa(J(m), \commutator{\xi}{\commutator*{J(m)}{\eta}})
				\\ 
&\qquad+ \kappa\Bigl(\tangent_m J(j \tau_j'(J(m)) 
(\eta \ldot m)), \xi\Bigr)
\end{split}
	\end{align}
\end{subequations}
for all \( \xi \in \LieA{g} \) and \( \eta \in \LieA{g} \).
\end{lemma}

\begin{proof}
For every \( \sigma \in \LieA{g}_m \) and \( \xi \in \LieA{g} \), 
using the identities \( \Upsilon_{\exp(-t \sigma)}(m) = m \),  
\( \Upsilon_g^{\,*} \xi^* = (\AdAction_{g^{-1}}\xi)^*\),
the linearity of the operation \(\LieA{g} \ni \eta \mapsto
\eta^* \in  \VectorFieldSpace(M)\), 
and~\eqref{equ_linearization_fundamental_vector_field}, we get
\begin{equation}\label{eq:normedsquared:tauOnAction}\begin{split}
\tau_m\sigma^*(\xi \ldot m) &=
\difFracAt{}{t}{t=0} \tangent_m 
\Upsilon_{\exp(t \sigma)}\bigl(\xi^*(m)\bigr) \\
&= \difFracAt{}{t}{t=0} \tangent_m 
\Upsilon_{\exp(t \sigma)}\Bigl(\xi^*
\bigl(\Upsilon_{\exp(-t \sigma)}(m)\bigr)\Bigr)  \\
&=\difFracAt{}{t}{t=0} 
\Upsilon_{\exp(-t \sigma)}^{\,*}(\xi^*)(m) \\
&=\difFracAt{}{t}{t=0} 
\left( \AdAction_{\exp(t \sigma)} \xi \right)^*(m) \\
&= \commutator{\sigma}{\xi} \ldot m \, .
\end{split}\end{equation}
Therefore, using~\eqref{eq:normedsquared:nonequiv}, we obtain
	\begin{equation}\label{eq:normedsquared:omegatau}\begin{split}
		\omega_m \bigl(\xi \ldot m, \tau_m \sigma^* (\eta \ldot m)\bigr)
			&= \omega_m \bigl(\xi \ldot m, \commutator{\sigma}{\eta} \ldot m\bigr) \\
			&= \Sigma\bigl(\xi, \commutator{\sigma}{\eta}\bigr) - 
			\kappa\bigl(J(m), \commutator{\xi}{\commutator{\sigma}{\eta}}\bigr) \\
			&= \Sigma\bigl(\adAction_\eta \sigma, \xi\bigr) - 
			\kappa\bigl(J(m), \adAction_{\commutator*{\eta}{\sigma}}\xi \bigr)\\
			&= \kappa\bigl(\Sigma_\kappa \adAction_\eta \sigma, \xi\bigr) - 
			\kappa\bigl(\CoadAction_{\commutator*{\eta}{\sigma}} J(m),\xi \bigr).
	\end{split}\end{equation}
Moreover, again by~\eqref{eq:normedsquared:nonequiv}, we find
	\begin{equation}
		\label{eq:normedsquared:nonequivinf}
		\tangent_m J(\xi \ldot m) = \Sigma_\kappa(\xi) - \CoadAction_\xi J(m) .
	\end{equation}
Inserting the expressions~\eqref{eq:normedsquared:omegatau} 
and~\eqref{eq:normedsquared:nonequivinf} into~\eqref{eq:normedsquared:hessian} 
yields~\eqref{eq:normedsquared:hessianOrbit}.

Furthermore, using~\eqref{eq:normedsquared:tauOnAction} 
and~\eqref{tangent_momentum}, we obtain
\begin{equation}\begin{split}
\omega_m\Bigl(\xi \ldot m, \tau_m \bigl(J(m)^*\bigr) j (\eta \ldot m)\Bigr)
&= - \omega_m\Bigl(\tau_m \bigl(J(m)^*\bigr) \xi \ldot m,  j (\eta \ldot m)\Bigr)\\
&= - \omega_m\Bigl(\commutator*{J(m)}{\xi} \ldot m, j (\eta \ldot m)\Bigr)\\
&= \kappa\Bigl(\tangent_m J(j (\eta \ldot m)), \commutator*{J(m)}{\xi}\Bigr)\\
&= \kappa\Bigl(\CoadAction_{J(m)}\tangent_m J(j (\eta \ldot m)), \xi\Bigr).
\end{split}\end{equation}
Using this equation and~\eqref{eq:normedsquared:nonequivinf} 
in~\eqref{eq:normedsquared:hessian} yields~\eqref{eq:normedsquared:hessianOrbitMixed}.

By the definition~\eqref{group_cocyle_in_terms_of_j} 
of \( \tau_j \), we have
	\begin{equation}
		\label{eq:normedsquared:commutatorActionj}
		\tangent_{g^{-1} \cdot m} \Upsilon_{g} \circ j_{g^{-1} \cdot m}
			= j_m \circ \tangent_{g^{-1} \cdot m} \Upsilon_g + 
			\tau_j(g) \circ \tangent_{g^{-1} \cdot m} \Upsilon_g
	\end{equation}
for all \( g \in G \). Using this identity for \( g = \exp(t \sigma) \), 
where \( \sigma \in \LieA{g}_m \), and differentiating in \( t \) 
yields \( \tau_m \sigma^* \circ j_m = j_m \circ \tau_m \sigma^* + 
\tau_j'(\sigma) \). Thus, again suppressing the dependency of \( j_m \) 
on \( m \), we get
	\begin{equation}\begin{split}
		\omega_m \bigl(j X, \tau_m \sigma^* (j Y)\bigr)
		& = \omega_m \bigl(j X, j \, \tau_m \sigma^*(Y) + 
		\tau_j'(\sigma) \, Y \bigr) \\
		& = \omega_m \bigl(X, \tau_m \sigma^*(Y)\bigr) + 
		\omega_m \bigl(j X, \tau_j'(\sigma) \, Y \bigr).
	\end{split}\end{equation}
Summarizing, setting \(X=\xi \ldot m\), \( Y = \eta \ldot m \), 
\(\sigma = J(m)\), and using~\eqref{eq:normedsquared:tauOnAction},
~\eqref{eq:normedsquared:nonequiv}, 
and~\eqref{tangent_momentum}, we find
\begin{equation}\begin{split}
&\omega_m(j \, (\xi \ldot m), \tau_m (J(m)^*) j \, (\eta \ldot m))\\
&\; = \omega_m \bigl(\xi \ldot m, \tau_m J(m)^*(\eta \ldot m)\bigr) + 
\omega_m \bigl(j \xi \ldot m, \tau_j'(J(m)) \, \eta \ldot m \bigr)\\
&\; = \omega_m \bigl(\xi \ldot m, \commutator*{J(m)}{\eta} \ldot m\bigr) - 
\omega_m \bigl(\xi \ldot m, j \tau_j'(J(m)) \, \eta \ldot m \bigr)\\
&\; = \Sigma(\xi, \commutator*{J(m)}{\eta}) - \kappa(J(m), 
\commutator{\xi}{\commutator*{J(m)}{\eta}})
+ \kappa\Bigl(\tangent_m J(j \tau_j'(J(m)) (\eta \ldot m)), \xi\Bigr).
\end{split}\end{equation}
Inserting this expression into~\eqref{eq:normedsquared:hessian} completes the 
proof of~\eqref{eq:normedsquared:hessianComplex}.
\end{proof}

The Hessian~\eqref{eq:normedsquared:hessianExpanded} can be 
expressed in terms of the Lichnerowicz operator introduced 
in~\eqref{L_m_L_mu} and in terms of the Calabi 
operator~\eqref{eq:decomposition:calabiOperatorDef}.
For this purpose, recall from~\eqref{T_operator} the 
representation of real-linear operators on \( \LieA{g}_\C \) 
in terms of matrices.

\begin{prop}
	\label{prop:normedsquared:propertiesL:Hessian}
	\label{prop:normedsquared:hessianSummary}
In the setting of \cref{prop:normedsquared:hessianExpanded}, 
at a critical point \( m \) of \( \norm{J}^2_\kappa  \), we have
\begin{subequations}
\label{eq:normedsquared:hessianSummary:usingL}
\begin{align}
\begin{split}
\frac{1}{2} \, \Hessian_m \norm{J}^2_\kappa \, 
\bigl(\xi \ldot m, \eta \ldot m\bigr)
&= - \kappa\bigl(\xi, Z_m(\adAction_{J(m)} + Z_m) \, \eta\bigr).
\end{split}\\
\begin{split}
\frac{1}{2} \, \Hessian_m \norm{J}^2_\kappa \, 
\bigl(\xi \ldot m, j \, (\eta \ldot m)\bigr)
&= \kappa\bigl(\xi, (\adAction_{J(m)}^* - Z_m) L_m\eta\bigr)
\end{split}
\intertext{and}
\begin{split}
\frac{1}{2} \, \Hessian_m \norm{J}^2_\kappa \,
\bigl(j \, (\xi \ldot m), j \, (\eta \ldot m)\bigr)
&= \kappa\bigl(\xi, L_m^2 - Z_m \adAction_{J(m)} \eta\bigr)\\
&\qquad + \kappa\Bigl(\tangent_m J\bigl(j \, 
\tau_j'(J(m)) \, \eta \ldot m\bigr), \xi\Bigr)
\end{split}
\end{align}
\end{subequations}
Equivalently, the Hessian satisfies
\begin{equation}
\label{eq:normedsquared:hessianComplexOrbitFinal}
\frac{1}{2} \, \Hessian_m \norm{J}^2_\kappa (\zeta \ldot m, 
\gamma \ldot m) = \Re \kappa_\C(\zeta, C^+_m R \gamma),
\end{equation}
where \( \zeta, \gamma \in \LieA{g}_\C \) and 
\begin{equation}
\label{R_operator}
R = \Matrix{0 & -\adAction_{J(m)} \\ \adAction_{J(m)} + Z_m & L_m}
= C^-_m \Matrix{0 & 0 \\ 0 & 1} + \I \, \bigl(\adAction_{J(m)} + Z_m\bigr).
\end{equation}
If \( J \) is equivariant with respect to the \( \adAction \)-action, 
then \( R = C^-_m \smallMatrix{0 & 0 \\ 0 & 1} \).
\end{prop}

\begin{proof} 
For~\eqref{eq:normedsquared:hessianSummary:usingL}, use the 
definition~\eqref{L_m_L_mu} of the operators \( L_m, Z_m \) and their 
symmetry properties in~\eqref{eq:normedsquared:hessianExpanded}.

Let \( T: \LieA{g}_\C \to \LieA{g}_\C \) be the \( \R \)-linear 
operator defined by
\begin{equation}
\frac{1}{2} \, \Hessian_m \norm{J}^2_\kappa (\zeta \ldot m, \gamma \ldot m) 
= \Re \kappa_\C(\zeta, T \gamma),
\end{equation}
where \( \zeta, \gamma \in \LieA{g}_\C \).
Writing \( T = \smallMatrix{T_{11} & T_{12} \\ T_{21} & T_{22}} \) 
and using~\eqref{T_operator},~\eqref{kappa_C}, this identity is equivalent to
\begin{equation}\begin{split}
\frac{1}{2} \, \Hessian_m \norm{J}^2_\kappa&(\xi_1 \ldot m + j \, 
(\xi_2 \ldot m), \eta_1 \ldot m + j \, (\eta_2 \ldot m)) \\ 
\qquad &= \kappa(\xi_1, T_{11} \eta_1 + T_{12} \eta_2) + 
\kappa(\xi_2, T_{21} \eta_1 + T_{22} \eta_2).
\end{split}\end{equation}
Comparing with \cref{eq:normedsquared:hessianSummary:usingL}, we read 
off that \( T_{11} = - Z_m (\adAction_{J(m)} + Z_m) \), \( T_{12} 
= (\adAction_{J(m)}^* - Z_m)L_m \), 
\( T_{21} = L_m(\adAction_{J(m)} + Z_m) \), and \( T_{22} = 
L_m^2 - Z_m \adAction_{J(m)} \).
By straightforward matrix multiplication and 
using~\eqref{eq:normedsquared:propertiesL:commute}, one then sees 
that \( T = C^+_m R \), with \( R \) as defined in~\eqref{R_operator}.
If \( J \) is \( \adAction \)-equivariant, then 
\( Z_m = - \adAction_{J(m)} \) and
\( R = \smallMatrix{0 & -\adAction_{J(m)} \\ 0 & L_m} = 
\smallMatrix{L_m & \adAction_{J(m)} \\ \adAction_{J(m)} & 
L_m} \smallMatrix{0 & 0 \\ 0 & 1} \). 
This establishes~\eqref{eq:normedsquared:hessianComplexOrbitFinal}.
\end{proof}

\begin{remark}
In the finite-dimensional setting, the 
expression~\eqref{eq:normedsquared:hessianComplexOrbitFinal} for the 
Hessian along the complex orbit has been established in 
\parencite[Theorem~3.8]{Wang2006} under the additional assumption 
that the momentum map is equivariant and \( \kappa \) is 
\( \AdAction \)-invariant. In~\parencite{Wang2006}, the commutativity 
of \( C^+_m \) and \( C^-_m \) was read off from 
formula~\eqref{eq:normedsquared:hessianComplexOrbitFinal} for 
the Hessian. Our proof proceeds, however, by first establishing 
the commutativity of \( C^+_m \) and 
\( C^-_m = C^+_m + 2 \I \adAction_{J(m)} \) in 
\cref{prop:normedsquared:propertiesL:commute} and then using this 
fact to express the Hessian in terms of \( C^+_m C^-_m \). 
\Cref{eq:normedsquared:propertiesL:commute} shows that 
the commutativity of \( C^+_m \) and \( C^-_m \) is a direct 
consequence of the \( \LieA{g}_m \)-invariance of \( j_m \) and 
of \( \kappa \), and that their difference is the adjoint action 
of an element of the stabilizer. This argument is completely 
independent of the norm-squared of the momentum map and its 
Hessian. In fact, from~\eqref{eq:normedsquared:hessianComplex} 
one sees that it is a \textquote{lucky coincidence} that under 
the same equivariance assumptions the formula for the Hessian 
simplifies considerably.

As an important consequence, 
\cref{prop:normedsquared:decompositionComplexStab} 
does not rely on the additional assumptions 
of \cref{prop:normedsquared:hessian} that are needed 
to calculate the Hessian in infinite dimensions. Moreover, 
this observation allowed us to establish the general decomposition 
result in \cref{prop:decomposition:decompositionComplexStab}, 
independently of the Hessian and for all points (not only the 
critical ones). 
\end{remark}

The completion \( \bar{\LieA{g}}_\C \) of \( \LieA{g}_\C \) with respect 
to the norm induced by \( \kappa_\C \) is a complex Hilbert space.
We continue to denote the inner product on  \( \bar{\LieA{g}}_\C \) 
by \( \kappa_\C \). By \cref{i:normedsquared:calabiSelfAdjoint}, the 
operators \( C^\pm_m \) give rise to densely defined, Hermitian 
operators \( C^\pm_m: \bar{\LieA{g}}_\C \supset \LieA{g}_\C \to 
\bar{\LieA{g}}_\C \).  We say that \( C^\pm_m: \LieA{g}_\C \to \LieA{g}_\C \) 
are \emphDef{essentially self-adjoint} if the closures of the operators 
\( C^\pm_m: \bar{\LieA{g}}_\C \supset \LieA{g}_\C \to \bar{\LieA{g}}_\C \) 
are self-adjoint. This is the case, for example, if \( C^\pm_m \) are 
elliptic operators defined on a closed manifold.

\begin{thm}
\label{prop:normedsquared:hessianPositiveDefAlongComplexOrbit}
Let \( (M, \omega) \) be a connected symplectic Fr\'echet 
manifold endowed with a symplectic action of a Fr\'echet Lie 
group \( G \). Assume that the action has an 
\( \adAction \)-equivariant momentum map \( J: M \to \LieA{g} \) 
relative to a non-degenerate symmetric pairing 
\( \kappa: \LieA{g} \times \LieA{g} \to \R \).
Let \( j \) be an almost complex structure on \( M \) 
compatible with \( \omega \). Let \( m \in M \) be a 
critical point of \( \norm{J}^2_\kappa \) such that 
\( j_m \) is invariant under \( \LieA{g}_m \).
In infinite dimensions, assume the following conditions {\rm(}which
always hold in finite dimensions{\rm)}:
	\begin{enumerate}
	\item
	\( G \) has a smooth exponential map.
	\item
	\( G_m \) is a Lie subgroup of \( G \).
	\item
	The isotropy representation of \( \LieA{g}_m \) on \( \TBundle_m M \) 
	has a momentum map \( \hat{J}: \TBundle_m M \to \LieA{g}_m \).
	\item
	The operators \( C^\pm_m \) are essentially self-adjoint.
	\end{enumerate}
Then the restriction of the Hessian of \( \norm{J}^2_\kappa \)
at \( m \) to the subspace \( \LieA{g}_\C \ldot m = 
\LieA{g} \ldot m + j \, (\LieA{g} \ldot m) \) of 
\( \TBundle_m M \) is positive semi-definite.
\end{thm}

Morally speaking, this theorem shows that the restriction of 
\( \norm{J}^2_\kappa \) to an orbit of the complexification 
\( G_\C \) of \( G \) is locally convex near a critical point.
Note, however, that in many infinite-dimensional examples of 
interest the complexified group \( G_\C \) does not exist.

\begin{proof}
By \cref{eq:normedsquared:hessianComplexOrbitFinal}, the Hessian 
of \( \norm{J}^2_\kappa \) at \( m \) satisfies
\begin{equation}\begin{split}
\frac{1}{2} \, \Hessian_m \norm{J}^2_\kappa \, 
\bigl(\zeta \ldot m, \gamma \ldot m \bigr) = 
\Re \kappa_\C\bigl(\zeta, C^+_m C^-_m 
\smallMatrix{0 & 0 \\ 0 & 1} \gamma\bigr).
\end{split}\end{equation}
Thus, we have to show that \( C^+_m C^-_m \) is a positive operator.
By \cref{i:normedsquared:kernelsIdentified}, the operators 
\( - C^+_m \) and \( - C^-_m \) are positive and, by 
\cref{i:normedsquared:calabiCommuteAdjoint}, they satisfy 
\( C^+_m C^-_m = (C^-_m)^* C^+_m \) on \( \LieA{g}_\C \).
Since \( - C^-_m \) is a positive, essentially self-adjoint 
operator, the operators \( \lambda \, \id_{\LieA{g}_\C} + 
C^-_m: \LieA{g}_\C \to \LieA{g}_\C \) are invertible with 
bounded inverse for all \( \lambda < 0 \).
Hence \parencite[Theorem~3.1~(viii)]{SebestyenStochel2003} 
implies that the composition \( (- C^+_m) (- C^-_m) = C^+_m C^-_m \) 
is a positive operator.
This finishes the proof.
\end{proof}

\section{Application: K\"ahler geometry}
\label{sec:kaehler}

\subsection{Momentum map for the action of 
\texorpdfstring{$\DiffGroup(M, \omega)$}{Diff(M, w)}}
\label{sec:kaehler:momentumMap}

Let \( (M, \omega) \) be a compact symplectic \( 2n \)-dimensional manifold.
An almost complex structure \( j \) on \( M \) is said to be 
\textit{compatible} with \( \omega \) if 
\( g_j \defeq \omega(\cdot, j \cdot) \) is a Riemannian 
metric, \ie, \( \omega(X, j X) > 0 \) for all \( X \neq 0 \), 
and \(j\) is a symplectic isomorphism on every fiber, \ie, 
\( \omega(j \cdot, j \cdot) = \omega \).
Consider the Fr\'echet manifold \( \SectionSpaceAbb{I} \) of all almost 
complex structures on \( M \) compatible with \( \omega \).
Each compatible almost complex structure \( j \in \SectionSpaceAbb{I} \) 
reduces the symplectic frame bundle to an \( \UGroup(n) \)-bundle.
Hence \( \SectionSpaceAbb{I} \) is identified with the space of smooth 
sections of a bundle over \( M \) with typical fiber 
\( \SpGroup(2n, \R) \slash \UGroup(n) \).
This homogenous space is the Siegel upper half space, and thus 
carries a symplectic structure.
The symplectic structure on the fiber induces naturally a symplectic 
structure \( \Omega \) on \( \SectionSpaceAbb{I} \); see 
\parencite{DiezRatiuAutomorphisms,Donaldson2003} for details.
Note that the tangent space \( \TBundle_j \SectionSpaceAbb{I} \) 
is the space of \( g_j \)-symmetric endomorphisms of \( \TBundle M \) 
that anti-commute with \( j \).
In an appropriate normalization, we have then
\begin{equation}
	\label{eq:kaehler:symplecticForm}
\Omega_j (A, B) = \frac{1}{4} \int_M \tr (A \, j \, B) \, \mu_\omega 
= \frac{1}{4} \int_M \tensor{A}{_p^q} \, \tensor{j}{_r^p} \, 
\tensor{B}{_q^r} \, \mu_\omega
\end{equation}
where \( A, B \in \TBundle_j \SectionSpaceAbb{I} \) and 
\( \mu_\omega = \frac{\omega^n}{n!} \).
Moreover, the almost complex structure
\begin{equation}
	\label{eq:kaehler:almostComplexStructure}
	\SectionSpaceAbb{j}_j (A) = -j A = Aj
\end{equation}
on \( \SectionSpaceAbb{I} \) is compatible with \( \Omega \) and the 
induced Riemannian metric is just the \( \LFunctionSpace^2 \)-inner product.

The group \( \DiffGroup(M, \omega) \) of symplectomorphisms acts naturally 
on \( \SectionSpaceAbb{I} \) by push-forward and leaves the symplectic form 
\( \Omega \) invariant. \Textcite{Fujiki1992,Donaldson1997} showed that 
the scalar curvature furnishes a momentum map for the action of the subgroup 
of \emph{Hamiltonian} diffeomorphisms. As a first step, we extend this result 
and determine the momentum map for the action of the full group of 
symplectomorphisms.

For the construction of the momentum map, we need to clarify what we mean by 
the dual space of \( \VectorFieldSpace(M, \omega)\defeq \set{\xi \in 
\VectorFieldSpace(M) \given \difLie_\xi \omega = 0} \). Note that 
the map \( \xi \mapsto \xi \contr \omega \) identifies 
\( \VectorFieldSpace(M, \omega) \) with the space of closed \( 1 \)-forms 
on \( M \). 
This suggests the choice \( {\VectorFieldSpace(M, \omega)}^* \defeq 
\DiffFormSpace^{2n-1}(M) \slash \dif \DiffFormSpace^{2n-2}(M) \) for the dual 
space of \( \VectorFieldSpace(M, \omega) \) relative to the pairing
\begin{equation}
\label{eq:kaehler:pairing}
\kappa\bigl(\equivClass{\alpha},\xi\bigr) = 
\frac{1}{(n-1)!}\int_M \alpha \wedge (\xi \contr\omega),
\end{equation}
where \( \equivClass{\alpha} \in \DiffFormSpace^{2n-1}(M) \slash \dif 
\DiffFormSpace^{2n-2}(M) \) and \( \xi \in \VectorFieldSpace(M, \omega) \).

\begin{prop}
\label{prop:kaehler:momentumMap}
The action of \( \DiffGroup(M, \omega) \) on \( \SectionSpaceAbb{I} \) 
leaves the symplectic form \( \Omega \) invariant and has a momentum 
map \( \SectionMapAbb{J}: \SectionSpaceAbb{I} \to 
\DiffFormSpace^{2n-1}(M) \slash \dif \DiffFormSpace^{2n-2}(M) \) relative 
to the pairing~\eqref{eq:kaehler:pairing}.
For every \( j_0 \in \SectionSpaceAbb{I} \), the unique momentum map 
\( \SectionMapAbb{J} \) satisfying \( \SectionMapAbb{J}(j_0) = 0 \) 
is given by assigning to \( j \in \SectionSpaceAbb{I} \) the equivalence 
class of \( J(j_0, j) \wedge \omega^{n-1} \), where the \( 1 \)-form 
\( J(j_0, j) \) is defined by
\begin{equation}
	\label{eq:kaehler:momentumMap:definitionJ}
J(j_0, j)(Y) = -\frac{1}{2} \tr \Bigl(\nabla (j-j_0) Y\Bigr) - 
\frac{1}{4} \tr \left((j+j_0)^{-1}(j-j_0) \, \nabla_Y (j+j_0)\right), 
\quad Y \in \VectorFieldSpace(M).
\end{equation}
Here, \( \nabla \) is a torsion-free connection satisfying 
\( \nabla \mu_\omega = 0 \), and \( \SectionMapAbb{J} \) 
does not depend on the choice of the connection \( \nabla \) used in its 
definition.
\end{prop}

For the proof, we construct a smooth contraction of the space of almost complex 
structures \( \SectionSpaceAbb{I} \) and apply the general results of 
\cref{sec:contractible}. It will be convenient to first consider the linear 
case and then apply these considerations to each fiber of \( \TBundle M \).
Thus, let \( (V, \omega) \) be a finite-dimensional symplectic vector space.
Similar to the nonlinear setting, the space \( \SectionSpaceAbb{I}(V, \omega) \) 
of complex structures compatible with \( \omega \) carries a symplectic form
\begin{equation}
\label{symplectic_form_on_complex_structures_vector_space}
\Omega_{j}(A, B) = \frac{1}{4} \tr (A \, j \, B), \qquad
j \in \SectionSpaceAbb{I}(V, \omega), \quad  
A,B \in \TBundle_j \SectionSpaceAbb{I}(V, \omega).
\end{equation}
Note that \( \Omega \) is invariant under the action of 
\( \SpGroup(V, \omega) \) given by \( g \cdot j = g j g^{-1} \).
The associated Lie algebra action is \( \xi \ldot j = \xi j - j \xi \), 
where \( \xi \in  \LieA{sp}(V, \omega) \defeq \set{\zeta \in 
\LieA{gl}(V) \given \omega(\zeta \cdot, \cdot) + 
\omega(\cdot, \zeta \cdot) = 0 } \).

Following \parencite[Proposition~II.2.3]{Audin2012}, for each \( j_0 \) 
in \( \SectionSpaceAbb{I}(V, \omega) \), the generalized Cayley transform
\begin{equation}
\label{Cayley_transform_complex_structures_vector_space}
	\phi_{j_0}(j) = (j + j_0)^{-1} (j - j_0)
	=-(j - j_0)(j + j_0)^{-1}
\end{equation}
is a diffeomorphism from \( \SectionSpaceAbb{I}(V, \omega) \) onto the open 
unit ball of the vector space of endomorphisms \( S: V \to V \) that are 
symmetric with respect to \( g_{j_0} = \omega(\cdot, j_0 \cdot) \) and that 
satisfy \( j_0 S + S j_0 = 0 \). We have \( \phi_{j_0}^{-1}(S)
= j_0 (I+S)(I-S)^{-1} \).
The map
\begin{equation}
	\Lambda(j_0, j, t) = \phi_{j_0}^{-1}\bigl(t \phi_{j_0}(j)\bigr)
\end{equation}
is a smooth contraction of \( \SectionSpaceAbb{I}(V, \omega) \).
It can easily be checked that \( \Lambda \) is equivariant with respect to 
the conjugation action by \( \SpGroup(V, \omega) \).
As a preparation for the nonlinear setting, we determine the \( 1 \)-forms 
occurring in the definition of the momentum map 
in~\eqref{eq:contractible:momentumMap}.
\begin{lemma}
\label{prop:kaehler:contractionLinear}
Let \( j_0, j \in \SectionSpaceAbb{I}(V, \omega) \).
For every \( A \in \TBundle_{j} \SectionSpaceAbb{I}(V, \omega) \), 
\begin{equation}
\label{kaehler_contraction_linear_first}
\int_0^1 \left(\Lambda_{j_0}^* \Omega\right)_{(j, t)} (\difp_t, A) \dif t 
= \frac{1}{4} \tr \bigl(\phi_{j_0}(j) A\bigr),
\end{equation}
where \( \Lambda_{j_0}\defeq \Lambda(j_0, \cdot,\cdot) \).
Moreover, for every \( A_0 \in \TBundle_{j_0} \SectionSpaceAbb{I}(V, \omega) \),
\begin{equation}
\label{kaehler_contraction_linear_second}
\int_0^1 \left(\bar{\Lambda}_j^* \, \Omega\right)_{(j_0, t)} 
(\difp_t, A_0) \dif t = \frac{1}{4} \tr \bigl(\phi_{j_0}(j) A_0\bigr),
\end{equation}
where \( \bar{\Lambda}_j \defeq \Lambda(\cdot, j, \cdot) \).
\end{lemma}

\begin{proof}
Using the identities
\begin{subequations}\label{eq:kaehler:contraction:derivativePhi}\begin{align}
\tangent_j \phi_{j_0} (A) &= 2 (j+j_0)^{-1} A (j+j_0)^{-1} j_0,
		\\
\tangent_S \phi_{j_0}^{-1} (C) &= 2 j_0 (1-S)^{-1} C (1-S)^{-1}
\end{align}\end{subequations}
we find
\begin{equation}
\tangent_j \Lambda_{j_0} (\difp_t) = - 2 S (1+tS)^{-2} j_0
\end{equation}
and 
\begin{equation}
\tangent_j \Lambda_{j_0}(A) = 
4 t (1+tS)^{-1} (j+j_0)^{-1} A (j+j_0)^{-1} (1-tS)^{-1}.
\end{equation}
Here and in the following, we abbreviated \( S \defeq \phi_{j_0}(j) \).
Using these formulas 
and~\eqref{symplectic_form_on_complex_structures_vector_space}, 
we calculate
\begin{equation}\begin{split}
\left(\Lambda_{j_0}^* \Omega\right)_{(j, t)} (\difp_t, A)
&= \Omega_{\Lambda_{j_0}(j,t)}\bigl(\tangent_j \Lambda_{j_0} 
(\difp_t), \tangent_j \Lambda_{j_0}(A)\bigr)
			\\
&= 2t \tr\left(S \bigl(1-t^2 S^2\bigr)^{-2} (j+j_0)^{-1} A (j+j_0)^{-1} \right).
\end{split}\end{equation}
Since
\begin{equation}
\difFrac{}{t} \bigl(1-t^2 S^2\bigr)^{-1} = 2 t S^2 \bigl(1-t^2 S^2\bigr)^{-2},
\end{equation}
we obtain\footnote{Note that it is enough to verify this identity for invertible \( S \) by density of the invertible elements in \( \End(V) \).}
\begin{equation}\begin{split}
\int_0^1 \left(\Lambda_{j_0}^* \Omega\right)_{(j, t)} (\difp_t, A) \dif t
&= \tr \left(S \bigl(1-S^2\bigr)^{-1} (j+j_0)^{-1} A (j+j_0)^{-1}\right)  \\
&= \frac{1}{4} \tr \bigl(S A\bigr).
\end{split}\end{equation}
This establishes~\eqref{kaehler_contraction_linear_first}.
	
The second identity~\eqref{kaehler_contraction_linear_second} follows 
from a similar, but slightly more involved, calculation.
Indeed, the derivative of the map \( \bar{\Lambda}_{j}(j_0, t) \defeq 
\Lambda(j_0, j, t) \) is given by
\begin{equation}
\tangent_{j_0} \bar{\Lambda}_j (A_0) = 
(1-t) (1+tS)^{-1} (A_0 - t S A_0 S) (1-tS)^{-1}.
\end{equation}
Thus, we find
\begin{equation}
\left(\bar{\Lambda}_{j}^* \Omega\right)_{(j_0, t)} (\difp_t, A_0)
= \frac{1}{2} \tr \left(\bigl(1-t^2 S^2\bigr)^{-2} S \bigl(1-tS^2\bigr) 
(1-t) A_0 \right).
\end{equation}
Using 
\begin{equation}
\difFrac{}{t} \bigl(-2 S^2 t+S^2 + 1\bigr)\bigl(1-t^2 S^2\bigr)^{-1} 
= -2 S^2 (1-t) \bigl(1-t S^2\bigr) \bigl(1-t^2 S^2\bigr)^{-2},
\end{equation}
this expression can easily be integrated over \( t \) to obtain
\begin{equation}
\int_0^1 \left(\bar{\Lambda}_{j}^* \Omega\right)_{(j_0, t)} 
(\difp_t, A_0) \dif t
= \frac{1}{4} \tr \bigl(S A_0\bigr).
\end{equation}
This verifies~\eqref{kaehler_contraction_linear_second} and finishes 
the proof.
\end{proof}

As a direct application, let us compute the momentum map \( J \) for the 
action of \( \SpGroup(V, \omega) \) on \( \SectionSpaceAbb{I}(V, \omega) \).
According to \cref{prop:contractible:momentumMap}, the unique momentum map 
vanishing at \( j_0 \) is given by
\begin{equation}
\kappa\bigl(J(j), \xi \bigr) = 
\frac{1}{4} \tr \bigl(\phi_{j_0}(j) (\xi \ldot j + 
\xi \ldot j_0)\bigr) = \frac{1}{2} \tr \bigl((j-j_0)\xi\bigr).
\end{equation}
Thus, identifying \( \SpAlgebra(V, \omega)^* \) with 
\( \SpAlgebra(V, \omega) \) using the pairing \( \kappa(\alpha, \xi) 
= \frac{1}{2} \tr(\alpha \xi) \), we find \( J(j) = j - j_0 \).
\medskip

We now finish the proof of \cref{prop:kaehler:momentumMap}.
In the nonlinear setting, the momentum map for the action of the 
symplectomorphism group \( \DiffGroup(M, \omega) \) on 
\( \SectionSpaceAbb{I} \equiv \SectionSpaceAbb{I}(M, \omega) \) 
is computed in a very similar way, with the twist that the final 
dualization involves integration by parts.
Let us discuss the details.
First, we extend the definition of the generalized Cayley transform 
\( \phi \) and the contraction \( \Lambda \) to \( \SectionSpaceAbb{I} \) 
by applying these maps pointwise.
The resulting map \( \Lambda: \SectionSpaceAbb{I} \times 
\SectionSpaceAbb{I} \times [0,1] \to \SectionSpaceAbb{I} \) is smooth 
in the \( \sFunctionSpace \)-topology by 
\parencite[Theorem~II.2.2.6]{Hamilton1982}.
Since \( \DiffGroup(M, \omega) \) acts on \( \SectionSpaceAbb{I} \) 
by push-forward, the infinitesimal action is given by 
\( \xi \ldot j = - \difLie_{\xi} j \), where \(\xi 
\in  \VectorFieldSpace(M, \omega ) \defeq \set{\zeta \in \VectorFieldSpace(M) \given
\difLie_{\zeta} \omega =0} \).
By \cref{prop:contractible:momentumMap,prop:kaehler:contractionLinear}, 
the momentum map \( \SectionSpaceAbb{J}: \SectionSpaceAbb{I} \to 
\DiffFormSpace^{2n-1}(M) \slash \dif \DiffFormSpace^{2n-2}(M) \) 
vanishing at \( j_0 \in \SectionSpaceAbb{I} \) is uniquely characterized by
\begin{equation}
	\label{eq:kaehler:momentumMapIntegralFormula}
\kappa\bigl(\SectionSpaceAbb{J}(j), \xi \bigr) = -
\frac{1}{4} \int_M \tr \bigl((j+j_0)^{-1}(j-j_0) 
\difLie_\xi (j + j_0)\bigr) \, \mu_\omega.
\end{equation}
In order to eliminate the derivative in \( \xi \)-direction, following 
\parencite{Garcia-PradaSalamonTrautwein2018}, we fix a torsion-free 
connection \( \nabla \) on \( M \) with \( \nabla \mu_\omega = 0 \) and 
introduce \( \tau(j, A) \in \DiffFormSpace^1(M) \), for \( j \in 
\SectionSpaceAbb{I} \) and \( A \in \TBundle_j \SectionSpaceAbb{I} \), 
by
\begin{equation}
	\label{eq:kaehler:tau}
\tau^\nabla(j, A)(Y) = \tr \bigl((\nabla A) Y\bigr) + 
\frac{1}{2} \tr \left(A j \, \nabla_Y j\right) 
= Y^r \nabla_p \tensor{A}{_r^p} + 
\frac{1}{2} Y^r \tensor{A}{_s^p}\tensor{j}{_q^s} \, \nabla_r \tensor{j}{_p^q}, 
\end{equation}
for \( Y \in \VectorFieldSpace(M) \).
By \parencite[Theorem~2.6]{Garcia-PradaSalamonTrautwein2018}, 
\( \tau^\nabla(j, A) \) does not depend on the connection \( \nabla \).
However, for the Levi-Civita connection of \( g_j \), the expression 
of \( \tau^\nabla \) simplifies considerably.
\begin{lemma}
	\label{prop:kaehler:tauForLeviCivita}
Let \( j \in \SectionSpaceAbb{I} \) and 
\( A \in \TBundle_j \SectionSpaceAbb{I} \).
Then, for the Levi-Civita connection \( \nabla \) of \( g_j \), 
we have
\begin{equation}
	\label{eq:kaehler:tauForLeviCivita}
\tau^\nabla(j, A)(Y) = \tr \bigl((\nabla A)Y\bigr) = 
Y^r \nabla_p \tensor{A}{_r^p}.
		\qedhere
	\end{equation}
\end{lemma}
\begin{proof}
By taking the covariant derivative of \( g_j(j \cdot, j \cdot) 
= g_j(\cdot, \cdot) \), we see that \( \nabla_Y j \) is 
antisymmetric with respect to \( g_j \) for every 
\( Y \in \VectorFieldSpace(M) \).
Thus, also \( j \nabla_Y j \) is antisymmetric.
But \( A \in \TBundle_j \SectionSpaceAbb{I} \) is a \( g_j \)-symmetric 
tensor, so \( A j \nabla_Y j \) is trace-free. 
\end{proof}
The importance of \( \tau \) lies in the following integration 
by parts identity \parencite[Theorem~2.6]{Garcia-PradaSalamonTrautwein2018}:
\begin{equation}
	\label{eq:kaehler:integrationByPartsTau}
	\frac{1}{2} \int_M \tr(A \difLie_\xi j) \, \mu_\omega 
	= -\int_M \tau^\nabla(j, Aj) \wedge (\xi \contr \mu_\omega)
	= - \kappa\bigl(\tau^\nabla(j, Aj) \wedge \omega^{n-1}, \xi\bigr),
\end{equation}
for \( \xi \in \VectorFieldSpace(M, \omega) \) and 
\( A \in \TBundle_j \SectionSpaceAbb{I} \).
Using this fact, we read-off from~\eqref{eq:kaehler:momentumMapIntegralFormula} 
that the momentum map is given by
\begin{equation}
\SectionSpaceAbb{J}(j) = 
\frac{1}{2} \Bigl(\tau^\nabla\bigl(j, (j+j_0)^{-1}(j-j_0) j\bigr) 
+ \tau^\nabla\bigl(j_0, (j+j_0)^{-1}(j-j_0) j_0\bigr)\Bigr) \wedge \omega^{n-1}.
\end{equation}
Now, for \( Y \in \VectorFieldSpace(M) \), we have by~\eqref{eq:kaehler:tau}
\begin{equation}\begin{split}
	\tau^\nabla\bigl(j, &(j+j_0)^{-1}(j-j_0) j\bigr)(Y) + 
	\tau^\nabla\bigl(j_0, (j+j_0)^{-1}(j-j_0) j_0\bigr)(Y)
		\\
	&= \tr \left(\nabla \bigl((j+j_0)^{-1}(j-j_0) (j+j_0)\bigr) Y 
	- \frac{1}{2} (j+j_0)^{-1}(j-j_0) \, \nabla_Y (j+j_0)\right)
		\\
	&= -\tr \bigl(\nabla (j-j_0) Y\bigr) - 
	\frac{1}{2} \tr \left((j+j_0)^{-1}(j-j_0) \,\nabla_Y (j+j_0)\right).
\end{split}\end{equation}
This finishes the proof of \cref{prop:kaehler:momentumMap}.

Alternatively, one can directly verify \cref{prop:kaehler:momentumMap} 
using the integration by parts relation~\eqref{eq:kaehler:integrationByPartsTau} 
and the following expression for the variation of the \( 1 \)-form 
\( J(j_0, j) \).
\begin{lemma}
	\label{prop:kaehler:variationJ}
	For every \( j_0, j \in \SectionSpaceAbb{I} \) and \( A \in \TBundle_j 
	\SectionSpaceAbb{I} \), we have
	\begin{equation}
		\label{eq:kaehler:variationJ}
		\tangent_j J(j_0, \cdot)(A) = - \frac{1}{2} \tau^\nabla(j, A) 
		- \frac{1}{4} \dif \tr\bigl(A \phi_{j_0}(j)\bigr).
		\qedhere
	\end{equation}
\end{lemma}
\begin{proof}
Continuing using the notation \( \phi_{j_0}(j) \) for the Cayley transform 
introduced in~\eqref{Cayley_transform_complex_structures_vector_space}, 
using~\eqref{eq:kaehler:contraction:derivativePhi} we get 
\begin{equation}\begin{split}
\tangent_j J(j_0, \cdot)(A)(Y)
&= -\frac{1}{2} \tr\bigl((\nabla A)Y\bigr) - 
\frac{1}{2} \tr \bigl((j+j_0)^{-1} A (j+j_0)^{-1} j_0 
\nabla_Y (j + j_0)\bigr) \\
&\qquad- \frac{1}{4} \tr \bigl(\phi_{j_0}(j) \nabla_Y A\bigr).
\end{split}\end{equation}
On the other hand,
\begin{equation}
\phi_{j_0}(j) jA = Aj - 2 (j+j_0)^{-1} A (j+j_0)^{-1} (j_0 + j).
\end{equation}
Since \( \phi_{j_0}(j) \), \( A \), and \( \nabla_Y j \) all anti-commute 
with \( j \), we obtain
\begin{equation}
0 = \tr(\phi_{j_0}(j) jA) = \tr\bigl(Aj \nabla_Y j\bigr) - 
2 \tr\bigl((j+j_0)^{-1} A (j+j_0)^{-1} (j_0 + j) \nabla_Y j \bigr)
\end{equation}
and thus
\begin{equation}\begin{split}
\tr\bigl(Aj \nabla_Y j\bigr) &- 2 \tr\bigl((j+j_0)^{-1} A 
(j+j_0)^{-1} j_0 \nabla_Y (j+j_0) \bigr) \\
&= 2 \tr\left((j+j_0)^{-1} A (j+j_0)^{-1} \bigl(j \nabla_Y j 
- j_0 \nabla_Y j_0 \bigr)\right) \\
&= 2 \tr\left(A (j+j_0)^{-1} \bigl(j \nabla_Y j - 
j_0 \nabla_Y j_0 \bigr) (j+j_0)^{-1}\right) \\
&= -\tr(A \nabla_Y \phi_{j_0}(j)).
\end{split}\end{equation}
Hence we conclude
\begin{equation}\begin{split}
\tangent_j J(j_0, \cdot)(A)(Y)
&= -\frac{1}{2} \tr\bigl((\nabla A)Y\bigr) - 
\frac{1}{4} \tr\bigl(Aj \nabla_Y j\bigr) - 
\frac{1}{4}\tr(A \nabla_Y \phi_{j_0}(j)) \\
&\qquad- \frac{1}{4} \tr \bigl(\phi_{j_0}(j) \nabla_Y A\bigr)\\
&= - \frac{1}{2} \tau^\nabla(j, A)(Y) - 
\frac{1}{4} \nabla_Y \tr\bigl(A \phi_{j_0}(j)\bigr).
\end{split}\end{equation}
This finishes the proof.
\end{proof}

In order to give a geometric interpretation of the \( 1 \)-form, and 
thereby of the momentum map \( \SectionMapAbb{J} \), we recall a few 
basic facts from almost complex geometry; see 
\parencite{GarciaPradaSalamon2020,Gauduchon2017} for more details.
For every almost complex structure \( j \), the Levi-Civita 
connection \( \nabla \) associated with the Riemannian metric 
\( g_j = \omega(\cdot, j \cdot) \) induces the Chern connection 
(also called the second canonical Hermitian connection)
\begin{equation}
	\tilde{\nabla}_X Y = \nabla_X Y - \frac{1}{2} j \, (\nabla_X j) Y.
\end{equation}
This is the unique connection that preserves the metric \( g_j \),
the complex structure \( j \), and the symplectic form \( \omega \), 
and whose torsion is the Nijenhuis tensor \( N_j \).
The \emphDef{Chern--Ricci form} is defined by
\begin{equation}
\mathrm{Ric}_j(X, Y) \defeq 
\frac{1}{2} \tr \left(R^{\tilde{\nabla}}(X, Y) j\right),
\end{equation}
where \( R^{\tilde{\nabla}} \) is the curvature of the Chern connection 
(on \( \TBundle M \)). With these conventions, \( \I \, \mathrm{Ric}_j \)
is the curvature of the induced Chern connection on the anti-canonical bundle 
\( \KBundle^{-1}_j M = \ExtBundle^{0,n} (\CotBundle M) \) and 
\( \frac{1}{2\pi}\mathrm{Ric}_j \) represents 
\( c_1(M) = c_1 \bigl(\KBundle^{-1}_j M\bigr) \in \sCohomology^2(M, \Z) \).
Moreover, \( S_j \defeq \tr_\omega \mathrm{Ric}_j \) is 
the \emph{Chern scalar curvature}.
In the following, we also need the normalized version
\begin{equation}
	\label{eq:kaehler:scalarCurvatureNormalized}
	\bar{S}_{j} \defeq S_{j} 
	- \frac{1}{\vol_{\mu_\omega}(M)} \int_M S_{j} \, \mu_\omega \, .
\end{equation}

Since the space \( \SectionSpaceAbb{I} \) of almost complex structures 
compatible with \( \omega \) is contractible, the anti-canonical bundles 
\( \KBundle^{-1}_j M \) are all isomorphic as \( j \) varies 
in \( \SectionSpaceAbb{I} \).
For every \( j_0, j \in \SectionSpaceAbb{I} \), choose an isomorphism 
\( \KBundle^{-1}_{j_0} M \isomorph \KBundle^{-1}_j M \), and 
let \( \widebar{J}(j_0, \cdot) \) be the difference of the 
Chern connections on \( \KBundle^{-1}_{j_0} M \) and 
\( \KBundle^{-1}_j M \) (under this isomorphism).
Choosing a different isomorphism of the anti-canonical bundles changes 
\( \widebar{J}(j_0, \cdot) \) by an exact \( 1 \)-form. This gives 
rise to a well-defined map
\begin{equation}
\widebar{J}(j_0, \cdot): \SectionSpaceAbb{I} \to 
\DiffFormSpace^1(M) \slash \dif \DiffFormSpace^0(M).
\end{equation}
\Textcite{Mohsen2003} showed that the derivative of this map is given by
\begin{equation}
\tangent_j \bigl(\widebar{J}(j_0, \cdot)\bigr)(A) = 
-\frac{1}{2} \tau^\nabla(j, A) \mod \dif \DiffFormSpace^0(M)
\end{equation}
for the Levi-Civita connection \( \nabla \); see also 
\parencites[Proposition~9.5.1]{Gauduchon2017}[Proposition~9]{Vernier2020}.
Clearly, \( \widebar{J}(j_0, \cdot) \) vanishes at \( j_0 \), and 
hence, by \cref{prop:kaehler:variationJ}, the maps 
\( j \mapsto \widebar{J}(j_0, j) \) and 
\( j \mapsto J(j_0, j) \mod \dif \DiffFormSpace^0(M) \) 
have to coincide.
In other words, the \( 1 \)-form \( J(j_0, j) \) is the difference of 
the Chern connections on the anti-canonical bundles 
\( \KBundle^{-1}_{j_0} M \) and \( \KBundle^{-1}_j M \) under the isomorphism 
\( \KBundle^{-1}_{j_0} M \isomorph \KBundle^{-1}_j M \) induced by 
the generalized Cayley transform \( \Lambda \).
Based on this discussion, an equivalent restatement of 
\cref{prop:kaehler:momentumMap}, enlightening the geometric meaning, 
is the following.
\begin{thm}
\label{prop:kaehler:momentumMapViaCanonicalBundle}
For every \( j_0, j \in \SectionSpaceAbb{I} \), let \( \widebar{J}(j_0, j) 
\in \DiffFormSpace^{1}(M) \slash \dif \DiffFormSpace^{0}(M) \) be the 
difference of the Chern connections on the anti-canonical 
bundles \( \KBundle^{-1}_{j_0} M \) 
and \( \KBundle^{-1}_j M \) under an arbitrary isomorphism 
\( \KBundle^{-1}_{j_0} M \isomorph \KBundle^{-1}_j M \).
Then the unique momentum map \( \SectionMapAbb{J}: \SectionSpaceAbb{I} 
\to \DiffFormSpace^{2n-1}(M) \slash \dif \DiffFormSpace^{2n-2}(M) \) 
for the action of \( \DiffGroup(M, \omega) \) on \( \SectionSpaceAbb{I} \) 
satisfying \( \SectionMapAbb{J}(j_0) = 0 \) is given by 
\( \SectionMapAbb{J}(j) = \widebar{J}(j_0, j) \wedge \omega^{n-1} \).
\end{thm}

\begin{remark}
In \parencite{DiezRatiuAutomorphisms}, we have investigated the action 
of \( \DiffGroup(M, \omega) \) on \( \SectionSpaceAbb{I} \) and showed 
that it admits a so-called group-valued momentum map.
Let us briefly outline the construction. Assume that \( \omega \) has 
integral periods so that there exists a prequantum bundle \( L \to M \).
Let \( \KBundle_j M \) be the canonical bundle induced by 
\( j \in \SectionSpaceAbb{I} \), and consider the map 
\begin{equation}
\skew{3}{\tilde}{\SectionMapAbb{J}}: \SectionSpaceAbb{I} \to 
\csCohomology^{2n}(M, \UGroup(1)), \qquad j \mapsto \KBundle_j M \star L^{n-1},
\end{equation}
where \( \csCohomology^{k}(M, \UGroup(1)) \) is the group of Cheeger-Simons 
differential characters, and \( \star: \csCohomology^{k}(M, \UGroup(1)) \times 
\csCohomology^{l}(M, \UGroup(1)) \to \csCohomology^{k+l}(M, \UGroup(1)) \) 
is the natural ring structure; see, \eg, \parencite{BaerBecker2013}.
By construction, \( \KBundle^{-1}_j M \star L^{n-1} \) can be viewed as a 
higher bundle with connection whose curvature is 
\( \mathrm{Ric}_j \wedge \omega^{n-1} = \frac{S_j}{2n} \omega^n \).
By \parencite[Theorem~4.10]{DiezRatiuAutomorphisms}, 
\( \skew{3}{\tilde}{\SectionMapAbb{J}} \) is a group-valued momentum 
map for the action of the group of symplectomorphisms in the sense 
that the left logarithmic derivative 
\( \difLog \skew{3}{\tilde}{\SectionMapAbb{J}} \in 
\DiffFormSpace^1\bigl(\SectionSpaceAbb{I}, 
\DiffFormSpace^{2n-1}(M) \slash \dif \DiffFormSpace^{2n-2}(M)\bigr) \) 
satisfies
\begin{equation}
\xi^* \contr \Omega + \kappa\bigl(\difLog \skew{3}{\tilde}{\SectionMapAbb{J}},
\xi\bigr) = 0,
\end{equation}
where \( \xi^* \) is the fundamental vector field on \( \SectionSpaceAbb{I} \) 
induced by the action of \( \xi \in \VectorFieldSpace(M, \omega) \).

Choose \( j_0 \in \SectionSpaceAbb{I} \).
Since the Chern class of the anti-canonical bundle \( \KBundle^{-1}_j M \) 
is independent 
of the almost complex structure \( j \), there exists a map 
\( \tilde{J}(j_0, \cdot): \SectionSpaceAbb{I} \to \DiffFormSpace^{1}(M) \slash 
\clZDiffFormSpace^{1}(M) \) such that 
\begin{equation}
\label{eq:kaehler:momentumMap:asDifferenceCanonicalBundle}
	\KBundle^{-1}_j M = \KBundle^{-1}_{j_0} M - 
	\iota\bigl(\tilde{J}(j_0, j)\bigr),
\end{equation}
where \( \iota: \DiffFormSpace^{k}(M) \slash \clZDiffFormSpace^{k}(M) 
\to \csCohomology^{k+1}(M, \UGroup(1)) \) is the inclusion of topologically 
trivial characters and \( \clZDiffFormSpace^{k}(M) \) is the space of 
closed forms with integral periods. This identity states that 
\( \KBundle^{-1}_j M \) and \( \KBundle^{-1}_{j_0} M \) are isomorphic, 
and \( \tilde{J}(j_0, j) \) 
is the difference of the Chern connections on these bundles up to gauge 
transformations. Since \( \SectionSpaceAbb{I} \) is contractible, there 
exists a lift \( \widebar{J}(j_0, \cdot): \SectionSpaceAbb{I} 
\to \DiffFormSpace^{1}(M) \slash \dif \DiffFormSpace^{0}(M) \) of 
\( \tilde{J}(j_0, \cdot) \) covering the projection 
\( \pr: \DiffFormSpace^{k}(M) \slash \dif \DiffFormSpace^{k-1}(M) 
\to \DiffFormSpace^{k}(M) \slash \clZDiffFormSpace^{k}(M) \). Thus,
\begin{equation}
\skew{3}{\tilde}{\SectionMapAbb{J}}(j) = 
\skew{3}{\tilde}{\SectionMapAbb{J}}(j_0) + 
\iota \circ \pr \bigl(\widebar{J}(j_0, j) \wedge \omega^{n-1} \bigr).
\end{equation}
Hence, the logarithmic derivative is given by 
\( \difLog \skew{3}{\tilde}{\SectionMapAbb{J}} = 
\bigl(\tangent \widebar{J}(j_0, \cdot)\bigr) \wedge \omega^{n-1} \).
This shows that the map
\begin{equation}
\SectionSpaceAbb{I} \ni j \mapsto \widebar{J}(j_0, j) \wedge \omega^{n-1} 
\in \DiffFormSpace^{2n-1}(M) \slash \dif \DiffFormSpace^{2n-2}(M)
\end{equation}
is a momentum map for the action of \( \DiffGroup(M, \omega) \) 
on \( \SectionSpaceAbb{I} \). Clearly, it vanishes at \( j_0 \) and 
thus by uniqueness has to coincide with the momentum map 
\( \SectionMapAbb{J} \).
In this way, we recover 
\cref{prop:kaehler:momentumMap,prop:kaehler:momentumMapViaCanonicalBundle}.
Note that the group-valued momentum map \( \skew{3}{\tilde}{\SectionMapAbb{J}} \) 
is equivariant under the action of \( \DiffGroup(M, \omega) \), but this 
equivariance is broken for \( \SectionMapAbb{J} \) by choosing a reference 
complex structure. We study this non-equivariance in more detail below 
and will see that it has a topological character.
\end{remark}

\Cref{prop:kaehler:momentumMap} implies that the Chern scalar 
curvature is the momentum map for the action of the subgroup of Hamiltonian 
diffeomorphisms.
In this way we recover the result of \textcite{Fujiki1992,Donaldson1997} 
that the Chern scalar curvature is the momentum map for the action of the 
subgroup of Hamiltonian diffeomorphisms.
\begin{coro}
	\label{prop:kaehler:momentumMapHamiltonian}
The action of the group of Hamiltonian diffeomorphisms on 
\( \SectionSpaceAbb{I} \) has a momentum map
\begin{equation}
\SectionMapAbb{J}_{\HamDiffGroup}(j) = 
\frac{1}{2} \bigl(S_{j} - S_{j_0}\bigr) \, \mu_\omega
\end{equation}
relative to the integration pairing of \( \sFunctionSpace_0(M) \) 
and \( \dif \DiffFormSpace^{2n-1}(M) \).
The non-equivariance one-cocycle 
\begin{equation}
\HamDiffGroup(M, \omega) \to \dif \DiffFormSpace^{2n-1}(M), 
\qquad \phi \mapsto \frac{1}{2} \bigl(S_{j_0} \circ 
\phi^{-1} - S_{j_0}\bigr) \, \mu_\omega
\end{equation}
is a coboundary.
\end{coro}
\begin{proof}
Consider the isomorphism of the space 
\(\HamVectorFields(M, \omega)\) of Hamiltonian vector fields with the 
space \( \sFunctionSpace_0(M) \) of smooth functions on \( M \) with zero 
mean given by the map \( f \mapsto X_f \).
By Hodge theory, the natural integration pairing gives a non-degenerate 
pairing of \( \sFunctionSpace_0(M) \) and \( \dif \DiffFormSpace^{2n-1}(M) \).
The following calculation, for 
\( \alpha \in \DiffFormSpace^{2n-1}(M) \) and 
\( f \in \sFunctionSpace_0(M) \),
\begin{equation}
	(n-1)! \, \kappa\bigl(\equivClass{\alpha}, X_f \bigr)
		= \int_M \alpha \wedge (X_f \contr \omega)
		= - \int_M \alpha \wedge \dif f
		= -\int_M f \dif \alpha
		= \dualPair{-\dif \alpha}{f},
\end{equation}
shows that the adjoint of the map \( f \mapsto X_f \) is essentially 
given by the exterior differential \( \dif: \DiffFormSpace^{2n-1}(M) 
\slash \dif \DiffFormSpace^{2n-2}(M) \to \dif \DiffFormSpace^{2n-1}(M) \).
Thus, the momentum map for the action of \( \HamDiffGroup(M, \omega) \) 
on \( \SectionSpaceAbb{I} \) is given by
\begin{equation}
\SectionMapAbb{J}_{\HamDiffGroup}(j) = 
\frac{-1}{(n-1)!} \dif \SectionMapAbb{J}(j) = 
-\dif J(j_0, j) \wedge \frac{\omega^{n-1}}{(n-1)!}.
\end{equation}
Now 
\begin{equation}
	\label{eq:kaehler:difOfJForm}
	\dif J(j_0, j) = \mathrm{Ric}_{j_0} - \mathrm{Ric}_{j}.
\end{equation}
This identity follows either from a direct calculation using identities 
of \parencite[Proof of Theorem~2.6]{Garcia-PradaSalamonTrautwein2018} or 
from the identification of \( J(j_0, j) \) as the difference of the 
Chern connections on the anti-canonical bundle by recalling 
that \( \I \, \mathrm{Ric}_j \) is 
the curvature of the Chern connection on \( K^{-1}_j M \).
Thus, invoking~\eqref{eq:symplectic:formWedgeOmegaNMinus1} and the
definition of the scalar curvature in terms of the Ricci form, we find
\begin{equation}
\SectionMapAbb{J}_{\HamDiffGroup}(j) = 
\bigl(\mathrm{Ric}_{j} - \mathrm{Ric}_{j_0}\bigr) \wedge 
\frac{\omega^{n-1}}{(n-1)!} = \frac{1}{2} (S_{j} - S_{j_0}) \, \mu_\omega.
\end{equation}
The expression for the non-equivariance cocycle follows directly from 
the definition~\eqref{group_one_cocycle}.
This finishes the proof.
\end{proof}

\subsection{Central extension and quasimorphism of \texorpdfstring{$\DiffGroup(M, \omega)$}{Diff(M, w)}}
\label{sec:kaehler:extension}
In contrast to the momentum map \( \SectionMapAbb{J}_{\HamDiffGroup} \) for the 
subgroup of Hamiltonian diffeomorphisms, the momentum map \( \SectionMapAbb{J} \) 
for the full group of symplectomorphisms is not equivariant, in general.
The different equivariance properties of these momentum maps can be succinctly 
captured using the work of \textcite{Vizman2006}, who used the exact sequence
\begin{equationcd}
	0 \to[r]
		& \HamVectorFields(M, \omega) \to[r]
		& \VectorFieldSpace(M, \omega) \to[r]
		& \deRCohomology^1(M) \to[r]
		& 0
\end{equationcd}
to show that the second continuous Lie algebra cohomology 
of \( \VectorFieldSpace(M, \omega) \) consists of sums of extensions of 
certain \( 2 \)-cocycles on \( \HamVectorFields(M, \omega) \) and pull-backs 
of elements of \( \ExtBundle^2 {\deRCohomology^1(M)}^* \); see 
\parencite[Corollary~4.4]{Vizman2006}.

In order to describe the corresponding decomposition of the non-equivariance 
cocycle of \( \SectionMapAbb{J} \), we need the following description of a 
cocycle associated with a closed \( 2 \)-form.
\begin{lemma}
\label{prop:kaehler:LichnerowiczCocyclePullback}
Let \( (M, \omega) \) be a closed \( 2n \)-dimensional symplectic manifold 
with \( n \geq 2 \), and \( \lambda \) be a closed \( 2 \)-form on \( M \).
Then the associated Lichnerowicz \( 2 \)-cocycle\footnotemark{}
\footnotetext{This is a slight abuse of conventions since, usually, the 
name \textquote{Lichnerowicz cocycle} refers to the cocycle defined on 
the Lie algebra of volume-preserving vector fields. However, we are 
mainly interested in its restriction to the subalgebra of symplectic 
vector fields.}
\begin{equation}
\lambda_c (\xi, \eta) \defeq \int_M \lambda(\xi, \eta) \, \mu_\omega
\end{equation}
on \( \VectorFieldSpace(M, \omega) \) is cohomologous to the pull-back
by the map  
\( \VectorFieldSpace(M, \omega) \ni \xi \mapsto  [\xi \contr \omega] 
\in \deRCohomology^1(M) \) of the bilinear form
\begin{equation}
\label{eq:kaehler:LichnerowiczCocyclePullback:onH1}
\lambda_c^H\bigl(\equivClass{\alpha}, \equivClass{\beta}\bigr) 
= \int_M \Biggl(\frac{\mathrm{Av}_\omega (\lambda)}{2 (n-1)}  \, \omega 
- \lambda\Biggr) \wedge \alpha \wedge \beta \wedge 
\frac{\omega^{n-2}}{(n-2)!}
\end{equation}
on \( \deRCohomology^1(M) \), where \( \mathrm{Av}_\omega (\lambda) \defeq 
\frac{1}{\vol_{\mu_\omega}(M)} \int_M \tr_\omega (\lambda) \, \mu_\omega \) and 
\( \tr_\omega (\lambda) = \tensor{\lambda}{_i^i} = \varpi^{ij} \lambda_{ij}\); 
see~\eqref{sharp}.
In particular, the restriction of \( \lambda_c \) to 
\( \HamVectorFields(M, \omega) \) is trivial in Lie algebra cohomology.

Moreover, \( \lambda_c^H \) vanishes if \( \lambda \wedge \omega^{n-2} \) 
is exact. If the cup product yields an isomorphism 
\( \ExtBundle^2 {\deRCohomology^1(M)} = \deRCohomology^2(M) \), then 
exactness of \( \lambda \wedge \omega^{n-2} \) is also necessary for 
\( \lambda_c^H \) to vanish.
\end{lemma}
Note that the bilinear form \( \lambda_c^H \) factors through the 
cup product \( \deRCohomology^1(M) \times \deRCohomology^1(M) \to 
\deRCohomology^2(M) \), and as such is closely related to the 
skew-structures on \( \pi_1(M) \) studied by \textcite{JohnsonRees1991}.

\begin{proof}
By construction, \( \bar{\lambda} = 
\tr_\omega (\lambda) - \mathrm{Av}_\omega (\lambda) \) 
has average value zero. Thus, there exists 
\( \tau \in \DiffFormSpace^{2n-1}(M) \) such 
that \( \dif \tau = \bar{\lambda} \mu_\omega \).
Now, the calculation
\begin{equation}
	\int_M \bar{\lambda} \omega(\xi, \eta) \, \mu_\omega 
		= \int_M \omega(\xi, \eta) \dif \tau
		= - \int_M \difLie_{\eta}(\xi \contr \omega) \wedge \tau
		= \int_M (\commutator{\xi}{\eta} \contr \omega) \wedge \tau
\end{equation}
shows that \( \lambda_c \) is cohomologous to the cocycle
\begin{equation}
\tilde{\lambda}_c(\xi, \eta) = \int_M \left(\lambda - 
\frac{\bar{\lambda}}{2} \omega\right) (\xi, \eta) \, \mu_\omega \, .
\end{equation}

On the other hand, using~\eqref{eq:symplectic:formWedgeOmegaNMinus1}, 
we find for every \( 2 \)-form \( \alpha \) that
\begin{equation}\begin{split}
\alpha(\xi, \eta) \, \mu_\omega	
&= (\xi \contr \omega) \wedge (\eta \contr \alpha) 
\wedge \frac{\omega^{n-1}}{(n-1)!}
			\\
&= \alpha \wedge \frac{\omega^{n-1}}{(n-1)!} \, \omega(\xi, \eta) 
- \lambda \wedge (\xi \contr \omega) \wedge (\eta \contr \omega) \wedge 
\frac{\omega^{n-2}}{(n-2)!}
			\\
&= \frac{1}{2} \tr_\omega(\alpha) \, \omega(\xi, \eta) \, \mu_\omega - 
\alpha \wedge (\xi \contr \omega) \wedge (\eta \contr \omega) 
\wedge \frac{\omega^{n-2}}{(n-2)!}.
\end{split}\end{equation}
The first and second identities follow from contracting 
\( (\xi \contr \alpha) \wedge \omega^n = 0 \) and 
\( (\xi \contr \omega) \wedge \alpha \wedge \omega^{n-1} = 0 \) 
with \( \eta \), respectively. 
Hence, using this relation for \( \alpha = \lambda \) 
and in the second line for \( \alpha = \omega \), we obtain
\begin{equation}\begin{split}
\tilde{\lambda}_c(\xi, \eta) 
&= \int_M \frac{1}{2} \mathrm{Av}_\omega (\lambda) \, 
\omega (\xi, \eta) \, \mu_\omega - 
\lambda \wedge (\xi \contr \omega) \wedge 
(\eta \contr \omega) \wedge \frac{\omega^{n-2}}{(n-2)!}
			\\
&= \int_M \Biggl(\frac{1}{2 (n-1)} \mathrm{Av}_\omega(\lambda) \, \omega 
- \lambda\Biggr) \wedge (\xi \contr \omega) \wedge 
(\eta \contr \omega) \wedge \frac{\omega^{n-2}}{(n-2)!}\,.
\end{split}\end{equation}
From this expression, it is evident that \( \tilde{\lambda}_c \) is the 
pull-back from \( \deRCohomology^1(M) \) of \( \lambda_c^H \).

Finally, if \( \lambda \wedge \omega^{n-2} \) is exact, then the average 
of \( \lambda \) vanishes by~\eqref{eq:symplectic:formWedgeOmegaNMinus1}, 
and thus also \( \lambda_c^H = 0 \). Conversely, assume that cup product 
identifies \( \ExtBundle^2 {\deRCohomology^1(M)} \) with 
\( \deRCohomology^2(M) \) and that \( \lambda_c^H = 0 \).
Then, the linear functional \( \deRCohomology^2(M) \to \R \) given by 
integration against \( \sigma \wedge \omega^{n-2} \) with \( \sigma = 
\frac{\mathrm{Av}_\omega(\lambda)}{2 (n-1)} \omega - \lambda \) 
has to vanish, \ie, \( \sigma \wedge \omega^{n-2} \) has to be exact.
But then \( 0 = \mathrm{Av}_\omega(\sigma) = 
\mathrm{Av}_\omega(\lambda) \bigl(\frac{n}{n-1} - 1\bigr) \), 
and thus \( \mathrm{Av}_\omega(\lambda) 
= 0 \). Hence, \( \sigma = \lambda \) and 
\( \lambda \wedge \omega^{n-2} \) has to be exact.
\end{proof}

Applied to the non-equivariance cocycle of \( \SectionMapAbb{J} \), we 
find the following. 
\begin{prop}
The class of the non-equivariance cocycle of \( \SectionMapAbb{J} \) in the second 
continuous Lie algebra cohomology of \( \VectorFieldSpace(M, \omega) \) 
coincides with the pull-back along the natural map 
\( \VectorFieldSpace(M, \omega) \to \deRCohomology^1(M) \) of the 
antisymmetric bilinear form
\begin{equation}
\Sigma^H_{j_0}\bigl(\equivClass{\alpha}, \equivClass{\beta}\bigr) \defeq \int_M 
\left((\mathrm{Ric}_{j_0})_{rs} - \frac{1}{2}\bar{S}_{j_0}\omega_{rs} 
\right) \alpha^r \beta^s \, \mu_\omega
\end{equation}
on \( \deRCohomology^1(M) \), where the indices of \( \alpha \) and 
\( \beta \) are raised using \( \omega \).
Moreover, the non-equivariance cocycle is trivial in Lie algebra 
cohomology if \( c_1(M) \cup \equivClass{\omega}^{n-2} = 0 \), and 
this condition is necessary when the cup product yields an 
isomorphism \( \ExtBundle^2 {\deRCohomology^1(M)} = \deRCohomology^2(M) \).
\end{prop}
Note that the class of the non-equivariance cocycle \( \Sigma^H_{j_0} \) 
is independent of the reference complex structure \( j_0 \) and thus is 
a well-defined invariant of the symplectic manifold \( (M, \omega) \).
Recall that a K\"ahler manifold with vanishing first real Chern class is 
called a Calabi--Yau manifold. Thus, for Calabi--Yau manifolds, the 
non-equivariance cocycle is trivial in Lie algebra cohomology.
To emphasize the close relationship, we say that a symplectic manifold 
\( (M, \omega) \) is \emph{weakly Calabi--Yau} if \( \Sigma^H_j = 0 \) 
for some compatible almost complex structure \(  j \).
\begin{proof}
Let \( j_0 \in \SectionSpaceAbb{I} \) and let 
\( X, Y \in \VectorFieldSpace(M, \omega) \).
By~\eqref{eq:contractible:twoCocycle}, the non-equivariance 2-cocycle 
\( \Sigma \) is given by \( \Sigma(X, Y) = -
\kappa\Bigl(\tangent_{j_0}\SectionMapAbb{J} \bigl(\difLie_X j_0\bigr), 
Y\Bigr) \). On the other hand, by~\eqref{eq:kaehler:variationJ}, we find
\begin{equation}\begin{split}
\tangent_{j_0}J(j_0, \cdot) \bigl(\difLie_X j_0\bigr)
&= -\frac{1}{2} \tau^\nabla(j_0, \difLie_X j_0)
= - X \contr \mathrm{Ric}_{j_0} -\frac{1}{2} \dif \divergence(j_0X),
\end{split}\end{equation}
where the second equality follows from 
\parencite[Theorem~2.7]{Garcia-PradaSalamonTrautwein2018}.
Thus, in summary,
\begin{equation}\begin{split}
\Sigma(X, Y) 
&= \kappa\bigl((X \contr \mathrm{Ric}_{j_0})\wedge \omega^{n-1}, Y\bigr)
			\\
&= \int_M (X \contr \mathrm{Ric}_{j_0}) \wedge \frac{\omega^{n-1}}{(n-1)!} 
\wedge (Y \contr \omega)
			\\
&= \int_M \mathrm{Ric}_{j_0}(X, Y) \, \mu_\omega.
\end{split}\end{equation}
Alternatively, this identity for the non-equivariance cocycle follows 
directly from~\parencite[Eq.~(2.33)]{Garcia-PradaSalamonTrautwein2018}.
The claim now is a consequence of 
\cref{prop:kaehler:LichnerowiczCocyclePullback}.
\end{proof}

From \cref{prop:contractible:groupExt} we know that the 
non-equivariance cocycle of \( \SectionMapAbb{J} \) integrates to 
a central extension of \( \DiffGroup(M, \omega) \).
In fact, the associated group \( 2 \)-cocycle \( c \) on 
\( \DiffGroup(M, \omega) \) can be explicitly computed 
using~\eqref{eq:kaehler:contraction:derivativePhi}, at least in principle.
The cocycle \( c \) also coincides with the cocycle given in 
\textcite{Reznikov1999} (where it appeared out of thin air).
Moreover, according to \cref{rem:contractible:quasimorhpism}, the 
cocycle \( c \) is bounded in the sense of Gromov.
This follows from the fact that \( \SectionSpaceAbb{I} \) is a 
Domic-Toledo space, essentially because it is the space of sections 
of a bundle whose typical fiber is a Domic-Toledo space; see 
\parencite[Section~1.7]{Shelukhin2014} for details.

Although this prescription yields a direct way to construct the 
central extension of \( \DiffGroup(M, \omega) \), its geometric 
interpretation still remains unclear.
This can be compared to the description of the momentum map above: 
\cref{prop:kaehler:momentumMap} gives a concrete formula for the 
momentum map \( \SectionMapAbb{J} \), but its geometric interpretation 
in terms of the anti-canonical bundle is not obvious from this formula.
A first step towards a geometric interpretation of the central extension 
is to better understand the prequantum bundle of \( \SectionSpaceAbb{I} \).
For example, realize it as a determinant line bundle of certain Dirac 
operators or use the asymptotic prescription of \parencite{FothUribe2007}.

If \( \Sigma^H_{j} = 0  \) for some compatible almost complex structure 
\( j \), \ie, \( (M, \omega) \) is weakly Calabi--Yau, then the momentum 
map \( \SectionMapAbb{J} \) is infinitesimally equivariant.
We can thus apply the general construction of \textcite{Shelukhin2014} 
to obtain a quasimorphism on the universal covering of 
\( \DiffGroup(M, \omega)_0 \).
By construction, the restriction of this quasimorphism to the subgroup 
of Hamiltonian diffeomorphisms is the (non-trivial) quasimorphism 
constructed in \parencite{Shelukhin2014}.
For completeness, let us record this observation.
\begin{prop}
If \( (M, \omega) \) is weakly Calabi--Yau, then the universal 
covering of \( \DiffGroup(M, \omega)_0 \) admits a non-trivial 
quasimorphism and hence has infinite commutator length.
\end{prop}
Under the stronger assumption that the first Chern class vanishes, 
\textcite{Entov2004} constructed a quasimorphism on the universal 
covering of \( \DiffGroup(M, \omega)_0 \) that coincides with the 
Shelukhin quasimorphism on Hamiltonian diffeomorphisms.
Hence, it is natural to conjecture that the quasimorphism constructed 
above is a natural generalization of the Entov quasimorphism; see 
also \parencite[Point~3.2]{Shelukhin2014}.

\subsection{Norm-squared momentum map}

We identify the Lie algebra \( \HamVectorFields(M, \omega) \) of 
Hamiltonian vector fields with the Lie algebra \( \sFunctionSpace_0(M) \) 
of smooth functions with average zero by \( F \mapsto X_F \).
Using this identification, there is a natural inner product on 
\( \HamVectorFields(M, \omega) \) defined by
\begin{equation}
\label{eq:kaehler:pairingOnHam}	
\scalarProd{X_F}{X_G} = \frac{1}{2} \int_M F G \, \mu_\omega.
\end{equation}
To extend this inner product to \( \VectorFieldSpace(M, \omega) \), 
consider the exact sequence
\begin{equation}
\label{eq:kaehler:symplecticVectorFieldsExactSequence}
0 \to \HamVectorFields(M, \omega) \to \VectorFieldSpace(M, \omega) 
\to \deRCohomology^1(M) \to 0.
\end{equation}
We split this sequence by choosing a reference almost complex 
structure \( j_0 \in \SectionMapAbb{I} \).
Let \( \mathfrak{har}_{j_0}(M) \isomorph \deRCohomology^1(M) \) 
be the space of all vector fields \( \xi^h \) such that 
\( \omega^\flat(\xi^h) \) is a \( g_{j_0} \)-harmonic 
\( 1 \)-form. We call such infinitesimally 
symplectic vector fields \textit{harmonic}.
Then every \( \xi \in \VectorFieldSpace(M, \omega) \) 
can be uniquely written as \( \xi = X_F + \xi^h \) for some 
\( F \in \sFunctionSpace_0(M) \) and 
\( \xi^h \in \mathfrak{har}_{j_0}(M) \). Define the following 
inner product on \( \VectorFieldSpace(M, \omega) \):
\begin{equation}
\label{eq:kaehler:pairingOnLieAlgebra}
\scalarProd{\xi}{\eta}_{j_0} = 
\int_M \frac{1}{2} F G \, \mu_\omega \, + \omega^\flat(\xi^h) 
\wedge \hodgeStar_{g_{j_0}} \omega^\flat(\eta^h) = 
\int_M \left(\frac{1}{2} F G + 
g_{j_0}(\xi^h, \eta^h)\right) \, \mu_\omega 
\end{equation}
for \( \xi = X_F + \xi^h \) and \( \eta = X_G + \eta^h \),
where \(\hodgeStar_{g_{j_0}}\) is the
Hodge star operator defined by the Riemannian metric \(g_{j_0}\).
From \cref{prop:kaehler:momentumMapViaCanonicalBundle}, we obtain 
the momentum map relative to the inner 
product \( \scalarProdDot_{j_0} \)
given in~\eqref{eq:kaehler:pairingOnLieAlgebra}.
\begin{thm}
\label{prop:kaehler:momentumMap:LieAlgebraValued}
For every \( j_0 \in \SectionSpaceAbb{I} \), the action 
of \( \DiffGroup(M, \omega) \) on \( \SectionSpaceAbb{I} \) 
has a momentum map 
\( \SectionMapAbb{J}: \SectionSpaceAbb{I} \to 
\VectorFieldSpace(M, \omega) \) relative to the inner 
product \( \scalarProdDot_{j_0} \) given by assigning 
to \( j \in \SectionSpaceAbb{I} \) the symplectic vector field
\begin{equation}
- X_{\bar{S}_j} + \omega^\sharp\bigl(j_0 J(j_0, j)\bigr)^h,
\end{equation}
where \( \bar{S}_j \) is the normalized Chern scalar curvature 
given by~\eqref{eq:kaehler:scalarCurvatureNormalized}, the 
\( 1 \)-form \( J(j_0, j) \) is defined 
in~\eqref{eq:kaehler:momentumMap:definitionJ}, and the 
superscript \( h \) refers to taking the 
\( g_{j_0} \)-harmonic part.

Moreover, \( \SectionMapAbb{J}_\HamDiffGroup: 
\SectionSpaceAbb{I} \ni j \mapsto - X_{\bar{S}_j} 
\in \HamVectorFields(M, \omega) \) is a momentum 
map for the action of \( \HamDiffGroup(M, \omega) \), 
relative to the inner product \( \scalarProdDot \) 
defined in~\eqref{eq:kaehler:pairingOnHam}.
\end{thm}
\begin{proof}
With respect to the pairing \( \kappa \) 
from~\eqref{eq:kaehler:pairing}, we have for every 
\( \xi = X_F + \xi^h \in \VectorFieldSpace(M, \omega) \) 
and \( \alpha \in \DiffFormSpace^1(M) \): 
\begin{equation}\begin{split}
\kappa\bigl(\equivClass{\alpha \wedge \omega^{n-1}}, \xi\bigr)
&= \frac{1}{(n-1)!} \int_M \alpha \wedge \omega^{n-1} 
\wedge (\xi \contr \omega)
			\\
&= \frac{1}{(n-1)!} \int_M F \, \dif \alpha \wedge \omega^{n-1} 
+ \alpha \wedge (\xi^h \contr \omega) \wedge \omega^{n-1}
			\\
&= \int_M \left(\frac{1}{2} F \tr_\omega(\dif \alpha) 
+ \alpha(\xi^h)\right) \, \mu_\omega \,,
\end{split}\end{equation}
where the last equality follows 
from~\eqref{eq:symplectic:formWedgeOmegaNMinus1}.
On the other hand, we have \( \alpha(\xi^h) = 
g_{j_0}\bigl(\omega^\flat (\xi^h), j_0 \alpha\bigr) \) 
with \( j_0 \alpha \defeq - \alpha(j_0 \cdot) \).
By assumption, \( \omega^\flat (\xi^h) \) is a 
\( g_{j_0} \)-harmonic \( 1 \)-form.
Thus, \( \LFunctionSpace^2 \)-orthogonality of the 
Hodge decomposition implies that only the harmonic 
part \( (j_0\alpha)^h \) of \( j_0 \alpha \) contributes 
to the integral, and we obtain
\begin{equation}\label{eq:kaehler:projectionDualSymp}\begin{split}
\kappa\bigl(\equivClass{\alpha \wedge \omega^{n-1}}, \xi\bigr)
&= \int_M \left(\frac{1}{2} F \tr_\omega(\dif \alpha) + 
g_{j_0}\left(\omega^\sharp\bigl((j_0\alpha)^h\bigr), 
\xi^h\right)\right) \, \mu_\omega
		\\
&= \scalarProd{\eta_\alpha}{\xi}_{j_0},
\end{split}\end{equation}
for \( \eta_\alpha = X_G + \omega^\sharp\bigl((j_0 \alpha)^h\bigr) \) 
with \( G = \tr_\omega(\dif \alpha) -  \mathrm{Av}_\omega(\dif \alpha)\).
Note that if \( \alpha \) is exact, say \( \alpha = \dif f \), then the 
K\"ahler identities, see \parencite[Proposition~1.14.1]{Gauduchon2017},
imply \( j_0 \alpha = \diF (f \omega) \) so that \( (j_0 \alpha)^h = 0 \), 
and hence \( \eta_\alpha = 0 \).
This verifies that \( \eta_\alpha \) depends only on the equivalence 
class of \( \alpha \) modulo exact forms.
	
Finally, \cref{prop:kaehler:momentumMap} and~\eqref{eq:kaehler:difOfJForm} 
imply that \( \SectionMapAbb{J}(j) = X_F + \eta^h \) with
\begin{equation}
F = S_{j_0} - S_j \quad \text{and} \quad \omega^\flat (\eta^h) 
= \bigl(j_0 J(j_0, j)\bigr)^h
\end{equation}
is a momentum map relative to the inner product 
\( \scalarProdDot_{j_0} \).
Clearly, we can shift \( \SectionMapAbb{J} \) by a constant 
and still obtain a momentum map.
The momentum map for the subgroup of Hamiltonian diffeomorphisms 
can be calculated in a similar way.
This finishes the proof.
\end{proof}

\begin{remark}
\label{rem:kaehler:momentumMap:ciAsHolonomies}
Assume that \( \mathrm{Ric}_{j_0} = \mathrm{Ric}_{j} \), 
so that \( J(j_0, j) \) is closed by~\eqref{eq:kaehler:difOfJForm}.
In this situation, the harmonic form \( (j_0 J(j_0, j))^h \) in 
\cref{prop:kaehler:momentumMap:LieAlgebraValued} has a nice 
geometric interpretation.

To find it, choose an orthonormal basis 
\( \set{\alpha_p} \) of \( g_{j_0} \)-harmonic \( 1 \)-forms, 
and expand \( (j_0 J(j_0, j))^h = \sum_p c_p \alpha_p \).
The coefficients \(c_p\) are given by
\begin{equation}\begin{split}
c_p &= \dualPair{j_0 J(j_0, j)}{\alpha_p}_{j_0}
= \int_M g_{j_0}(j_0 J(j_0, j), \alpha_p) \, \mu_\omega
			\\
&= \int_M \omega(J(j_0, j), \alpha_p) \, \mu_\omega
= \int_M J(j_0, j) \wedge \alpha_p \wedge \frac{\omega^{n-1}}{(n-1)!}\,,
\end{split}\end{equation}
where the last equality follows 
from~\eqref{eq:symplectic:symplecticHodge1Form}.
Thus, in summary, \( c_p = \int_{\gamma_p} J(j_0, j) \) 
for the Poincar\'e dual \( \gamma_p \) of the \( (2n-1) \)-form 
\( \alpha_p \wedge \frac{\omega^{n-1}}{(n-1)!} \).
Now recall that \( J(j_0, j) \) is the difference of the Chern 
connections of \( j_0 \) and \( j \) on \( \KBundle_{j_0} M \) 
(relative to an identification of 
\( \KBundle_{j_0} M \isomorph \KBundle_j M \)).
Thus, \( c_p \) is the difference of the holonomies of the 
Chern connections of \( j_0 \) and \( j \) around the 
loop \( \gamma_p \).
\end{remark}

The norm-squared of the momentum map for the 
action of Hamiltonian diffeomorphisms is  the 
\( \LFunctionSpace^2 \)-norm of the (normalized) scalar curvature
(see
\cref{prop:kaehler:momentumMap:LieAlgebraValued}), 
that is, the Calabi energy functional on \( \SectionSpaceAbb{I} \):
\begin{equation}
\label{eq:kaehler:momentumMapHamSquared}
\norm{\SectionMapAbb{J}_\HamDiffGroup}^2(j) = 
\frac{1}{2} \int_M {\bar{S}_j}^2 \mu_\omega \,. 
\end{equation}
Critical points of \( \norm{\SectionMapAbb{J}_\HamDiffGroup}^2 \) 
are called \emph{extremal almost-K\"ahler metrics}; see 
\parencite{Calabi1985} in the K\"ahler setting and 
\parencite{Lejmi2010} without the integrability condition.
According to \cref{prop:normedsquared:criticalPoints}, these 
are precisely the almost complex structures \( j \) for which 
the Hamiltonian vector field \( X_{\bar{S}_j} \) is a real 
holomorphic vector field, \ie, \( \difLie_{X_{\bar{S}_j}} j = 0 \).
Constant scalar curvature metrics constitute an important special 
case of extremal almost-K\"ahler metrics, and they correspond 
to zeros of \( \SectionMapAbb{J}_\HamDiffGroup \).

Similarly, the norm-squared of the momentum map 
for the full group of symplectomorphisms (see
\cref{prop:kaehler:momentumMap:LieAlgebraValued}) 
yields the following 
functional on \( \SectionSpaceAbb{I} \):
\begin{equation}
\label{eq:kaehler:momentumMapSquared}
\norm{\SectionMapAbb{J}}_{j_0}^2(j) = 
\frac{1}{2} \int_M \left({\bar{S}_j}^2 + 
2 \norm{\bigl(j_0 J(j_0, j)\bigr)^h}_{j_0} \right) \mu_\omega \,, 
\end{equation}
where on the right-hand side the norm of the one-form 
\( \bigl(j_0 J(j_0, j)\bigr)^h \) is taken with respect 
to the metric \( g_{j_0} \).
The first summand is again the Calabi energy.
The second summand penalizes the difference between the 
Chern connections of \( j \) and \( j_0 \).
In other words, \( \norm{\SectionMapAbb{J}}_{j_0}^2 \) 
can be viewed as a localized Calabi energy.
Zeros of \( \norm{\SectionMapAbb{J}}_{j_0}^2 \) are, if they exist, 
almost complex structures \( j \) with constant Chern scalar curvature 
and \( j_0 J(j_0, j) \) having no \( g_{j_0} \)-harmonic component.
In analogy with Calabi's extremal metrics, we say that an almost 
complex structure \( j \) is a \emphDef{\( j_0 \)-extremal metric} 
if it is a critical point of \( \norm{\SectionMapAbb{J}}_{j_0} \).
By \cref{prop:normedsquared:criticalPoints}, this is equivalent 
to \( \SectionMapAbb{J}(j) \) being a real holomorphic vector field.

Our general results in 
\cref{sec:momentumMapSquared,sec:stabilizer_algebra_of_the complexified_action} 
require a few technical properties. Let us check that these 
are satisfied in the present situation.
\begin{lemma}
The following holds:
\begin{thmenumerate}
\item The group of symplectomorphisms \( \DiffGroup(M, \omega) \) 
is an infinite-dimensional Fr\'echet Lie group and has a smooth 
exponential map given by the flow.
\item For every \( j \in \SectionSpaceAbb{I} \), the stabilizer 
\( \DiffGroup(M, \omega)_j \) is a finite-dimensional Lie 
subgroup of \( \DiffGroup(M, \omega) \) consisting of isometries 
of \( g_j \).
\item For every \( j \in \SectionSpaceAbb{I} \), the isotropy 
representation of \( \DiffGroup(M, \omega)_j \) on 
\( \TBundle_j \SectionSpaceAbb{J} \) is Hamiltonian with momentum 
map given by
\begin{equation}
\label{eq:kaehler:momentumMapIsotropyRepresentation}
\widehat{\SectionMapAbb{J}_j}(A) = -\frac{1}{8} \pr_j 
\left(X_{\tr_\omega(\dif \alpha)} + 
\omega^\sharp \bigl((j_0 \alpha)^h\bigr)\right),
\end{equation}
where \( \alpha(Y) \defeq \tr(A j \nabla_Y A) - 
2 \tr\bigl(\nabla (A j A) Y\bigr) \) for a torsion-free 
connection \( \nabla \) preserving the volume form and 
\( \pr_j \) is the orthogonal projection onto 
\( \VectorFieldSpace(M, \omega)_j \).
\item \label{i:kaehler:adjointLieDerivative}
For every \( X \in \VectorFieldSpace(M, \omega) \), 
the adjoint of \( \difLie_X: \VectorFieldSpace(M, \omega) \to 
\VectorFieldSpace(M, \omega) \) with respect to 
\( \dualPairDot_{j_0} \) is
\begin{equation}
\label{eq:kaehler:adjointLieDerivative}
X_F + \xi^h \mapsto - \difLie_X X_F + \frac{1}{2} (F j_0 X)^h,
\end{equation}
for \( F \in \sFunctionSpace_0(M) \) and  
\[ \xi^h \in \mathfrak{har}_{j_0}(M) \defeq \{\zeta \in  
\mathfrak{X}(M, \omega) \mid 
\omega^\flat(\zeta ) \text{ is a } g_{j_0}\text{-harmonic } 
1\text{-form}\}. 
\]
Moreover,
\begin{equation}\begin{split}
	\label{eq:kahler:adjointLieDerivative:pairing}
	\dualPair*{\difLie_X (X_F + \xi^h)}{X_G + \eta^h}_{j_0}
		&= - \dualPair*{X_F + \xi^h}{\difLie_X(X_G + \eta^h)}_{j_0}\\
		&\qquad+ \frac{1}{2} \dualPair*{(Fj_0 X)^h}{\eta^h}_{j_0}\\
		&\qquad+ \frac{1}{2} \dualPair*{\xi^h}{(Gj_0 X)^h}_{j_0}.
\end{split}\end{equation}
This shows that \( \dualPairDot_{j_0} \) is 
not invariant under the Lie derivative by vector fields in
\(\VectorFieldSpace(M, \omega) \). However,  
\( \dualPairDot_{j_0} \) is invariant under the Lie 
derivative by elements of the stabilizer 
\( \VectorFieldSpace(M, \omega)_{j_0} \).
\item The almost complex structure \( \SectionMapAbb{j} \) is 
equivariant with respect to the push-forward action of 
\( \DiffGroup(M, \omega) \).
		\qedhere
\end{thmenumerate}
\end{lemma}
\begin{proof}
The group of symplectomorphisms is a Fr\'echet Lie group 
by \parencite[Theorem~43.12]{KrieglMichor1997} and the 
automorphism group of an almost complex structure is a 
finite-dimensional Lie group by 
\parencite[Corollary~I.4.2]{Kobayashi1972}.

Let \( \nabla \) be a torsion-free connection-preserving the 
volume form. Using \( \difLie_X A = 
\nabla_X A + \commutator{A}{\nabla X} \), we find
\begin{equation}
\tr(A j \difLie_X A) 
= \tr(A j \nabla_X A) + 2 \tr(A j A \nabla X)
= \alpha(X) + 2 \tr\bigl(\nabla(A j A X)\bigr).
\end{equation}
On the other hand, \( \tr(\nabla Y) \) is the divergence of the 
vector field \( Y \) so that upon integration over \( M \), 
the last term vanishes and we obtain
\begin{equation}\begin{split}
\dualPair{\widehat{\SectionMapAbb{J}_j}(A)}{X}
&\stackrel{\eqref{j_hat_momentum}}{=} 
\frac{1}{2} \Omega_j(A, -\difLie_X A) \\
&\, \, \, = - \frac{1}{8} \int_M \tr(A j \difLie_X A) \, \mu_\omega \\
&\, \, \, = - \frac{1}{8} \int_M \alpha(X) \, \mu_\omega \\
&\, \, \, = - \frac{1}{8} \int_M \alpha \wedge (X \contr \mu_\omega) \\
&\,\stackrel{\eqref{eq:kaehler:pairing}}{=} - \frac{1}{8} 
\kappa\bigl(\equivClass{\alpha \wedge \omega^{n-1}}, X\bigr) \\
&\stackrel{\eqref{eq:kaehler:projectionDualSymp}}{=} -\frac{1}{8} 
\dualPair*{X_{\tr_\omega(\dif \alpha)} + 
\omega^\sharp \bigl((j_0 \alpha)^h\bigr)}{X}_{j_0}. \\
\end{split}\end{equation}
From this identity we directly read 
off~\eqref{eq:kaehler:momentumMapIsotropyRepresentation}.

Let \( X \in \VectorFieldSpace(M, \omega) \), 
\( F, G \in \sFunctionSpace_0(M) \) and 
\( \xi^h, \eta^h \in \mathfrak{har}_{j_0}(M) \).
Then we find
\begin{equation}\begin{split}
&\dualPair*{\difLie_X (X_F + \xi^h)}{X_G + \eta^h}_{j_0}\\
&\qquad= \dualPair*{X_{\omega(X, X_F + \xi^h)}}{X_G + \eta^h}_{j_0}
			\\
&\qquad= \frac{1}{2} \int_M \omega(X, X_F + \xi^h) G \, \mu_\omega
			\\
&\qquad= \frac{1}{2} \int_M \left(\dif F (X) G - 
G \omega^\flat(\xi^h)(X)\right) \mu_\omega
			\\
&\qquad= \frac{1}{2} \int_M - \dif G (X) F \mu_\omega - 
G \omega^\flat(\xi^h) \wedge (X \contr \mu_\omega)
			\\
&\qquad= - \dualPair*{X_F + \xi^h}{\difLie_X X_G}_{j_0} - 
\frac{1}{2} \int_M G \omega^\flat(\xi^h) 
\wedge *_{{g_j}_0} g^\flat_{j_0}(X)
			\\
&\qquad= - \dualPair*{X_F + \xi^h}{\difLie_X X_G}_{j_0} + 
\frac{1}{2} \int_M G \omega^\flat(\xi^h) 
\wedge *_{{g_j}_0} \omega^\flat(j_0X)
			\\
&\qquad= \dualPair*{X_F + \xi^h}{- \difLie_X X_G + 
\frac{1}{2} (G j_0 X)^h}_{j_0}.
\end{split}\end{equation}
This verifies~\eqref{eq:kaehler:adjointLieDerivative}.
Using this equation, we obtain
\begin{equation}\begin{split}
\dualPair*{\difLie_X (X_F + \xi^h)}{X_G + \eta^h}_{j_0}
&= -\dualPair*{X_F + \xi^h}{\difLie_X(X_G + \eta^h)}_{j_0}\\
&\qquad+ \frac{1}{2} \dualPair*{X_F + \xi^h}{\difLie_X \eta^h}_{j_0}\\
&\qquad+ \frac{1}{2} \dualPair*{X_F + \xi^h}{(Gj_0X)^h}_{j_0}.
\end{split}\end{equation}
Applying again~\eqref{eq:kaehler:adjointLieDerivative} on the 
second summand, we get~\eqref{eq:kahler:adjointLieDerivative:pairing}.

Now, if \( X \) is Killing, then \( 0 = \difLie_X \alpha = 
\dif (X \contr \alpha) \) for every harmonic 
\( 1 \)-form \( \alpha \). Applied to 
\( \alpha = \omega^\flat(\xi^h) \) in the above 
chain of equalities at step \( 3 \), we conclude 
that then \( \dualPair*{\difLie_X (X_F + \xi^h)}{X_G + \eta^h}_{j_0} 
= \frac{1}{2} \int_M (\dif F(X) G) \, \mu_\omega = 
-\dualPair*{X_F + \xi^h}{\difLie_X(X_G + \eta^h)}_{j_0} \), 
which shows that \( \dualPairDot_{j_0} \) is invariant under 
the adjoint action of Killing vector fields.
	
For every \( \phi \in \DiffGroup(M, \omega) \), we have
\begin{equation}
\phi_* \, \SectionMapAbb{j}_j(A) = \phi_* (A j) = 
(\phi_* A) (\phi_* j) = 
\SectionMapAbb{j}_{\phi_* j} (\phi_* A).
\end{equation}
Thus, \( \SectionMapAbb{j} \) is equivariant.
\end{proof}

Let us discuss the analogs of the Lichnerowicz and Calabi operators.
Relative to the splitting of the exact 
sequence~\eqref{eq:kaehler:symplecticVectorFieldsExactSequence}, every 
operator \( T: \VectorFieldSpace(M, \omega) \to 
\VectorFieldSpace(M, \omega) \) gives rise to operators 
\( T^{SS}: \sFunctionSpace(M) \to \sFunctionSpace(M) \), 
\( T^{HS}: \mathfrak{har}_{j_0}(M) \to \sFunctionSpace(M) \), 
\( T^{SH}: \sFunctionSpace(M) \to \mathfrak{har}_{j_0}(M) \) and 
\( T^{HH}: \mathfrak{har}_{j_0}(M) \to \mathfrak{har}_{j_0}(M) \). 
The defining equations for these operators are
\begin{equation}
T(X_f) = X_{T^{SS}(f)} + T^{SH}(f), 
\qquad T(X^h) = X_{T^{HS}\bigl(X^h\bigr)} 
+ T^{HH} X^h
\end{equation}
with \( f \in \sFunctionSpace(M) \) and 
\( X^h \in \mathfrak{har}_{j_0}(M) \), and where 
we keep identifying \( \deRCohomology^1(M) \) with 
\( g_{j_0} \)-harmonic \( 1 \)-forms.
Using this notation, we calculate the 
Lichnerowicz operator introduced in the general setting in~\eqref{L_m_L_mu}.
\begin{prop}
	\label{eq:kaehler:Loperators}
For a \( j_0 \)-extremal almost K\"ahler metric
\( j \in \SectionSpaceAbb{I} \), the operators 
\( L_j \xi = \tangent_j \SectionMapAbb{J} (j \, \difLie_\xi j) \) 
and \( Z_j \xi = -\tangent_j \SectionMapAbb{J} (\difLie_\xi j) \) 
on \( \VectorFieldSpace(M, \omega) \) are given by:
\begin{subequations}
\begin{equation}\label{eq:kaehler:Lj}\begin{split}
L_j^{SS}(f) &= - \frac{1}{2}\tr_\omega 
(\dif \tau^\nabla(j, j \difLie_{X_f} j))\\
L_j^{SH}(f) &= - \frac{1}{2} \omega^\sharp (
j_0 \tau^\nabla(j, j \difLie_{X_f} j))^h\\
L_j^{HS}(X^h) &= - \frac{1}{2}\tr_\omega 
(\dif \tau^\nabla(j, j \difLie_{X^h} j)) \\
L_j^{HH}(X^h) &= - \frac{1}{2} \omega^\sharp (j_0 \tau^\nabla(j, 
j \difLie_{X^h} j))^h,
\end{split}\end{equation}
\begin{equation}\label{eq:kaehler:Zj}\begin{split}
	Z_j^{SS}(f) &= \poisson{f}{\bar{S}_j}\\
	Z_j^{SH}(f) &= \omega^\sharp (j_0 (X_f \contr \mathrm{Ric}_j))^h\\
	Z_j^{HS}(X^h) &= \dif \bar{S}_j(X^h)\\
	Z_j^{HH}(X^h) &= \omega^\sharp (j_0 (X^h 
	\contr \mathrm{Ric}_j))^h
		\end{split}\end{equation}
	\end{subequations}
for \( f \in \sFunctionSpace(M) \) and every 
\( X^h \in \mathfrak{har}_{j_0}(M) \).
\end{prop}
\begin{proof}
These follow from direct computations.
\paragraph*{Calculation of \( Z_j \):}
By~\cref{prop:kaehler:variationJ}, \( \tangent_j J(j_0, \cdot)(A) \) 
equals \( - \frac{1}{2} \tau^\nabla(j, A) \) modulo an exact form.
On the other hand, from 
\parencite[Theorem 2.7]{Garcia-PradaSalamonTrautwein2018}, 
for every \( \xi \in \VectorFieldSpace(M, \omega) \), we have
\begin{equation}
\tau^\nabla(j, \difLie_\xi j) = 2 \xi \contr \mathrm{Ric}_j 
+ \dif \divergence(j\xi).
\end{equation}
Hence, viewing \( \SectionMapAbb{J} \) as a map into 
\( \DiffFormSpace^{2n-1} M \slash \dif \DiffFormSpace^{2n-2} M \), 
we have 
\begin{equation}
- \tangent_j \SectionMapAbb{J}(\difLie_\xi j) = 
\frac{1}{2} \tau^\nabla(j, \difLie_\xi j) \wedge 
\omega^{n-1} \textrm{ mod exact}
= (\xi \contr \mathrm{Ric}_j) \wedge \omega^{n-1} 
\textrm{ mod exact}.
\end{equation}
Thus, composing with the projection~\eqref{eq:kaehler:projectionDualSymp}, 
we obtain \( Z_j \xi = X_G + \omega^\sharp \beta^h \) for
\begin{equation}
G = \tr_\omega (\dif (\xi \contr \mathrm{Ric}_j))
= \tr_\omega (\difLie_\xi \mathrm{Ric}_j) 
= \difLie_\xi S_j = \omega(\xi, X_{S_j})
\end{equation}
and
\begin{equation}
\beta^h = (j_0 (\xi \contr \mathrm{Ric}_j))^h.
\end{equation}
Using these equations for \( \xi = X_f \) or \( \xi = X^h \) 
yields~\eqref{eq:kaehler:Zj}.

\paragraph*{Calculation of \( L_j \):}
Using a similar argument as above, we obtain  
\( L_j \xi = X_G + \omega^\sharp \beta^h \) for
\begin{equation}
G = - \frac{1}{2}\tr_\omega (\dif \tau^\nabla(j, j \difLie_\xi j))
\end{equation}
and
\begin{equation}
\beta^h = - \frac{1}{2} (j_0 \tau^\nabla(j, j \difLie_\xi j))^h.
\end{equation}
Using these equations for \( \xi = X_f \) or \( \xi = X^h \) 
yields~\eqref{eq:kaehler:Lj}.
\end{proof}

In the integrable case, the operator \( L_j^{SS}: \sFunctionSpace_0(M) 
\to \sFunctionSpace_0(M) \) recovers the classical Lichnerowicz 
operator, and a few different ways of writing it down are known 
in the literature, see \eg \parencite{Gauduchon2017}.
For non-integrable \( j \), similar expressions for \( L_j^{SS} \) 
are obtained in \parencite{Vernier2020,HeZheng2023}.
In both cases, one concludes from these explicit expressions 
that \( L_j^{SS} \) is a \( 4 \)-order elliptic differential 
operator. In particular, \( L_j \) is a Fredholm operator.

Following the general procedure, 
\cf~\eqref{eq:decomposition:calabiOperatorDef}, for 
every \( j \in \SectionMapAbb{I} \), the \emphDef{Calabi operators} 
\( C^\pm_j: \VectorFieldSpace(M, \omega)_\C \to 
\VectorFieldSpace(M, \omega)_\C \) are defined by
\begin{equation}
	C^\pm_j = L_j \pm \I Z_j.
\end{equation}
Recall that a real vector field \( X \) is called holomorphic 
if \( \difLie_X j = 0 \).
Using a similar argument as above, one sees that \( C^\pm_j \) 
are Fredholm operators.
In particular, their kernels are finite-dimensional.

\begin{coro}
\label{prop:kaehler:complexStabilizerToHolomorphic}
For every \( j \in \SectionMapAbb{I} \), the kernel of \( C^+_j \) 
coincides with the stabilizer \( (\VectorFieldSpace(M, \omega)_\C)_j \) 
under the complexified action. If \( j \) is integrable, then the 
map \( \VectorFieldSpace(M, \omega)_\C \ni \xi + 
\I \eta \mapsto \xi - j \eta \in \VectorFieldSpace(M) \) 
restricts to a surjection from \( (\VectorFieldSpace(M, \omega)_\C)_j \) 
onto the space of real holomorphic vector fields and it has kernel 
\( \mathfrak{har}_{j}(M) \).
\end{coro}
We do not know whether \( (\VectorFieldSpace(M, \omega)_\C)_j \) 
is a Lie subalgebra of \( \VectorFieldSpace(M, \omega)_\C \).
\begin{proof}
The first statement follows directly from 
\cref{i:normedsquared:kernelsIdentified}.
For the second statement, we observe that \( \xi + \I \eta \) 
is in the stabilizer \( (\VectorFieldSpace(M, \omega)_\C)_j \) 
if and only if \( 0 = \xi \ldot j + \SectionMapAbb{j}_j (\eta \ldot j) 
= - \difLie_\xi j  + j \difLie_\eta j = 
- \difLie_\xi j + \difLie_{j \eta} j \),
where the last equality uses the integrability of \( j \), 
see, \eg, \parencite[Lemma~1.1.1]{Gauduchon2017}.
In other words, \( \xi - j \eta \) is a real holomorphic vector 
field. Conversely, by \parencite[Lemma~2.1.1]{Gauduchon2017}, 
every real holomorphic vector field \( X \) on a compact K\"ahler 
manifold can be uniquely written as the sum \( X = j X^h + j X_f 
+ X_g \) for \( X^h \in \mathfrak{har}_j(M) \) and 
\( f,g \in \sFunctionSpace_0(M) \). This shows that the map 
\( \xi + \I \eta \mapsto \xi - j \eta \) is surjective.
Finally, if \( \xi = j \eta \) with \( \xi, \eta \in 
\VectorFieldSpace(M, \omega) \), then both \( \omega^\flat(\xi) \) 
and \( \omega^\flat(j\xi) = j \omega^\flat(\xi) \) are closed.
By \parencite[Proposition~2.3.1]{Gauduchon2017}, this is equivalent 
to \( \omega^\flat(\xi) \) being harmonic.
\end{proof}
As a direct application of \cref{prop:normedsquared:decompositionComplexStab}, 
we obtain the following result.

\begin{thm}
	\label{prop:kaehler:decompositionComplexStabRelativeExtremal}
Let \( (M, \omega) \) be a compact symplectic manifold and 
\( j_0 \) a compatible almost complex structure. For every 
\( j_0 \)-extremal almost complex structure \( j \) satisfying 
\( \SectionSpaceAbb{J}(j) \in \VectorFieldSpace(M, \omega)_{j_0} \), 
the following decomposition holds:
\begin{equation}
(\VectorFieldSpace(M, \omega)_\C)_j = \LieA{c} \oplus 
\bigoplus_{\lambda \neq 0} \LieA{k}_\lambda,
\end{equation}
where:
\begin{thmenumerate}
\item \( \LieA{c} \) is the Lie subalgebra of 
\( (\VectorFieldSpace(M, \omega)_\C)_j \) consisting of all 
elements that commute with \( \SectionSpaceAbb{J}(j) \);
\item \( \C \SectionSpaceAbb{J}(j) \subseteq \LieA{c} \); 
\( \mathfrak{har}_j \subseteq \LieA{c} \);
\item \( \LieA{k}_\lambda \) are eigenspaces of 
\( 2 \I \difLie_{\SectionSpaceAbb{J}(j)} \) with eigenvalue 
\( \lambda \in \R \) {\rm(}with the convention that 
\( \LieA{k}_\lambda = \set{0} \) if \( \lambda \) is not an 
eigenvalue{\rm)}; in particular, $\mathfrak{c} = \mathfrak{k}_0$;
\item \( \commutator{\LieA{k}_\lambda}{\LieA{k}_\mu} 
\intersect (\VectorFieldSpace(M, \omega)_\C)_j \subseteq 
\LieA{k}_{\lambda + \mu} \) if $\lambda + \mu$ is an eigenvalue 
of \( 2 \I \difLie_{\SectionSpaceAbb{J}(j)} \); otherwise 
\( \commutator{\LieA{k}_\lambda}{\LieA{k}_\mu} \intersect 
(\VectorFieldSpace(M, \omega)_\C)_j = 0 \).
\qedhere
\end{thmenumerate}
\end{thm}
\begin{proof}
The only statement that does not follow directly from 
\cref{prop:normedsquared:decompositionComplexStab} is the 
inclusion \( \mathfrak{har}_j \subseteq \LieA{c} \).
But this follows from the fact that the Lie derivative with 
respect to a symplectic vector field commutes with the musical 
isomorphism \( \omega^\flat \): for every 
\( X^h \in \mathfrak{har}_j \), we have
\begin{equation}
\omega^\flat \difLie_{\SectionSpaceAbb{J}(j)} X^h
= \difLie_{\SectionSpaceAbb{J}(j)} (\omega^\flat X^h)
= 0,
\end{equation}
where the last equality follows from the fact that 
\( \omega^\flat X^h \) is a \( g_j \)-harmonic form 
and \( \SectionSpaceAbb{J}(j) \) is a \( g_j \)-Killing 
vector field.
\end{proof}
The assumption that \( \SectionSpaceAbb{J}(j) \in 
\VectorFieldSpace(M, \omega)_{j_0} \) is not essential 
and only serves to ensure that \( 2 \I \difLie_{\SectionSpaceAbb{J}(j)} \) 
is symmetric and thus diagonalizable, 
\cf \cref{i:kaehler:adjointLieDerivative}.
Without this assumption a similar statement 
holds using generalized eigenspaces; see 
\cref{rem:decomposition:decompositionComplexStab:nonInvariantKappa}.

Similarly, for extremal metrics, we obtain the following theorem 
generalizing the classical result of \textcite{Calabi1985} which holds
in the integrable case.

\begin{thm}
	\label{prop:kaehler:decompositionComplexStabAndHol}
Let \( (M, \omega) \) be a compact symplectic manifold. For every 
extremal almost complex structure \( j \in \SectionMapAbb{I} \), 
the following decomposition holds:
\begin{equation}
\label{eq:kaehler:extremal:decompositionComplexStab}
(\VectorFieldSpace(M, \omega)_\C)_j = \LieA{c} \oplus 
\bigoplus_{\lambda \neq 0} \LieA{k}_\lambda,
\end{equation}
where:
\begin{thmenumerate}
\item
\( \LieA{c} \) is the subset of 
\( (\VectorFieldSpace(M, \omega)_\C)_j \) consisting of 
all elements that commute with \( X_{S_j} \);
\item \( \C X_{S_j} \subseteq \LieA{c} \); 
\( \mathfrak{har}_j \subseteq \LieA{c} \);
\item
\( \LieA{k}_\lambda \) are eigenspaces of \( 2 \I \difLie_{X_{S_j}} \) 
with eigenvalue \( \lambda \in \R \) {\rm(}with the convention that 
\( \LieA{k}_\lambda = \set{0} \) if \( \lambda \) is not an 
eigenvalue{\rm)}; in particular, $\mathfrak{c} = \mathfrak{k}_0$;
\item
\( \commutator{\LieA{k}_\lambda}{\LieA{k}_\mu} \subseteq 
\LieA{k}_{\lambda + \mu} \intersect 
(\VectorFieldSpace(M, \omega)_\C)_j \) if 
\( \lambda + \mu \) is an eigenvalue 
of \( 2 \I \difLie_{X_{S_j}} \); otherwise 
\( \commutator{\LieA{k}_\lambda}{\LieA{k}_\mu} 
\intersect (\VectorFieldSpace(M, \omega)_\C)_j = 0 \).
\end{thmenumerate}
Moreover, if \( j \) is integrable, then the Lie algebra 
\( \LieA{h}(M, j) \) of real holomorphic vector fields 
admits the following decomposition:
\begin{equation}
\label{eq:kaehler:extremal:decompositionHolomorphic}
\LieA{h}(M, j) = \mathrlap{\overbrace{\phantom{\LieA{a}(M, g_j) 
\oplus \LieA{k}_{\textnormal{ham}}(M,g_j)}}^{\LieA{k}(M,g_j)}} 
\LieA{a}(M,j) \oplus \underbrace{\LieA{k}_{\textnormal{ham}}(M,g_j) 
\oplus j \LieA{k}_{\textnormal{ham}}(M, g_j) \oplus 
\bigoplus_{\lambda \neq 0} \LieA{h}_\lambda(M,j)}_{\LieA{h}_{
\textnormal{red}}(M, j)},
\end{equation}
where \( \LieA{a}(M, g_j) \) is the complex Abelian Lie 
subalgebra of \( \LieA{h}(M, j) \) consisting of vector 
fields that are parallel with respect to the Levi-Civita 
connection of \( g_j \), \( \LieA{k}(M, g_j) \) is the Lie 
algebra of Killing vector fields, 
\( \LieA{k}_{\textnormal{ham}}(M, g_j) \) the subalgebra 
of Hamiltonian Killing vector fields, and 
\( \LieA{h}_{\textnormal{red}}(M, j) \) is the Lie 
algebra of the reduced automorphism group 
(see \eg \parencite[Section~2.4]{Gauduchon2017}), and 
\( \LieA{h}_\lambda(M,j) \) are 
\( \lambda \)-eigenspaces of \( - 2 j \difLie_{X_{S_j}} \).
\end{thm}
\begin{proof}
This statement does not directly follow from 
\cref{prop:normedsquared:decompositionComplexStab} applied to 
the action of \( \HamDiffGroup(M, \omega) \), since this would 
only yield a decomposition of \( (\HamVectorFields(M, \omega)_\C)_j \).
Instead, we use the fact that 
\( \SectionMapAbb{J}_{\HamDiffGroup} (j) \) is an 
element of the stabilizer \( \VectorFieldSpace(M, \omega)_j \) 
as \( j \) is extremal. Then the first part of the statement follows 
from \cref{prop:decomposition:decompositionComplexStab} relative to 
the \( \DiffGroup(M, \omega) \) action, applied to the one-dimensional 
subalgebra \( \LieA{t} \subseteq \VectorFieldSpace(M, \omega)_j \) 
spanned by \( \SectionMapAbb{J}_{\HamDiffGroup} (j) = X_{S_j} \).

The image of the 
decomposition~\eqref{eq:kaehler:extremal:decompositionComplexStab} 
under the map 
\( \VectorFieldSpace(M, \omega)_\C \ni \xi + \I \eta \mapsto 
\xi - j \eta \in \VectorFieldSpace(M) \) yields the decomposition 
\( \LieA{h}(M, j) = \bigoplus_{\lambda} \LieA{h}_\lambda(M,j) \), 
\cf \cref{prop:kaehler:complexStabilizerToHolomorphic} (this uses 
the fact that the kernel of this map is \( \mathfrak{har}_j \), 
which is completely included in \( \LieA{c} \)).
A direct calculation shows that \( 2 \I \difLie_{X_{S_j}} \) under 
this map takes the form \( -2 j \difLie_{X_{S_j}} \), which identifies 
\( \LieA{h}_\lambda(M,j) \) as eigenspaces of \( -2 j \difLie_{X_{S_j}} \).
Finally, the further decomposition of the zero eigenspace and the 
identification of \( \LieA{k}(M,g_j) \) and 
\( \LieA{h}_{\textnormal{red}}(M, j) \) in this decomposition 
are standard; see \eg \parencite[Theorem~3.4.1]{Gauduchon2017}.
This finishes the proof of the 
decomposition~\eqref{eq:kaehler:extremal:decompositionHolomorphic}.
\end{proof}

\begin{thm}
Let \( (M, \omega) \) be a compact symplectic manifold.
The Hessian of the Calabi functional 
\( \norm{\SectionMapAbb{J}_\HamDiffGroup}^2 \) at an 
extremal \( j \in \SectionMapAbb{I} \) is given by
\begin{equation}
\frac{1}{2} \Hessian_j \norm{\SectionMapAbb{J}_\HamDiffGroup}^2 
(\zeta \ldot j, \gamma \ldot j) = 
\Re \, \dualPair{\zeta}{C^+_j C^-_j \smallMatrix{0 & 0\\0 & 1} \gamma}_{\C}
\end{equation}
for \( \zeta, \gamma \in \HamVectorFields(M, \omega)_\C \).
Moreover, the restriction of \( \Hessian_j 
\norm{\SectionMapAbb{J}_\HamDiffGroup}^2 \) to 
\( \VectorFieldSpace(M, \omega)_\C \ldot j \subseteq 
\TBundle_j \SectionSpaceAbb{J} \) is positive semi-definite.

Similarly, for every \( j_0 \)-extremal \( j \in \SectionMapAbb{I} \) 
(with respect to a given almost complex structure 
\( j_0 \in \SectionMapAbb{I} \)), the Hessian of 
\( \norm{\SectionMapAbb{J}}_{j_0}^2 \) at \( j \) is given by
\begin{equation}
\frac{1}{2} \Hessian_j \norm{\SectionMapAbb{J}}^2_{j_0} 
(\zeta \ldot j, \gamma \ldot j) = 
\Re  \, \dualPair{\zeta}{C^+_j R_j \gamma}_{j_0, \C},
\end{equation}
for \( \zeta, \gamma \in \VectorFieldSpace(M, \omega)_\C \) and
\begin{equation}
R_j = C^-_j \Matrix{0 & 0 \\ 0 & 1} + 
\I \, \bigl(\difLie_{\SectionMapAbb{J}(j)} + Z_j\bigr)
\end{equation}
where \( Z_j \) is calculated in~\eqref{eq:kaehler:Zj}.
\end{thm}
\begin{proof}
The expression for the Hessian follows directly from 
\cref{prop:normedsquared:hessianSummary} in both cases.
The fact that the restriction of 
\( \Hessian_j \norm{\SectionMapAbb{J}_\HamDiffGroup}^2 \) to 
\( \VectorFieldSpace(M, \omega)_\C \ldot j \subseteq 
\TBundle_j \SectionSpaceAbb{J} \) is positive semi-definite 
is a direct consequence of 
\cref{prop:normedsquared:hessianPositiveDefAlongComplexOrbit}.
In fact, the only missing assumption to verify is that the 
Calabi operators \( C^\pm_j \) are essentially self-adjoint. 
But this is clearly the case as these operators are elliptic 
operators on \( \sFunctionSpace(M)_\C \isomorph 
\sFunctionSpace(M, \C) \).
\end{proof}

The first part concerning the Hessian of the Calabi energy 
recovers the classical work of \textcite[Theorem~2]{Calabi1985} 
in the integrable case (in which case, the vector space  
\( \VectorFieldSpace(M, \omega)_\C \ldot j \) is identified with
 the tangent 
space to the space of K\"ahler forms in a given cohomology class 
up to automorphisms; see 
\parencites[p.~408f]{Donaldson1997}[Proposition~9.1.1]{Gauduchon2017}) 
and the recent result of 
\textcite[Theorem~1.1]{HeZheng2023} in the non-integrable case.
The second part of the theorem is thus a natural generalization 
of these insights to the case of \( j_0 \)-extremal metrics.

\begin{remark}[Mabuchi and K\"ahler--Ricci solitons]
In \parencite{Donaldson2015a}, Donaldson introduced another symplectic 
form on the space \( \SectionSpaceAbb{I} \) on a Fano manifold.
This new symplectic form is induced from the space of differential 
\( n \)-forms with values in the prequantum bundle \( (L, \theta) \) 
over \( M \), essentially via the Pl\"ucker embedding.
With respect to this symplectic form, the action of the group 
\( \AutGroup(L, \theta) \) is Hamiltonian whose momentum map is 
the logarithm of the Ricci potential. This momentum map is equivariant.
Zeros of the momentum map are precisely the K\"ahler--Einstein metrics.
The norm-squared of the momentum map yields the Ricci--Calabi 
functional, whose critical points are generalized K\"ahler--Einstein 
metrics, also known as Mabuchi solitons after \parencite{Mabuchi2001}.
As an application of our general results in 
\cref{prop:normedsquared:decompositionComplexStab}, we recover the 
Matsushima type decomposition theorem for holomorphic vector fields 
in the presence of generalized K\"ahler--Einstein metrics of 
\parencite[Theorem~4.1]{Mabuchi2001}.
Moreover, from 
\cref{prop:normedsquared:hessianPositiveDefAlongComplexOrbit,prop:normedsquared:hessianSummary}, we recover the Hessian of the Ricci--Calabi 
functional which has been calculated in 
\parencite[Theorem~1.1]{Nakamura2019a}.

As an alternative to the norm-squared of the momentum map, one 
may also consider the composition of the momentum map with a 
certain convex function on the Lie algebra \( \sFunctionSpace(M) \) 
of \( \AutGroup(L, \theta) \).
The resulting functional on \( \SectionSpaceAbb{I} \) is the 
\( \SectionMapAbb{H} \)-functional introduced in \parencite{He2016}, 
whose critical points are K\"ahler--Ricci solitons.
We expect that our results can be used to study the 
\( \SectionMapAbb{H} \)-functional as well.
In fact, the results about the Hessian of the 
\( \SectionMapAbb{H} \)-functional and the decomposition of 
holomorphic vector fields in the presence of K\"ahler--Ricci 
solitons obtained in \parencite{Fong2016,Nakamura2019} and 
\parencite[Theorem~A]{TianZhu2000}, respectively, should follow 
from an extension of our results to allow arbitrary convex functions 
on the Lie algebra along the lines of the finite-dimensional/formal 
picture of \parencite{LeeSturmWang2022}.
\end{remark}

\begin{remark}[Coupled K\"ahler--Einstein metrics]
A different application of our results is to the coupled 
K\"ahler--Einstein equations introduced by 
\textcite{HultgrenWittNystroem2019}.
Following \parencite{DatarPingali2019}, this setting fits into 
our infinite-dimensional symplectic framework.
In this case, the Hessian and the Matsushima-type decomposition 
recover the recent results of \parencite{Nakamura2023}.
\end{remark}

\begin{remark}[\( f \)-extremal K\"ahler metrics]
Let \( f \) be a positive function\footnotemark{} on the symplectic 
manifold \( (M, \omega) \), and denote its Hamiltonian vector field 
by \( K = X_f \)\footnotetext{We are skipping over some technical 
points here, like the assumption that \( K \) has to lie in a 
certain torus.}. \Textcite{ApostolovMaschler2019} defined on the 
space \( \SectionSpaceAbb{I}_K(M, \omega) \) of \( K \)-invariant 
(almost) complex structures on \( M \) a \( f \)-deformed version 
of the symplectic form~\eqref{eq:kaehler:symplecticForm} as follows:
\begin{equation}
	\Omega^f_j (A, B) = \frac{1}{4} \int_M \tr (A \, j \, B) \, 
	\frac{\mu_\omega}{f^{2n-1}}.
\end{equation}
They showed that the momentum map for the action of 
\( \HamDiffGroup_K(M, \omega) \) on 
\( \SectionSpaceAbb{I}_K(M, \omega) \) is given by 
assigning to \( j \) the Hermitian scalar curvature 
of \( f^{-2} g_j \).
Thus, zeros of the momentum map correspond to conformally 
K\"ahler--Einstein metrics (cKEM) and the norm-squared 
of the momentum map is the \( f \)-weighted Calabi functional, whose 
critical points are called \( f \)-extremal K\"ahler metrics.
The Calabi program for \( f \)-extremal K\"ahler metrics has 
been initiated by \textcite{FutakiOno2017,Lahdili2019}.
Naturally, this setting fits into our infinite-dimensional 
symplectic framework and we can recover these results using 
our general 
\cref{prop:normedsquared:decompositionComplexStab,prop:normedsquared:hessianPositiveDefAlongComplexOrbit}.
Moreover, based on our discussion above, it would be interesting 
to study the momentum map of all symplectic diffeomorphisms 
preserving \( K \) (and not just the subgroup of Hamiltonian 
diffeomorphisms).
\end{remark}

\begin{remark}[Sasakian geometry]
In the odd-dimensional counterpart to K\"ahler geometry, 
Sasakian metrics and their non-integrable pendant $K$-contact 
structures are another important class of geometries that can 
be studied using our results.
\Textcite{He2014,LejmiUpmeier2015} have shown that the space 
of $K$-contact structures on a compact contact manifold is 
an infinite-dimensional symplectic manifold and that the 
action of the group of strict contactomorphisms is Hamiltonian 
with momentum map given by the transverse Hermitian scalar 
curvature. The critical points of the norm-squared of the 
momentum map have been studied in \parencite{BoyerGalickiSimanca2008} 
and are called extremal Sasakian metrics.
We expect that our results can be used to study the Hessian of 
the norm-squared of the momentum map and the decomposition of 
the complexified stabilizer of a $K$-contact structure.
In particular, the decomposition theorem 
\parencite[Theorem~11.3.1]{Boyer2008} of the space of transverse 
holomorphic vector fields in the presence of an extremal 
Sasakian metric should directly follow from  
\cref{prop:normedsquared:decompositionComplexStab}.
Moreover, it would be interesting to study the action of the 
whole group of contactomorphisms (and not just the subgroup of strict 
contactomorphisms) on the space of $K$-contact structures.
In parallel to our discussion of the K\"ahler case, one would 
expect that the momentum map is then non-equivariant and one 
obtains a natural central extension of the group of contactomorphisms.
\end{remark}

\section{Application: Symplectic connections}
\label{sec:symplecticConnections}

\subsection{Momentum map for the action of 
\texorpdfstring{$\DiffGroup(M, \omega)$}{Diff(M, w)}}
\label{sec:symplecticConnections:momentumMap}
First, we briefly review the necessary background on symplectic connections 
summarizing definitions and conventions following \parencite{Tondeur1961, 
Hess1980, MarsdenRatiuEtAl1991, CahenGutt2005}.
We make heavy use of the Penrose notation, which is reviewed in \cref{sec:notation}.
An affine connection \( \nabla \) on a symplectic manifold \( (M, \omega) \) 
is a \emphDef{symplectic connection}, if it is torsion-free and satisfies 
\( \nabla \omega =0 \), \ie, \( X[\omega(Y,Z)] =\omega(\nabla_X Y, Z) + 
\omega(Y,\nabla_X Z) \), for all \( X, Y, Z \in  \VectorFieldSpace(M) \).
This condition is equivalent to the parallel transport operator being a 
symplectic isomorphism between the tangent spaces to \( M \).
In contrast to the Levi-Civita connection on a Riemannian manifold, there does 
not exist a \emph{unique}\footnotemark{} symplectic connection on a given 
symplectic manifold.
\footnotetext{The non-uniqueness of torsion free symplectic connections cannot 
be improved even if \( M = \CotBundle Q \), endowed with its standard exact 
symplectic form and \(Q\) is Riemannian.
Then \( \TBundle Q\) has a naturally induced Riemannian metric.
Pull back this metric to \( \CotBundle Q\) using the given Riemannian metric 
on \( Q \) to endow \( \CotBundle Q \) with a Riemannian metric.
So, we have now a Levi-Civita connection  on \( \CotBundle Q \).
It turns out that it is symplectic if and only if the given Riemannian 
metric on \( Q \) is flat.}
If \( \nabla^1 \) and \( \nabla^2 \) are symplectic connections on 
\( (M,\omega) \), then
\begin{equation}
\label{two_nablas}
	\nabla^2_i X^k = \nabla^1_i X^k + \tensor{A}{_{ij}^k} X^j
\end{equation}
for some tensor \( \tensor{A}{_{ij}^k} \) such that 
\( \tensor{A}{_i_j_k} \defeq \tensor{A}{_{ij}^l}\omega_{lk}\)
is symmetric in all indices, see \parencite{CahenGutt2005}.
The Penrose index notation in~\eqref{two_nablas} stands for
the intrinsic formula \(\nabla^2_Y X = \nabla^1_Y X + A(Y, X, \cdot) \).
We abbreviate~\eqref{two_nablas} by writing 
\( \nabla^2 = \nabla^1 + A \).
The important conclusion of these considerations is that the space 
\(\ConnSpace_\omega(M) \) of symplectic connections on the symplectic 
manifold \( (M, \omega) \) is an affine space whose linear model space is 
isomorphic to the space \( \SymTensorFieldSpace_3(M) \) of symmetric 
covariant \( 3 \)-tensor fields on \(M\).
In particular, \( \ConnSpace_\omega(M) \) is always non-empty.

\textit{In the following we assume \( M \) to be compact.} We endow 
\( \ConnSpace_\omega(M) \) with its natural \( \sFunctionSpace \)-Fr\'echet 
topology. According to \parencite{CahenGutt2005}, the space 
\( \ConnSpace_\omega(M) \) carries a natural affine (weak) symplectic 
structure \( \Omega \) defined by
\begin{equation}
\label{symplectic_form_connections}
\Omega_\nabla(A, B) = \int_M \tensor{A}{_i_j_k} \tensor{B}{^i^j^k} \mu_\omega,
\end{equation}
where \( \nabla \in \ConnSpace_\omega(M) \), \( A, B \in 
\SymTensorFieldSpace_3(M) \), and \( \mu_\omega = \frac{\omega^n}{n!} \) 
is the Liouville volume form. Note that \( \Omega_\nabla \) does not depend 
on \(\nabla \). The Fr\'echet Lie group \( \DiffGroup(M, \omega) \) of 
symplectomorphisms of \( (M, \omega) \) acts on the left on 
\( \ConnSpace_\omega(M) \) by push-forward according to
\begin{equation}
\label{eq:symplecticConnections:action}
(\phi \cdot \nabla)_X Y \defeq \phi_* \left(\nabla_{\phi^{-1}_* X} (\phi^{-1}_* Y)
\right)
\end{equation}
for \(\phi \in \DiffGroup(M, \omega)\) and \(\nabla \in \ConnSpace_\omega(M)\).
This action is clearly affine and the induced linear action is given by the 
natural left action
\begin{equation}
(\phi \cdot A)(X, Y, Z) = A\bigl(\phi^{-1}_* X, \phi^{-1}_* Y, \phi^{-1}_*Z
\bigr)
\end{equation}
of \( \DiffGroup(M, \omega) \) on \( \SymTensorFieldSpace_3(M) \).
Using this expression for the linear action, it is straightforward to verify 
that the \( \DiffGroup(M, \omega) \)-action on \( \ConnSpace_\omega(M) \) 
preserves the symplectic form \( \Omega \). The infinitesimal action of a 
symplectic vector field \( \xi \in \VectorFieldSpace(M, \omega) \) 
on \( \ConnSpace_\omega(M) \) is given by
\begin{equation}
	\label{eq:symplecticConnections:infAction}
	\tensor{\bigl(\xi \ldot \nabla\bigr)}{_i_j^k}
		= - \tensor{\bigl(\difLie_\xi \nabla\bigr)}{_{ij}^k}
		= - \nabla_i \nabla_j \xi^k - \tensor{R}{_l_i_j^k} \xi^l.
\end{equation}
Similarly, the infinitesimal action of \( \VectorFieldSpace(M, \omega) \) 
on \( \SymTensorFieldSpace_3(M) \) takes the form
\begin{equation}
	\label{eq:symplecticConnections:infActionLinear}
	\tensor{\bigl(\xi \ldot A\bigr)}{_i_j^k}
		= - \tensor{\bigl(\difLie_\xi A\bigr)}{_{ij}^k}
		= - \bigl(\xi^p \nabla_p \tensor{A}{_{ij}^k} + \tensor{A}{_{pj}^k} 
		\nabla_i \xi^p + \tensor{A}{_{ip}^k} \nabla_j \xi^p - 
		\tensor{A}{_i_j^q} \nabla_q \xi^k\bigr).
\end{equation}

As we are working in an infinite-dimensional setting, we have to pay 
attention to functional analytic problems. We will be brief here and refer 
the reader to \parencite{DiezRatiuAutomorphisms,DiezThesis} for background 
information and further technical details. 
For the construction of the momentum map, we need to clarify what we mean by 
the dual space of \( \VectorFieldSpace(M, \omega) \). Note that the map 
\( \xi \mapsto \xi \contr \omega \) identifies 
\( \VectorFieldSpace(M, \omega) \) with the space of closed \( 1 \)-forms 
on \( M \).
This suggests the choice \( {\VectorFieldSpace(M, \omega)}^* \defeq 
\DiffFormSpace^{2n-1}(M) \slash \dif \DiffFormSpace^{2n-2}(M) \) for the dual 
space of \( \VectorFieldSpace(M, \omega) \) relative to the pairing\footnotemark{}
\begin{equation}
	\label{eq:symplecticConnections:pairing}
	\kappa\bigl(\equivClass{\alpha},\xi\bigr) 
	= \frac{1}{(n-1)!}\int_M \alpha \wedge (\xi \contr\omega),
\end{equation}
where \( \equivClass{\alpha} \in \DiffFormSpace^{2n-1}(M) \slash \dif 
\DiffFormSpace^{2n-2}(M) \) and \( \xi \in \VectorFieldSpace(M, \omega) \).
\footnotetext{This is the same pairing as~\eqref{eq:kaehler:pairing} 
used above in the K\"ahler setting.}

The following proposition shows that the action has a momentum map.
\begin{prop}
	\label{prop:symplecticConnections:momentumMapStrange}
For \( \nabla \in \ConnSpace_\omega(M) \) and \( A \in 
\SymTensorFieldSpace_3(M) \), let \( J(\nabla + A) \in \DiffFormSpace^1(M) \) 
be given by
\begin{equation}\begin{split}
\label{eq:symplecticConnections:momentumMapStrange:oneForm}
J(\nabla + A)_p
= - \nabla_j \nabla_i \tensor{A}{^i^j_p} + R_{pijk}A^{ijk} + 
\frac{1}{2} (\nabla_p A_{ijk}) A^{ijk} - 
\frac{3}{2} \nabla_i \bigl(A^{ijk}  A_{pjk} \bigr),
\end{split}\end{equation}
where \(R\) is the curvature operator of 
\(\nabla \) and \( R_{pijk} = \tensor{R}{_{pij}^{s}} \omega_{sk} \).
For each \( \nabla \in \ConnSpace_\omega(M) \), the map
\begin{equation}
\label{eq:symplecticConnections:momentumMapStrange}
\SectionMapAbb{J}: \ConnSpace_\omega(M) \to \DiffFormSpace^{2n-1}(M) \slash 
\dif \DiffFormSpace^{2n-2}(M),   \qquad
\nabla + A \mapsto \equivClass*{J(\nabla + A) \wedge 
\omega^{n-1}}
\end{equation}
is the unique momentum map for the \( \DiffGroup(M, \omega) \)-action 
on \( \ConnSpace_\omega(M) \) relative to the 
pairing~\eqref{eq:symplecticConnections:pairing} that vanishes 
at \( \nabla \).
\end{prop}
\begin{proof}
According to \cref{prop:affineSymplectic:momentumMap}, the momentum map is 
given by
\begin{equation}
\label{eq:symplecticConnections:momentumMapStrange:fromAffine}
\kappa\bigl(\SectionMapAbb{J}(\nabla + A), \xi\bigr) = \Omega(A, \xi \ldot \nabla) 
+ \frac{1}{2} \Omega(A, \xi \ldot A).
\end{equation}
Let us start by evaluating the first summand on the right-hand side.
Using~\eqref{symplectic_form_connections} 
and~\eqref{eq:symplecticConnections:infAction}, integrating by parts yields
\begin{equation*}\begin{split}
\Omega(A, \xi \ldot \nabla)
&= \int_M A_{ijk} \bigl(\xi \ldot \nabla\bigr)^{ijk} \mu_\omega \\
&= - \int_M A_{ijk} \bigl(\nabla^i \nabla^j \xi^k + 
\tensor{R}{_p^i^j^k} \xi^p\bigr) \mu_\omega   \\
&= - \int_M \bigl( \nabla^j \nabla^i A_{ijp} + 
A_{ijk}\tensor{R}{_p^i^j^k} \bigr) \xi^p \mu_\omega   \\
&= \int_M \bigl( - \nabla_j \nabla_i \tensor{A}{^{ij}_p} + 
\tensor{R}{_p_i_j_k}A^{ijk} \bigr) \xi^p \mu_\omega .
\end{split}\end{equation*}
Using~\eqref{eq:symplecticConnections:infActionLinear} and the symmetry 
of \( A_{ijk} \) we find for the second summand 
in~\eqref{eq:symplecticConnections:momentumMapStrange:fromAffine}:
\begin{equation*}\begin{split}
\Omega(A, \xi \ldot A)
&= \int_M A_{ijk} \bigl(\xi \ldot A\bigr)^{ijk} \mu_{\omega}  \\
&= - \int_M A_{ijk} \bigl(\xi^p \nabla_p \tensor{A}{^{ij}^k} + 
\tensor{A}{_p^j^k} \nabla^i \xi^p + \tensor{A}{^i_p^k} \nabla^j \xi^p - 
\tensor{A}{^i^j^q} \nabla_q \xi^k\bigr) \mu_{\omega}   \\
&= \int_M \bigl(- A_{ijk} \nabla_p \tensor{A}{^{ij}^k} + 
\nabla^i (A_{ijk} \tensor{A}{_p^j^k}) + 
\nabla^j(A_{ijk} \tensor{A}{^i_p^k}) - 
\nabla_q (A_{ijp}\tensor{A}{^i^j^q})\bigr) \xi^p \mu_{\omega}    \\
&= \int_M \bigl(A^{ijk} \nabla_p \tensor{A}{_{ij}_k} - 
3 \nabla_i (A^{ijk} \tensor{A}{_p_j_k})\bigr) \xi^p \mu_{\omega}.
\end{split}\end{equation*}
Thus, comparing 
with~\eqref{eq:symplecticConnections:momentumMapStrange:oneForm}, we get
	\begin{equation*}
		\kappa\bigl(\SectionMapAbb{J}(\nabla + A),\xi\bigr)
			= \Omega(A, \xi \ldot \nabla) + \frac{1}{2} \Omega(A, \xi \ldot A) 
			= \int_M J(\nabla + A)_p \xi^p \mu_\omega .
	\end{equation*}
Finally, for every \( 1 \)-form \( \beta \), we have
\begin{equation*}
\int_M \beta_p \xi^p \mu_\omega 
=  \frac{1}{(n-1)!} \int_M \beta \wedge\omega^{n-1} \wedge (\xi \contr \omega)
= \kappa\bigl(\beta \wedge \omega^{n-1} , \xi\bigr),
\end{equation*}
which yields the expression~\eqref{eq:symplecticConnections:momentumMapStrange} 
for the momentum map \( \SectionMapAbb{J} \).
\end{proof}

Let us rewrite the momentum map \( \SectionMapAbb{J} \) in such a way that 
its geometric meaning becomes apparent. For this purpose, recall that the 
Ricci curvature is defined by \( R_{ij} = \tensor{R}{_{kij}^k} \) and the 
curvature \( 1 \)-form by
\begin{equation}
	\label{eq:symplecticConnections:curvature1Form}
	\rho_i = 2 \nabla^j \tensor{R}{_i_j} = - 2 \nabla_j \tensor{R}{_i^j} \, .
\end{equation}
We sometimes write \( \rho(\nabla) \) to emphasize the dependency on the 
symplectic connection \( \nabla \).
Moreover, let \( \pontryaginClass \equiv \pontryaginClass(\nabla) \) be 
the \( 4 \)-form
\begin{equation}
	\label{eq:symplecticConnections:pontryaginForm}
	\pontryaginClass_{ijkl} = 
		\frac{1}{4 \pi^2}
		\bigl(
			\tensor{R}{_i_j^p^q}\tensor{R}{_k_l_p_q}
			+ \tensor{R}{_i_k^p^q}\tensor{R}{_l_j_p_q}
			+ \tensor{R}{_i_l^p^q}\tensor{R}{_j_k_p_q}
		\bigr)
\end{equation}
representing the first Pontryagin class of \( (M, \nabla) \).
Chern--Weil theory entails that the \( 4 \)-forms \( \pontryaginClass(\nabla 
+ A) \) and \( \pontryaginClass(\nabla) \) associated with the connections 
\( \nabla + A \) and \( \nabla \), respectively, are cohomologous.
Indeed, a straightforward (but lengthy) calculation shows that
\begin{equation}
	\label{eq:symplecticConnections:pontryaginRelationTwoConnections}
	\pontryaginClass(\nabla + A) = \pontryaginClass(\nabla) - 
	\frac{1}{4 \pi^2} \dif \sigma,
\end{equation}
where the \( 3 \)-form \( \sigma \equiv \sigma(\nabla, A) \) is defined by
\begin{equation}\label{eq:symplecticConnections:pontryaginPreForm}
\begin{split}
	\sigma_{ijk} 
		&= \tensor{A}{_i_p^q}\tensor{R}{_j_k_q^p} + 
		\tensor{A}{_j_p^q}\tensor{R}{_k_i_q^p} + 
		\tensor{A}{_k_p^q}\tensor{R}{_i_j_q^p}
			\\
			&\quad + \frac{1}{2} \bigl(
				\tensor{A}{_k_p^q}\nabla_i\tensor{A}{_j_q^p} 
				+ \tensor{A}{_i_p^q}\nabla_j\tensor{A}{_k_q^p}
				+ \tensor{A}{_j_p^q}\nabla_k\tensor{A}{_i_q^p}
				\\
				&\qquad\quad
				- \tensor{A}{_j_p^q}\nabla_i\tensor{A}{_k_q^p}
				- \tensor{A}{_k_p^q}\nabla_j\tensor{A}{_i_q^p}
				- \tensor{A}{_i_p^q}\nabla_k\tensor{A}{_j_q^p}
			\bigr)
			\\
			&\quad - \tensor{A}{_i_a^b} (\tensor{A}{_j_b^c}\tensor{A}{_k_c^a} 
			- \tensor{A}{_k_b^c}\tensor{A}{_j_c^a}).
\end{split}
\end{equation}
Using these notions, the momentum map \( \SectionMapAbb{J} \) has the
following expression.

\begin{thm}
\label{prop:symplecticConnections:momentumMap}
The \( \DiffGroup(M, \omega) \)-action on \( (\ConnSpace_\omega(M), \Omega) \) 
defined in~\eqref{eq:symplecticConnections:action} is symplectic and has a 
momentum map, relative to the pairing~\eqref{eq:symplecticConnections:pairing}, 
given by
\begin{equation}
\label{eq:symplecticConnections:momentumMap}
\SectionMapAbb{J}: \ConnSpace_\omega(M) \to \DiffFormSpace^{2n-1}(M) \slash 
\dif \DiffFormSpace^{2n-2}(M),
		\quad
		\nabla + A \mapsto \equivClass*{J(\nabla + A) \wedge 
		\omega^{n-1}}, 
\end{equation}
where \( J(\nabla + A) \in \DiffFormSpace^1(M) \) is defined by
	\begin{equation}
		J(\nabla + A)_i
			= \frac{1}{2} \bigl( \rho(\nabla + A)_i - \rho(\nabla)_i - 
			\tensor{\sigma(\nabla, A)}{_i_j^j}\bigr) \,.
			\qedhere
	\end{equation}
\end{thm}
The momentum map \( \SectionSpaceAbb{J} \) involves two ingredients 
that have a different flavor. First, it contains the curvature 
\( 1 \)-form \( \rho(\nabla) \) which has a clear geometric meaning. 
Second, the correction term \( \sigma(\nabla, A) \) is closely related
to the Pontryagin class of \( M \) and thus has a more topological origin.
This is another manifestation of the general principle that momentum maps 
for diffeomorphism groups involve geometric as well as topological data 
(see \parencite{DiezRatiuAutomorphisms} for more examples).
\begin{proof}
Recall that the Ricci curvature \( \bar{R}_{ij} \) of the connection 
\( \bar{\nabla} = \nabla + A \) is given by
	\begin{equation}
		\bar{R}_{ij} = R_{ij} + \nabla_k \tensor{A}{_i_j^k} - 
		\tensor{A}{_i_k^l} \tensor{A}{_j_l^k}
	\end{equation}
and, for every tensor \( T^{ij} \), we have
\begin{equation}
\bar{\nabla}_k T^{ij} = \nabla_k T^{ij} + \tensor{A}{_k_l^i} T^{lj} + 
\tensor{A}{_k_l^j} T^{il} \, .
\end{equation}
We hence obtain
	\begin{equation}\begin{split}
		\rho(\bar{\nabla})_i 
			&= - 2 \bar{\nabla}_j \tensor{\bar{R}}{_i^j}
			\\
			&= - 2 (\nabla_j \tensor{\bar{R}}{_i^j} + 
			\tensor{A}{_j_l_i} \tensor{\bar{R}}{^{lj}})
			\\
			&= - 2 (
				\nabla_j \tensor{R}{_i^j} 
				+ \nabla_j \nabla_k \tensor{A}{_i^j^k} 
				- \tensor{A}{^j_l^k} \nabla_j \tensor{A}{_i_k^l} - 
				\tensor{A}{_i_k^l} \nabla_j \tensor{A}{^j_l^k} 
				\\
				&\qquad+ \tensor{A}{_j_l_i} R^{lj}
				+ \tensor{A}{_j_l_i} \nabla_k \tensor{A}{^l^j^k}
				- \tensor{A}{_j_l_i} \tensor{A}{^l_k^p} \tensor{A}{^j_p^k} )
			\\
			&= \rho(\nabla)_i - 2 (\nabla_j \nabla_k \tensor{A}{_i^j^k} + 
			\tensor{A}{_j_k_i} R^{kj} )
			\\
			&\qquad+ 2 \tensor{A}{^j_l^k} \nabla_j \tensor{A}{_i_k^l} - 
			4 \tensor{A}{_j_l_i} \nabla_k \tensor{A}{^l^j^k}
			+ 2 \tensor{A}{_j_l_i} \tensor{A}{^l_k^p} \tensor{A}{^j_p^k} \, .
	\end{split}\end{equation}
On the other hand, we have
\begin{equation}\label{eq:symplecticConnections:pontryaginPreFormContracted}
\begin{split}
\tensor{\sigma}{_{ij}^j} 
	&= 2 \tensor{A}{_i_p^q}\tensor{R}{_q^p} 
		+ 2 \tensor{A}{_j_p^q}\tensor{R}{^j_i_q^p}
	\\
	&\quad 
		+ \tensor{A}{^j_p^q}\nabla_i\tensor{A}{_j_q^p} 
		+ \tensor{A}{_i_p^q}\nabla_j\tensor{A}{^j_q^p}
		- \tensor{A}{^j_p^q}\nabla_j\tensor{A}{_i_q^p}
		\\
		&\quad - 2 \tensor{A}{_i_a^b} \tensor{A}{_j_b^c}\tensor{A}{^j_c^a} \, .
\end{split}
\end{equation}	
Comparing these identities 
with~\eqref{eq:symplecticConnections:momentumMapStrange:oneForm} shows that 
the momentum map \( \SectionMapAbb{J} \) can be indeed written in the 
form~\eqref{eq:symplecticConnections:momentumMap}.
\end{proof}

\begin{remark}
If \( (M, \omega) \) is a two-dimensional symplectic manifold, then the 
\( 3 \)-form \( \sigma \) necessarily vanishes and thus the momentum 
map~\eqref{eq:symplecticConnections:momentumMap} takes the simple form
	\begin{equation}
		\SectionMapAbb{J}: \ConnSpace_\omega(M) \to \DiffFormSpace^{1}(M) 
		\slash \dif \DiffFormSpace^{0}(M),
		\qquad
		\nabla + A \mapsto \frac{1}{2} \equivClass*{\bigl( \rho(\nabla + A) - 
		\rho(\nabla) \bigr)}.
	\end{equation}
Thus, we recover the formula for the momentum map in this setting established 
in \parencite[Theorem~1.2]{Fox2014} (up to some constant).
\end{remark}
Let us derive from the expression~\eqref{eq:symplecticConnections:momentumMap} 
the momentum map for the action of the subgroup \( \HamDiffGroup(M, \omega) 
\subseteq \DiffGroup(M, \omega) \) of Hamiltonian diffeomorphisms. 
As in \cref{sec:kaehler:momentumMap}, we identify the space 
\( \HamVectorFields(M, \omega) \) of Hamiltonian vector fields 
with \( \sFunctionSpace_0(M) \). Under this identification, the 
space \( \dif \DiffFormSpace^{2n-1}(M) \) is dual to  
\( \HamVectorFields(M, \omega) \) and the momentum map for the 
action of \( \HamDiffGroup(M, \omega) \) is given by post-composition 
of the momentum map for the action of \( \DiffGroup(M, \omega) \) 
with the exterior differential.
Accordingly, the momentum map associated with the action of 
\( \HamDiffGroup(M, \omega) \) on \( \ConnSpace_\omega(M) \) is given by
\begin{equation}
\label{eq:symplecticConnections:momentumMapHam}
\SectionMapAbb{J}_{\HamDiffGroup}: \ConnSpace_\omega(M) \to 
\dif \DiffFormSpace^{2n-1}(M),
		\qquad
\nabla + A \mapsto \dif J(\nabla + A) \wedge \frac{\omega^{n-1}}{(n-1)!} \, .
\end{equation}
Let \( \bar{\sigma}_i = \tensor{\sigma}{_i_j^j} \) be the contraction 
of \( \sigma \). A straightforward calculation 
using~\eqref{eq:symplecticConnections:pontryaginRelationTwoConnections} 
yields
	\begin{equation}
		\tensor{(\dif \bar{\sigma})}{_i^i}
			= \frac{1}{2} \tensor{(\dif \sigma)}{_i^i_j^j}
			= - 2 \pi^2 \tensor{\bigl(\pontryaginClass(\nabla + A) + 
			\pontryaginClass(\nabla)\bigr)}{_i^i_j^j} \, .
	\end{equation}
Thus, using~\eqref{eq:symplectic:formWedgeOmegaNMinus1}, we can 
rewrite the momentum map \( \SectionMapAbb{J}_{\HamDiffGroup} \) 
as
\begin{equation}\begin{split}
\SectionMapAbb{J}_{\HamDiffGroup}(\nabla + A)
&= \dif J(\nabla + A) \wedge \frac{\omega^{n-1}}{(n-1)!}
		\\
&= \frac{1}{2} \tensor{\bigl(\dif J(\nabla + A)\bigr)}{_i^i} 
	\mu_\omega
			\\
&= \frac{1}{4} \tensor{\bigl(\dif \rho(\nabla + A) - 
    \dif \rho(\nabla)\bigr)}{_i^i} \mu_\omega
	+ \frac{\pi^2}{2} \tensor{\bigl(\pontryaginClass(\nabla + A) -  
	\pontryaginClass(\nabla)\bigr)}{_i^i_j^j} \mu_\omega
			\\
	&\equiv \bigl(\SectionMapAbb{K}(\nabla + A) - 
	\SectionMapAbb{K}(\nabla)\bigr) \mu_\omega,
	\end{split}\end{equation}
where, in the last line, we introduced the map
	\begin{equation}
		\label{eq:symplecticConnections:cahenGuttMomentumMap}
		\SectionMapAbb{K}: \ConnSpace_\omega(M) \to \sFunctionSpace(M),
		\qquad 
		\nabla \mapsto \frac{1}{2} \left(\nabla_i \rho(\nabla)^i + 
		\pi^2 \tensor{\pontryaginClass(\nabla)}{_i^i_j^j}\right) \, .
	\end{equation}
It was shown in \parencite[Theorem~1.1]{Fox2014} that \( \SectionMapAbb{K} \) 
coincides with the Cahen--Gutt momentum map 
\parencites[Proposition~1.1]{CahenGutt2005}{Gutt2006} for the action of 
the group of Hamiltonian diffeomorphisms on \( \ConnSpace_\omega(M) \).
In other words, \( \SectionMapAbb{J}_{\HamDiffGroup} \) recovers 
the Cahen--Gutt momentum map (as a slight reformulation).
Let us record this observation.
\begin{prop}
\label{prop:symplecticConnections:momentumMapHam}
The action of the subgroup of Hamiltonian diffeomorphisms on 
\( \ConnSpace_\omega(M) \) has a momentum map \( \SectionMapAbb{J}_{\HamDiffGroup}: 
\ConnSpace_\omega(M) \to \dif \DiffFormSpace^{2n-1}(M) \) given by
\begin{equation}
\SectionMapAbb{J}_{\HamDiffGroup}(\nabla + A)
= \dif J(\nabla + A) \wedge \frac{\omega^{n-1}}{(n-1)!}
= \bigl(\SectionMapAbb{K}(\nabla + A) - \SectionMapAbb{K}(\nabla)\bigr) \mu_\omega ,
\end{equation}
where \( \SectionMapAbb{K}: \ConnSpace_\omega(M) \to \sFunctionSpace(M) \) is the 
Cahen--Gutt momentum map defined 
in~\eqref{eq:symplecticConnections:cahenGuttMomentumMap}.
\end{prop}
Note that \( \SectionMapAbb{J}_{\HamDiffGroup} \) is 
\( \HamDiffGroup(M, \omega) \)-equivariant.
Equivariance is, however, no longer the case for the 
momentum map \( \SectionMapAbb{J} \) 
for the full group of symplectomorphisms.

\subsection{Central extension of 
\texorpdfstring{$\DiffGroup(M, \omega)$}{Diff(M, w)}}

As we have seen in~\cref{sec:affineSymplectic}, the momentum map 
for an affine symplectic action is, in general, not equivariant.
For the action of \( \DiffGroup(M, \omega) \) on the space of symplectic 
connections, we obtain the following.
\begin{prop}
	\label{prop:symplecticConnections:nonEquivariance}
The non-equivariance \( 2 \)-cocycle \( \Sigma: \VectorFieldSpace(M, \omega) 
\times \VectorFieldSpace(M, \omega) \to \R \) associated with the momentum 
map \( \SectionMapAbb{J} \) {\rm(}see 
\cref{prop:symplecticConnections:momentumMapStrange} or 
\cref{prop:symplecticConnections:momentumMap}{\rm)} is given by
\begin{equation}
\label{eq:symplecticConnections:nonEquivariance:cocycle}
\Sigma(\xi, \eta) = 
\frac{1}{2} \kappa\bigl(\rho(\nabla), \commutator{\xi}{\eta}\bigr) - 
2 \pi^2 \int_M \tensor{\pontryaginClass}{_k^k_i_j}(\nabla) \xi^i \eta^j 
\mu_\omega \, .
\qedhere
\end{equation}
\end{prop}
The part of \( \Sigma \) not cohomologous to \( 0 \) is determined by the 
contracted Pontryagin class \( \tensor{\pontryaginClass}{_k^k_i_j} \in 
\DiffFormSpace^2(M) \) and thus carries topological information of the 
symplectic manifold \( (M, \omega) \).
\begin{proof}
It is straightforward to verify that the curvature \( 1 \)-form \( \rho \) 
transforms naturally under the action of \( \DiffGroup(M, \omega) \), 
that is, we have
	\begin{equation}
		\rho(\phi \cdot \nabla) = (\phi^{-1})^* \rho(\nabla) 
	\end{equation}
for every \( \phi \in \DiffGroup(M, \omega) \) and \( \nabla \in 
\ConnSpace_\omega(M) \).
Thus, according to~\eqref{eq:affineSymplectic:oneCocycle} 
and~\eqref{eq:symplecticConnections:momentumMap}, 
the non-equivariance \( 1 \)-cocycle \( \lambda: \DiffGroup(M, \omega) \to 
\DiffFormSpace^{2n-1}(M) \slash \dif \DiffFormSpace^{2n-2}(M) \) is given by
	\begin{equation}
		\lambda(\phi)
			= \SectionMapAbb{J}(\phi \cdot \nabla)
			= \frac{1}{2} \equivClass*{\Bigl((\phi^{-1})^* \rho(\nabla) - 
			\rho(\nabla) - \bar{\sigma}(\nabla, \phi \cdot \nabla - \nabla) 
			\Bigr) \wedge \omega^{n-1}},
	\end{equation}
where we recall that \( \bar{\sigma}_i = 
\tensor{\sigma}{_i_j^j} \). Differentiating this relation with respect 
to \( \phi \), we find for the non-equivariance Lie algebra \( 2 \)-cocycle:
	\begin{equation}\begin{split}
		\Sigma(\xi, \eta)
			&=
			\kappa\bigl(\tangent_e \lambda (\xi), \eta \bigr)
			\\
			&= 
			- \frac{1}{2} \dualPair*{\difLie_\xi \rho}{\eta}
			+ \dualPair*{\tensor{(\difLie_\xi \nabla)}{_i_p^q}\tensor{R}{_q^p} 
			+ \tensor{(\difLie_\xi \nabla)}{_j_p^q}\tensor{R}{^j_i_q^p}}{\eta}
			\\
			&= 
			\frac{1}{2} \dualPair*{\rho}{\commutator{\xi}{\eta}}
			\\
			&\qquad+ \dualPair*{\bigl(\nabla_i \nabla_p \xi^q + 
			\tensor{R}{_l_i_p^q} \xi^l\bigr)\tensor{R}{_q^p} + 
			\bigl(\nabla_j \nabla_p \xi^q + 
			\tensor{R}{_l_j_p^q} \xi^l\bigr)\tensor{R}{^j_i_q^p}}{\eta},
	\end{split}\end{equation}
where we used the 
expression~\eqref{eq:symplecticConnections:pontryaginPreFormContracted} 
for \( \bar{\sigma} \) and~\eqref{eq:symplecticConnections:infAction} 
for the Lie derivative of \( \nabla \); the brackets \( \dualPairDot \) 
in the last two lines denote the natural 
pairing of \( 1 \)-forms with vector fields by integration against 
\( \mu_\omega \). The terms on the right-hand side 
involving two partial derivatives cancel, which can be seen using 
integration by parts:
	\begin{equation}\begin{split}
		&\dualPair*{\bigl(\nabla_i \nabla_p \xi^q\bigr)\tensor{R}{_q^p} + 
		\bigl(\nabla_j \nabla_p \xi^q\bigr)\tensor{R}{^j_i_q^p}}{\eta}
			\\
			&\qquad\qquad\qquad
			= \int_M \bigl(\nabla_i \nabla_p \xi^q\bigr)\tensor{R}{_q^p} \eta^i 
			+ \bigl(\nabla_j \nabla_p \xi^q\bigr)\tensor{R}{^j_i_q^p} \eta^i
			\\
			&\qquad\qquad\qquad
			= - \int_M (\nabla_p \xi^q) (\nabla_i \tensor{R}{_q^p}) \eta^i + 
			(\nabla_p \xi^q) (\nabla_j \tensor{R}{^j_i_q^p}) \eta^i
			\\
			&\qquad\qquad\qquad
			= 0,
	\end{split}\end{equation}
where we used the fact that \( \nabla_i \eta_j \) is symmetric in \( (i, j) \) 
and thus the terms involving the covariant derivatives of \( \eta^i \) vanish.
Indeed, \(0 = (\difLie_\eta \omega)_{ij} = 
(\nabla_\eta \omega)_{ij} + \nabla_i \eta_j - \nabla_j \eta_i\) and
\(\nabla_\eta \omega =0\).
Thus, we get
\begin{equation}\label{eq:symplecticConnections:nonEquivariance:cocyclePrelim}
\begin{split}
		\Sigma(\xi, \eta) 
			= 
			\frac{1}{2} \dualPair*{\rho}{\commutator{\xi}{\eta}}
			+ \dualPair*{\bigl(\tensor{R}{_i_j_p^q} \tensor{R}{_q^p} + 
			\tensor{R}{_i_r_p^q} \tensor{R}{^r_j_q^p}\bigr) \xi^i}{\eta}.
	\end{split}\end{equation}
Finally, we find for the contraction of the Pontryagin form 
(see~\eqref{eq:symplecticConnections:pontryaginForm}):
	\begin{equation}
		4 \pi^2 \tensor{\pontryaginClass}{_k^k_i_j}
			=
			\tensor{R}{_k^k^p^q}\tensor{R}{_i_j_p_q}
			+ \tensor{R}{_k_i^p^q}\tensor{R}{_j^k_p_q}
			+ \tensor{R}{_k_j^p^q}\tensor{R}{^k_i_p_q}
			= 2 \tensor{R}{^p^q}\tensor{R}{_i_j_p_q} 
			+ 2 \tensor{R}{_k_i^p^q}\tensor{R}{_j^k_p_q} \, .
	\end{equation}
Inserting this relation into the 
expression~\eqref{eq:symplecticConnections:nonEquivariance:cocyclePrelim} 
for the non-equivariance cocycle \( \Sigma \) 
yields~\eqref{eq:symplecticConnections:nonEquivariance:cocycle}.
\end{proof} 

Recall from the discussion in \cref{sec:kaehler:extension}, that 
\( 2 \)-cocycles on \( \VectorFieldSpace(M, \omega) \) are sums of 
extensions of certain \( 2 \)-cocycles on \( \HamVectorFields(M, \omega) \) 
and pull-backs of elements of \( \ExtBundle^2 {\deRCohomology^1(M)}^* \).
Applied to the non-equivariance cocycle \( \Sigma \), we obtain the following.
\begin{prop}
	\label{prop:symplecticConnections:nonEquivariance:cocycleClass}
The class of the non-equivariance cocycle \( \Sigma \) in the second 
continuous Lie algebra cohomology of \( \VectorFieldSpace(M, \omega) \) 
coincides with the pull-back along the natural map 
\( \VectorFieldSpace(M, \omega) \to \deRCohomology^1(M) \) of the 
antisymmetric bilinear form
\begin{equation}
\bigl(\equivClass{\alpha}, \equivClass{\beta}\bigr) \mapsto \pi^2 \int_M 
\left(\SectionMapAbb{p} (\nabla) \, \varpi^{ij} - 
2 \tensor{\pontryaginClass}{_k^k^i^j} (\nabla)\right) \alpha_i \beta_j \, 
\mu_\omega
\end{equation}
on \( \deRCohomology^1(M) \), where
\begin{equation}
\SectionMapAbb{p}(\nabla) \defeq \tensor{\pontryaginClass}{_i^i_j^j}(\nabla) 
- \frac{1}{\vol_{\mu_\omega}(M)} \int_M \tensor{\pontryaginClass}{_i^i_j^j}
(\nabla) \, \mu_\omega \, .
\qedhere
\end{equation}
\end{prop}
It should not come as a big surprise that there is no contribution from the 
second Lie algebra cohomology of \( \HamVectorFields(M, \omega) \), because 
the momentum map \( \SectionMapAbb{K} \) for the action of the group of 
Hamiltonian diffeomorphisms is equivariant.
We conclude that the momentum map for the action of the full group of 
symplectomorphisms contains topological information of \( (M, \omega) \) 
in terms of the Pontryagin form
while \( \HamVectorFields(M, \omega) \) is not sensitive to these 
topological properties. A similar dichotomy has also been observed in 
\parencite{DiezRatiuAutomorphisms} for different actions of 
symplectomorphism groups. 

The prequantum bundle construction in  
\cref{prop:affineSymplectic:groupExt} shows that the \( 2 \)-cocycle 
\( \Sigma \) integrates to a central Lie group extension of 
\( \DiffGroup(M, \omega) \).
\begin{thm}
	\label{prop:symplecticConnections:nonEquivariance:groupExtension}
There exists a central Lie group \( \UGroup(1) \)-extension of the group 
\( \DiffGroup(M, \omega) \) of symplectomorphisms whose corresponding 
Lie algebra \( 2 \)-cocycle is cohomologous to the non-equivariance 
\( 2 \)-cocycle \( \Sigma \).
\end{thm}

Note that the central group extension in 
\cref{prop:symplecticConnections:nonEquivariance:groupExtension} has 
been obtained by means of the action of \( \DiffGroup(M, \omega) \) 
on the infinite-dimensional space of symplectic connections.
On the other hand, we have seen in 
\cref{prop:symplecticConnections:nonEquivariance:cocycleClass} that 
the non-equivariance cocycle \( \Sigma \) is essentially the pull-back 
of a cocycle on the finite-dimensional space \( \deRCohomology^1(M) \).
One may thus hope for a finite-dimensional construction of the central 
extension of \( \DiffGroup(M, \omega) \).
This is an issue for future research to explore.

\subsection{Norm-squared momentum map}
In this section, we apply our general results concerning the norm-squared 
of the momentum map to the action of the symplectomorphism 
group on the space  
of symplectic connections.
As discussed in \cref{sec:symplecticConnections:momentumMap}, the action of the 
group \( \DiffGroup(M, \omega) \) of symplectomorphisms leaves 
\( \Omega \) invariant and has a momentum map \( \SectionMapAbb{J}: 
\ConnSpace_\omega(M) \to \Omega^{2n - 1}(M) \slash \dif \Omega^{2n-2}(M) \)
as calculated in \cref{prop:symplecticConnections:momentumMap}.
Here, the target of \( \SectionMapAbb{J} \) is identified with the dual 
of \( {\VectorFieldSpace(M, \omega)}^* \) by the 
pairing~\eqref{eq:symplecticConnections:pairing}.
In order to fit this setting into the general framework of 
\cref{sec:momentumMapSquared}, we need to realize \( \SectionMapAbb{J} \) 
as a map into the space of symplectic vector fields.
For this purpose, let \( j \) be a complex\footnotemark{} structure
on \( M \) compatible with \( \omega \), 
\ie, \( \omega(j \,\cdot, j \,\cdot) = 
\omega(\cdot, \cdot)  \) and \( \omega(X, j X) > 0 \) for all non-zero 
\( X \in \TBundle M \).
\footnotetext{Most results of this section hold with minor modification 
also when \( j \) is not integrable.}
Denote the associated Riemannian metric by \( g(\cdot, \cdot )  
= \omega (\cdot, j \cdot ) \).
Using this data, consider the following non-degenerate 
pairing on \( \VectorFieldSpace(M, \omega) \):
\begin{equation}
	\label{eq:symplecticConnections:pairingOnLieAlgebra}
	\kappa(\xi, \eta) = \int_M g(\xi, \eta) \, \mu_\omega \, .
\end{equation}
Relative to this pairing, the momentum 
map~\eqref{eq:symplecticConnections:momentumMap} takes the form
\begin{equation}
\SectionMapAbb{J}: \ConnSpace_\omega(M) \to \VectorFieldSpace(M, \omega), 
\quad \nabla + A \mapsto - 
\frac{1}{2} \tensor{j}{^i_k} \bigl( \rho(\nabla + A)^k - 
\rho(\nabla)^k - \tensor{\sigma(\nabla, A)}{^k_j^j}\bigr) \,,
\end{equation}
where \( \rho_i \) and \( \sigma_{ijk} \) have been defined 
in~\eqref{eq:symplecticConnections:curvature1Form} 
and~\eqref{eq:symplecticConnections:pontryaginPreForm}, respectively.

In the following, it is often convenient to work on the complexified 
tangent bundle and use an abstract index notation that is adapted to 
the decomposition of \( \TBundle M \tensorProd \C = 
\TBundle^{(1,0)} M \oplus \TBundle^{(0,1)} M \) into 
\( \pm i \)-eigenspaces of \( j \).
For this purpose, we use capital Latin letters 
\( \mathsf{A, B}, \ldots \) 
to denote elements of \( \TBundle M \tensorProd \C \), Greek 
letters \( \alpha, \beta, \ldots \) for elements of 
\( \TBundle^{(1,0)} M \) and overlined Greek letters 
\( \bar{\alpha}, \bar{\beta}, \ldots \) for elements of  
\( \TBundle^{(0,1)} M \). For example, \( X^\mathsf{A}\) is a complex 
vector field and \( Y^{\bar{\alpha}} \) is a \( (0,1) \)-vector field.
Moreover, we use only the symplectic form and not the metric to lower and raise 
indices.

Using these conventions, the complex structure \( j \) on \( M \) 
defines a constant almost complex structure \( \SectionMapAbb{j} \) 
on \( \ConnSpace_\omega(M) \) by
\begin{center}
\begin{tabular}{ l c }
	$\mathsf{ABC}$ & $(\SectionMapAbb{j} A)_\mathsf{ABC}$ \\ \midrule
	$\alpha\beta\gamma$ & $- \I A_{\alpha\beta\gamma}$ \\
	$\bar{\alpha}\beta\gamma$ & $- \I A_{\bar{\alpha}\beta\gamma}$ \\  
	$\alpha\bar{\beta}\bar{\gamma}$ & $+ \I A_{\alpha\bar{\beta}\bar{\gamma}}$ \\  
	$\bar{\alpha}\bar{\beta}\bar{\gamma}$ & 
	$+ \I A_{\bar{\alpha}\bar{\beta}\bar{\gamma}}$     
	\end{tabular}
\end{center}
and symmetric extension, where $A$ is the symmetric covariant 3-tensor 
defined in the text following~\eqref{two_nablas}. Here, the 
possible components are listed in the first 
column, and the corresponding value of \( \SectionMapAbb{j} A \) is the entry 
in the second column, \eg the first row is equivalent to 
\( (\SectionMapAbb{j} A)_{\alpha\beta\gamma} = - \I A_{\alpha\beta\gamma} \).
Note that this complex structure is not just precomposition of \( A \) 
with \( j \), which would have a different sign in the second and third row.

A direct calculation yields \( \Omega(\SectionMapAbb{j} \cdot, 
\SectionMapAbb{j} \cdot) = \Omega(\cdot, \cdot) \).
Moreover, in \parencite[Proposition~17 and Remark~20]{Fuente-Gravy2015} 
(see also \parencite[Lemma~4.9]{FutakiOno2018}) it has been shown 
that \( \Omega(\cdot, \SectionMapAbb{j} \cdot) \) is positive definite 
on the complexified \( \DiffGroup(M, \omega) \)-orbit 
if the Ricci curvature is non-negative.
In the general setting above, we only used the non-degeneracy of 
\( \Omega(\cdot, \SectionMapAbb{j} \cdot) \) in 
\cref{i:normedsquared:kernelsIdentified} to determine
the kernel of the Calabi operators and, for this computation, 
non-degeneracy along the 
\( \VectorFieldSpace(M,\omega)_\C \)-orbit suffices.

The Levi-Civita connection \( \nabla^j \) associated with 
the Riemannian metric defined by \( j \) and \( \omega \) is 
a symplectic connection. We say that a compatible complex structure 
\( j \) is a \emphDef{Cahen--Gutt critical} if its Levi-Civita 
connection \( \nabla^j \) is a critical point of the norm-squared of the momentum 
map \( \norm{\SectionMapAbb{J}}_\kappa^2: 
\ConnSpace_\omega(M) \to \R \).

We need the following generalization of 
\parencite[Lemma~2.2 and 4.9]{FutakiOno2018} from 
Hamiltonian vector fields to symplectic vector fields.
\begin{lemma}
	\label{prop:symplecticConnections:stabilizersIdentified}
Let \( (M, \omega, j) \) be a compact K\"ahler manifold with 
Levi-Civita connection \( \nabla \). The following holds:
\begin{thmenumerate}
\item
For every \( X \in \VectorFieldSpace(M, \omega) \), \( \difLie_X \nabla = 0 \) 
if and only if \( X \) is real holomorphic.
\item
For every \( X + \I Y \in \VectorFieldSpace(M, \omega)_\C \), 
\( \difLie_X \nabla + \SectionMapAbb{j} \difLie_Y \nabla = 0 \) 
if and only if \( (X + \I Y)^{(1,0)} \) is holomorphic.
Moreover, the map \( X + \I Y \mapsto X - j Y \) yields a surjection 
from \( \VectorFieldSpace(M, \omega)_\C \) onto the space of 
holomorphic vector fields.
			\qedhere
\end{thmenumerate}
\end{lemma}
\begin{proof}
Let \( Z^\mathsf{A} \in \VectorFieldSpace(M, \omega)_\C \). Since 
\( \DiffGroup(M, \omega) \) acts on the space of 
symplectic connections, we know that 
\( (\difLie_Z \nabla)_\mathsf{ABC} \) is a symmetric tensor.
The only independent components are given as follows:
\begin{equation}
\label{eq:symplecticConnections:difLieNabla}
\begin{array}{ l c }
\mathsf{ABC} & (\difLie_Z \nabla)_\mathsf{ABC} \\ \midrule
\alpha\beta\gamma & \nabla_\alpha \nabla_\beta Z_\gamma \\
\bar{\alpha}\beta\gamma & \nabla_{\bar{\alpha}} \nabla_\beta Z_\gamma \\  
\alpha\bar{\beta}\bar{\gamma} & 
\nabla_\alpha \nabla_{\bar{\beta}} Z_{\bar{\gamma}} \\  
\bar{\alpha}\bar{\beta}\bar{\gamma} & 
\nabla_{\bar{\alpha}} \nabla_{\bar{\beta}} Z_{\bar{\gamma}} \\
\end{array}
\end{equation}
Here, we used~\eqref{eq:symplecticConnections:infAction} and 
the fact that the Riemann curvature of a K\"ahler metric 
has additional symmetry properties, so that, for example, 
\( R_{\mathsf{D}\alpha\beta\gamma} \) vanishes.	Thus, 
\( \difLie_Z \nabla = 0 \) implies, using integration 
by parts, that
\begin{equation}
\int_M g^{\alpha \bar{\gamma}} g^{\beta \bar{\delta}} \,
\bigl(\nabla_\alpha Z_\beta\bigr) \bigl(\nabla_{\bar{\gamma}} 
\bar{Z}_{\bar{\delta}}\bigr) \, \mu_\omega = 0 \, ,
\end{equation}
hence \( \nabla_\alpha Z_\beta = 0 \). Similarly, we conclude 
that \( \nabla_{\bar{\alpha}} Z_{\bar{\beta}} = 0 \).
Summarizing, \( \difLie_Z \nabla = 0 \) is equivalent to 
\( \nabla_\alpha Z_\beta = 0 = \nabla_{\bar{\alpha}} Z_{\bar{\beta}} \). 

On the other hand, by \parencite[Lemma~2.3]{Futaki2006}, we have
\begin{equation}
\tensor{(\difLie_Z j)}{_{\mathsf{A}}^{\mathsf{B}}} = 
- 2 \I \tensor{\delta}{_{\bar{\beta}}^{\mathsf{B}}} 
\tensor{\delta}{_{\mathsf{A}}^{\alpha}} \, 
\nabla_\alpha Z^{\bar{\beta}}
+ 2 \I \tensor{\delta}{_\beta^{\mathsf{B}}} 
\tensor{\delta}{_{\mathsf{A}}^{\bar{\alpha}}} \, 
\nabla_{\bar{\alpha}} Z^{\beta} \, .
\end{equation}
Thus, upon lowering the last index, the only non-zero components 
of \( \tensor{(\difLie_Z j)}{_{\mathsf{AB}}} \) are:
	\begin{center}
	\begin{tabular}{ l c }
	$\mathsf{AB}$ & $(\difLie_Z j)_\mathsf{AB}$ \\ \midrule
	$\alpha\beta$ & $-2 \I \, \nabla_\alpha Z_\beta$ \\
	$\bar{\alpha}\bar{\beta}$ & $2 \I \, \nabla_{\bar{\alpha}} Z_{\bar{\beta}}$     
	\end{tabular}
	\end{center}
Thus, we see that \( \difLie_Z \nabla = 0 \) if and only if 
\( \difLie_Z j = 0 \). In particular, this holds for \( Z = X \) 
being a real vector field. This proves (i).

Finally, let \( X, Y \in \VectorFieldSpace(M, \omega) \) be 
such that \( \difLie_X \nabla + \SectionMapAbb{j} \difLie_Y \nabla = 0 \).
The definition of \( \SectionMapAbb{j} \) 
and~\eqref{eq:symplecticConnections:difLieNabla} imply that 
this is equivalent to
\begin{equation}\begin{split}
\nabla_\alpha \nabla_\beta (X_\gamma - \I Y_\gamma) 
&= 0 = \nabla_{\bar{\alpha}} \nabla_\beta (X_\gamma - \I Y_\gamma), \\
\nabla_\alpha \nabla_{\bar{\beta}} (X_{\bar{\gamma}} + \I Y_{\bar{\gamma}}) 
&= 0 = \nabla_{\bar{\alpha}} \nabla_{\bar{\beta}} (X_{\bar{\gamma}} 
+ \I Y_{\bar{\gamma}}).
\end{split}\end{equation}
Using integration by parts as above, we see that these equations 
themselves are equivalent to \( \nabla_{\bar{\beta}} (X_{\bar{\gamma}} 
+ \I Y_{\bar{\gamma}}) = 0 \), that is \( \nabla_{\bar{\beta}} (X^{\gamma} 
+ \I Y^{\gamma}) = 0 \) upon lifting the second index, which
proves the first part of (ii).
The second part follows as in \cref{prop:kaehler:complexStabilizerToHolomorphic}:
\begin{equation}
	(\difLie_X j - \difLie_{jY} j)_{{\bar{\alpha}} {\bar{\beta}}} 
	= (\difLie_X j - j \difLie_Y j)_{{\bar{\alpha}} {\bar{\beta}}} 
	= 2 \I \nabla_{\bar{\alpha}} X_{\bar{\beta}} - 
	2 \nabla_{\bar{\alpha}} Y_{\bar{\beta}} 
	= 2 \I \nabla_{\bar{\alpha}}(X_{\bar{\beta}} + 
	\I Y_{\bar{\beta}}) = 0.
	\qedhere
\end{equation}
\end{proof}

This shows that every vector field \( X \) in the stabilizer 
\( \VectorFieldSpace(M, \omega)_\nabla \) of the Levi-Civita 
connection \( \nabla \) is real holomorphic, and thus Killing.
Hence, \( \kappa \) and \( \SectionMapAbb{j} \) are invariant 
under \( \VectorFieldSpace(M, \omega)_\nabla \).
Moreover, the stabilizer of \( \nabla \) under the 
\( \VectorFieldSpace(M,\omega)_\C \)-action
projects onto the Lie algebra of 
holomorphic vector fields; in particular, 
it is finite dimensional, too.
\Cref{prop:normedsquared:criticalPoints} implies that a compatible complex 
structure \( j \) is Cahen--Gutt critical if and only if 
\( \SectionMapAbb{J}(\nabla^j) \in \VectorFieldSpace(M, \omega) \) is 
real holomorphic.
Note that our notion of extremality is hence slightly different 
from \parencite{FutakiOno2018,Fox2014}.

As a consequence of \cref{prop:normedsquared:decompositionComplexStab} 
we obtain the following.
\begin{thm}
\label{prop:symplecticConnections:decompositionAut}
Let \( (M, \omega) \) be a compact symplectic manifold and let \( j \) 
be a compatible Cahen--Gutt critical 
complex structure on \( M \). Assume that the Ricci curvature of 
the Levi-Civita connection \( \nabla \) associated with \( g_j \) is non-negative.
Then the Lie algebra of real holomorphic vector fields admits the 
following decomposition:
\begin{equation}
\label{eq:symplecticConnections:decompositionAut}
\LieA{h}(M, j) = \LieA{c} \oplus 
\bigoplus_{\lambda \neq 0} \LieA{k}_\lambda,
\end{equation}
where:
\begin{thmenumerate}
\item
\( \LieA{c} \) is the Lie subalgebra of 
\( \LieA{h}(M, j) \) consisting of 
all elements that commute with \( \SectionMapAbb{J}(\nabla) \);
\item \( \C \SectionMapAbb{J}(\nabla) \subseteq \LieA{c} \);
\item
\( \LieA{k}_\lambda \) are eigenspaces of 
\( -2 j \difLie_{\SectionMapAbb{J}(\nabla)} \) 
with eigenvalue \( \lambda \in \R \) {\rm(}with the convention that 
\( \LieA{k}_\lambda = \set{0} \) if \( \lambda \) is not an 
eigenvalue{\rm)}; in particular, $\mathfrak{c} = \mathfrak{k}_0$;
\item
\( \commutator{\LieA{k}_\lambda}{\LieA{k}_\mu} \subseteq 
\LieA{k}_{\lambda + \mu} \) if 
\( \lambda + \mu \) is an eigenvalue 
of \( -2 j \difLie_{\SectionMapAbb{J}(\nabla)} \); otherwise 
\( \commutator{\LieA{k}_\lambda}{\LieA{k}_\mu} = 0 \).
\end{thmenumerate}
Moreover, the Hessian of \( \norm{\SectionMapAbb{J}}^2 \) at 
\( \nabla \) is given by
	\begin{equation}
		\frac{1}{2} \Hessian_\nabla \norm{\SectionMapAbb{J}}^2 
		(\zeta \ldot \nabla, 
		\gamma \ldot \nabla) = 
		\Re  \, \dualPair{\zeta}{C^+_\nabla R_\nabla \gamma}_{\C},
	\end{equation}
	for \( \zeta, \gamma \in \VectorFieldSpace(M, \omega)_\C \) and
	\begin{equation}
		C_\nabla^\pm = L_\nabla \pm \I Z_\nabla, \qquad 
		R_\nabla = C^-_\nabla \Matrix{0 & 0 \\ 0 & 1} + \I \, 
		\bigl(\difLie_{\SectionMapAbb{J}(\nabla)} + Z_\nabla\bigr)
	\end{equation}
where \( L_\nabla = - \tangent_\nabla \SectionMapAbb{J} (\SectionMapAbb{j} \, 
\difLie_\xi \nabla) \) and \( Z_\nabla = 
- \tangent_\nabla \SectionMapAbb{J} (\difLie_\xi \nabla) \).
\end{thm}
\begin{proof}
We can apply \cref{prop:normedsquared:decompositionComplexStab} 
to obtain a decomposition of the stabilizer 
\( \VectorFieldSpace(M, \omega) \ldot \nabla \), where \( \nabla \) 
is the Levi-Civita connection associated with \( g_j \).
The claims concerning the 
decomposition~\eqref{eq:symplecticConnections:decompositionAut} 
follow then directly under the map \( \VectorFieldSpace(M, \omega) \ldot 
\nabla \ni X + \I Y \mapsto X - j Y \in \LieA{h}(M, j) \), 
\cf \cref{prop:symplecticConnections:stabilizersIdentified} and the 
proof of \cref{prop:kaehler:decompositionComplexStabAndHol}.
The expression for the Hessian follows directly from 
\cref{prop:normedsquared:hessianSummary}.
\end{proof}
\begin{remarks}
\item If one proceeds in an analogous way for the action of the 
subgroup of Hamiltonian diffeomorphism, then one recovers 
\parencite[Theorem~4.7 and~4.11]{FutakiOno2018} as a direct application of 
\cref{prop:normedsquared:decompositionComplexStab,prop:normedsquared:hessianSummary}.
\item A similar theorem holds for connections that are critical points 
of the momentum map squared (seen as a functional on 
\( \ConnSpace_\omega(M) \)) 
without being necessarily the Levi-Civita connection of some compatible 
Riemannian metric. However, in this case, we do not know of a result 
similar to \cref{prop:symplecticConnections:stabilizersIdentified} that 
allows us to identify the stabilizers.
\item Instead of using the almost complex structure \( \SectionMapAbb{j} \) 
on \( \ConnSpace_\omega(M) \) defined above, one could also work with 
the almost complex structure that sends 
\( A \in \SymTensorFieldSpace_3(M) \) 
to \( A(j \cdot, j \cdot, j \cdot) \).
In this case, the stabilizer of the Levi-Civita connection under the 
complexified action is a proper subalgebra of \( \LieA{h}(M, j) \).
\item In \parencite{FutakiOno2018,Fuente-Gravy2016} a slightly different 
viewpoint is used: instead of working on the symplectic manifold 
\( \ConnSpace_\omega(M) \) of symplectic connections as we do above, 
the pull-back of the symplectic form \( \Omega \) along the Levi-Civita 
map \(\SectionMapAbb{lc}: j \mapsto \nabla^j\) to the space 
\( \SectionSpaceAbb{I}(M, \omega) \) of integrable complex structures 
on \( M \) compatible with \( \omega \) is used.
In this setting, the condition of non-negative Ricci curvature is necessary 
to guarantee the non-degeneracy of \( \SectionMapAbb{lc}^* \Omega \); see 
\parencite[Proposition~17]{Fuente-Gravy2016}.
\end{remarks}

\section{Application: Yang--Mills}
\label{sec_yang_mills}

Let \( G \) be a compact connected Lie group and let \( P \to M \) be a 
principal \( G \)-bundle over a closed connected Riemann surface \( M \).
Fix an \( \AdAction \)-invariant pairing on the Lie algebra \( \LieA{g} \) 
of \( G \). The space \( \ConnSpace(P) \) of connections on \( P \) is an 
affine space modeled on the tame Fr\'echet space 
\( \DiffFormSpace^1(M, \AdBundle P) \) 
of \( 1 \)-forms on \( M \) with values in the adjoint bundle 
\( \AdBundle P \).
The \( 2 \)-form \( \omega \) on \( \ConnSpace(P) \) defined by the 
integration pairing
\begin{equation}
	\label{eq:yangMillsSurface:symplecticForm}
	\omega_A (\alpha, \beta) = \int_M \wedgeDual{\alpha}{\beta}
\end{equation}
for \( \alpha, \beta \in \DiffFormSpace^1(M, \AdBundle P) \) is a 
symplectic form, where \( \wedgeDualDot \) denotes the wedge 
product\footnotemark{} relative to the 
\( \AdAction \)-invariant pairing on \( \LieA{g} \).
\footnotetext{For \( \alpha, \beta \in \DiffFormSpace^1(M, \AdBundle P) \) 
and \( X, Y \in \VectorFieldSpace(M) \), we have 
\( \wedgeDual{\alpha}{\beta}(X,Y) = \dualPair{\alpha(X)}{\beta(Y)} - 
\dualPair{\alpha(Y)}{\beta(X)} \).}
The natural action on \( \ConnSpace(P) \) of the group \( \GauGroup(P) \) 
of gauge transformations of \( P \) is smooth and preserves the symplectic 
structure \( \omega \). The \( \AdAction \)-invariant pairing on 
\( \LieA{g} \) induces a natural pairing
\begin{equation}
	\kappa: \sSectionSpace(\AdBundle P) \times 
	\sSectionSpace(\AdBundle P) \to \R, \qquad (\phi, \varrho) \mapsto 
	\int_M \dualPair{\phi}{\varrho} \, \vol_g \, .
\end{equation}
A straightforward calculation verifies that the map
\begin{equation}
	\SectionMapAbb{J}: \ConnSpace(P) \to \sSectionSpace(\AdBundle P), 
	\qquad A \mapsto - \hodgeStar F_A
\end{equation}
is an equivariant momentum map for the \( \GauGroup(P) \)-action on 
\( \ConnSpace(P) \), see \parencite{AtiyahBott1983}.
The norm-squared of the momentum map 
\( \norm{\SectionMapAbb{J}}_\kappa^2(A) 
= \int_{M} F_A \wedge \hodgeStar F_A \) is the Yang--Mills action, whose 
critical points are, according to \cref{prop:normedsquared:criticalPoints}, 
precisely the Yang--Mills connections, \ie, connections \( A \) 
satisfying \( \dif \hodgeStar F_A = 0 \). This observation goes back 
at least to \parencite[Proposition~4.6]{AtiyahBott1983}. Moreover, 
the Hodge star operator squares to minus the identity on \( 1 \)-forms 
and so yields an almost complex structure 
\( \hodgeStar: \DiffFormSpace^1(M, \AdBundle P ) \to 
\DiffFormSpace^1(M, \AdBundle P ) \) that is compatible 
with \( \omega \). Upon complexification, we obtain a decomposition
\begin{equation}
	\DiffFormSpace^1(M, \AdBundle P \tensorProd \C) 
	= \DiffFormSpace^{1,0}(M, 
	\AdBundle P) \oplus \DiffFormSpace^{0,1}(M, \AdBundle P)
\end{equation}
in eigenspaces of \( \hodgeStar \) with eigenvalues \( -i \) and \( i \), 
respectively.
For a connection \( A \), the associated
exterior derivative \( \dif_A: 
\DiffFormSpace^0(M, \AdBundle P \tensorProd \C) \to \DiffFormSpace^1(M, 
\AdBundle P \tensorProd \C) \) decomposes accordingly into 
\( \dif_A = \difp_A + \difpBar_A \). Under the 
\( \hodgeStar \)-\( i \)-complex 
linear identification \( \DiffFormSpace^1(M, \AdBundle P) \ni 
\alpha \mapsto \I \alpha + \hodgeStar \alpha \in 
\DiffFormSpace^{0,1}(M, \AdBundle P) \), the 
complexified action on the Lie algebra level is the operator 
\( - 2 \I \difpBar_A: \GauAlgebra(P)_\C \to 
\DiffFormSpace^{0,1}(M, \AdBundle P) \).
Thus, the stabilizer \( \bigl(\GauAlgebra(P)_\C\bigr)_A \) 
is identified with 
the space \( \holSectionSpace_A(\AdBundle P \tensorProd \C) \) of 
holomorphic sections of \( \AdBundle P \tensorProd \C \).
Moreover, in \parencite[p.~556]{AtiyahBott1983}, it was shown that 
the eigenvalues of the endomorphism 
\( - 2 \I \, \commutator{\hodgeStar F_A}{\cdot} \) on 
\( \AdBundle P \) are locally constant, and that one thus 
obtains an eigenspace decomposition
\begin{equation}
\label{eq:yangMillsSurface:decompositionAdBundle}
\AdBundle P \tensorProd \C = \bigoplus_{\lambda} \AdBundle_\lambda P \, ,
\end{equation}
where \( \AdBundle_\lambda P \) is the eigenspace of \( - 2 \I \, 
\commutator{\hodgeStar F_A}{\cdot} \) corresponding to the eigenvalue 
\( \lambda \). Since \( \bigl(\GauAlgebra(P)_\C\bigr)_A \) is 
finite-dimensional, we can apply 
\cref{prop:normedsquared:decompositionComplexStab} to obtain the 
following.
\begin{thm}
Let \( G \) be a compact connected Lie group and let \( P \to M \) 
be a principal \( G \)-bundle over a closed connected Riemann 
surface \( M \). For every Yang--Mills connection \( A \) on \( P \), 
the following decomposition of 
the complex Lie algebra of holomorphic sections of 
\( \AdBundle P \tensorProd \C \) 
holds:
	\begin{equation}
		\holSectionSpace_A\bigl(\AdBundle P \tensorProd \C\bigr) = 
		\bigl(\GauAlgebra(P)_A\bigr)_\C \oplus \bigoplus_{\lambda < 0} 
		\holSectionSpace_A\bigl(\AdBundle_\lambda P\bigr)
	\end{equation}
such that \( \hodgeStar F_A \) lies in the center of 
\( \GauAlgebra(P)_A \) and 
	\begin{equation}
		\commutator*{\holSectionSpace_A\bigl(\AdBundle_\lambda P\bigr)}
		{\holSectionSpace_A\bigl(\AdBundle_\mu P\bigr)} \subseteq 
		\holSectionSpace_A\bigl(\AdBundle_{\lambda + \mu} P\bigr),
	\end{equation}
with the convention that \(\holSectionSpace_A\bigl(
\AdBundle_{\lambda+\mu} P\bigr)\) is trivial if \( \lambda + \mu \) 
is not an eigenvalue of \( - 2 \I \, \commutator{\hodgeStar F_A}{\cdot} \).
\end{thm}

\begin{proof}
This follows from \cref{prop:normedsquared:decompositionComplexStab} but for 
completeness we give a sketch of a direct proof.
	
The decomposition~\eqref{eq:yangMillsSurface:decompositionAdBundle} of 
\( \AdBundle P \tensorProd \C \) induces decompositions on the level of 
differential forms:
	\begin{equation}
		\DiffFormSpace^k(M, \AdBundle P \tensorProd \C) = 
		\bigoplus_{\lambda} \DiffFormSpace^k(M, \AdBundle_\lambda P).
	\end{equation}
As a consequence of the Yang--Mills equation, the operators \( \difpBar_A \) 
and \( \commutator{\hodgeStar F_A}{\cdot} \) commute.
Hence, \( \difpBar_A \) decomposes into the sum of operators 
\( \difpBar_{A, \lambda}: \DiffFormSpace^0(M, \AdBundle_\lambda P) \to 
\DiffFormSpace^1(M, \AdBundle_\lambda P) \) and so
\begin{equation}
\holSectionSpace_A\bigl(\AdBundle P \tensorProd \C\bigr) = 
\bigoplus_{\lambda} \holSectionSpace_A\bigl(\AdBundle_\lambda P\bigr) \, .
\end{equation}
By considering appropriate Laplacian operators, one can show that 
\( \holSectionSpace_A\bigl(\AdBundle_\lambda P\bigr) \) is isomorphic 
to \( \bigl(\GauAlgebra(P)_A\bigr)_\C \) for \( \lambda = 0 \) and 
is trivial for \( \lambda > 0 \); see 
\parencite[Lemma~5.9~(iii) and p.~559]{AtiyahBott1983}. 
\end{proof}
If \( P \) is a reduction of a \( G^\C \)-principal bundle 
\( P^\C \) to \( G \subseteq G^\C \), the space 
\( \holSectionSpace_A\bigl(\AdBundle P \tensorProd \C\bigr) \) is 
naturally identified with the space of sections of \( \AdBundle P^\C \) 
that are holomorphic with respect to the holomorphic structure 
\( \difpBar_A \) on \( P^\C \) induced by the connection \( A \).
Hence, \( \holSectionSpace_A\bigl(\AdBundle P \tensorProd \C\bigr) \) 
can be viewed as the stabilizer algebra of \( \difpBar_A \) under 
the action of \( \GauGroup(P^\C) \) on the space of holomorphic 
structures on \( P^\C \).

\begin{remark}
It is possible to extend the above results to the case when the 
base \( M \) is a compact symplectic manifold of arbitrary dimension; 
see \parencite[Section~4]{Donaldson1985} for the setup of the 
infinite-dimensional symplectic framework.
Then global minima and critical points of the norm-squared of the 
momentum map correspond to K\"ahler--Einstein connections and 
Hermitian Yang--Mills connections, respectively.
This extension is especially fruitful when coupled to other geometric 
structures on the base, such as one of the special K\"ahler metrics 
discussed in \cref{sec:kaehler}.
For example, we expect that our general results directly yield the 
reductiveness obstruction of solutions of the 
K\"ahler--Yang--Mills--Higgs equations \parencite[Theorem~3.6]
{AlvarezConsulGarciaFernandezGarciaPrada2019} (note, however, that 
the assumption of vanishing first Betti number in that theorem calls 
for a careful treatment, so one might expect to again encounter 
central extensions of the symplectomorphism group).
\end{remark}

\appendix
\section{Notation and conventions}
\label{sec:notation}
 \textbf{Penrose Notation.}
In \cref{sec:symplecticConnections}, we shall make extensive use of 
Penrose's abstract index notation. In this notation, indices are 
used as labels indicating the type of a tensor and \emph{do not 
denote the components of a tensor with respect to a local frame}. 
For example, a vector field is denoted by \( X^i \).
The superscript \( i \) in \( X^i \) does not refer to a particular 
component in local coordinates but serves as a label telling us 
that \( X \) is a vector field. Similarly, a \( 1 \)-form is
written as \( \alpha_j \). Contraction is indicated by labeling one 
covariant index and one contravariant 
index with the same letter, \eg, \( \alpha(X) \equiv \alpha_i X^i \).
\emph{Thus, the resulting calculus resembles the usual coordinate 
expressions but has the important advantage of being completely intrinsic 
and coordinate-free.}

Indices are raised and lowered using the symplectic form 
\( \tensor{\omega}{_i_j} \) as follows:
\begin{align}
\label{flat}
\omega^\flat: X^i &\mapsto X_i \equiv \omega_{ji} X^j \, ,   \\
\label{sharp}
\omega^\sharp: \alpha_j &\mapsto \alpha^j \equiv \varpi^{ji} \alpha_i \, ,
\end{align}
where \( \varpi^{ij} \) is the Poisson tensor\footnotemark{} associated with 
\( \omega_{ij} \) according to \( \varpi^{ik}\omega_{kj} = 
- \tensor{\delta}{^i_j} \).
\footnotetext{For a symplectic form \( \omega \) with associated Poisson 
tensor \( \varpi \), the Poisson bracket is given by \( \poisson{f}{g} = 
\omega(X_f, X_g) = \varpi(\dif f, \dif g) \), where \( X_f \) is the 
Hamiltonian vector field satisfying \( X_f \contr \omega = - \dif f \). 
We thus have
\begin{equation*}
	\poisson{f}{g}
		= \varpi^{ij} (\dif f)_i (\dif g)_j
		= \varpi^{ij} \omega_{ki} (X_f)^k \omega_{lj} (X_g)^l
		= \omega_{kl} (X_f)^k (X_g)^l.
\end{equation*}
In other words, \( \varpi^{ij} \omega_{ki} \omega_{lj} = \omega_{kl} \), 
which is equivalent to \( \varpi^{ik}\omega_{kj} = 
- \tensor{\delta}{^i_j} \).
}
Note that \( \omega^\flat \) and \( \omega^\sharp \) are inverses of 
each other. The minus sign in the definition of the Poisson tensor 
is a consequence of the skew-symmetry of \( \omega_{ij} \) and leads 
to some subtle consequences for the index calculus that are different 
from the Riemannian context. In particular, the position of the 
indices is important even if they are summed-over. For example, 
we have \( A_{ij} = \tensor{A}{_i^l} \omega_{lj}\) 
and \( B^{jk} = \varpi^{jp} \tensor{B}{_p^k} \) so that   
\begin{equation}
	A_{ij} \tensor{B}{^{jk}} 
		= \tensor{A}{_i^l} \omega_{lj} \varpi^{jp} \tensor{B}{_p^k}
		= - \tensor{A}{_i^p}\tensor{B}{_p^k} \, . 	
\end{equation}
Moreover, lowering the index of the identity map \( \tensor{\delta}{_i^j}: 
\TBundle M \to \TBundle M \) yields the \emph{skew-symmetric} map 
\( \delta_{ij} = \omega_{ij}: \TBundle M \times_M \TBundle M \to \R \).

Let \( \TensorFieldSpace^r_s(M) \) be the space of \( r \)-times 
contravariant and \( s \)-times covariant tensor fields.
An \emphDef{affine connection} on \( M \) is a linear map 
\begin{equation}
	\nabla: \VectorFieldSpace(M) \to \TensorFieldSpace^1_1(M),
	\qquad
	X^j \mapsto \nabla_i X^j,
\end{equation}
which satisfies the Leibniz rule \( \nabla_i (fX^j) = 
{(\dif f)}_i X^j + f \, \nabla_i X^j \).
The covariant derivative extends uniquely to all tensor fields 
by requiring \( \nabla _i \) to preserve the type of the tensor 
and to be a \( \R \)-linear tensor derivation, \ie, 
\( \nabla_i (t \otimes s) = (\nabla_i t) \otimes s + 
t \otimes (\nabla_i s) \) for any \( t, s \in \TensorFieldSpace(M) \), 
and to commute with contractions. In abstract Penrose index notation, 
the covariant derivative of a tensor field 
\( t^{j_1 \dotso j_p}_{k_1 \dotso k_q} \) is 
denoted by \( \nabla_i t^{j_1 \dotso j_p}_{k_1 \dotso k_q} \).

The \emphDef{Lie derivative of a connection} is defined by the requirement 
that it behaves like a derivation on all symbols, \ie, for each given 
\( X \in \VectorFieldSpace(M) \), the formula
\begin{equation}
	\difLie_X (\nabla_Y Z) = (\difLie_X \nabla)_Y Z + \nabla_{\difLie_X Y} Z 
	+ \nabla_Y \difLie_X Z,
\end{equation}
for all \( Y,Z \in \VectorFieldSpace(M) \), defines a new covariant 
derivative \( (\difLie_X \nabla)_Y \) along the vector field \( Y \).  
\smallskip

The \emphDef{torsion} of \( \nabla \) is the \( 1 \)-contravariant, 
\( 2 \)-covariant tensor field \( \tensor{T}{_i_j^k} \) defined by 
\begin{equation}
\label{torsion_definition}
\tensor{T}{_i_j^k} X^i Y^j = 
X^i \nabla_i Y^k - Y^j \nabla_j X^k - [X,Y]^k
	\quad
\text{for all}\quad X^i,Y^j \in \VectorFieldSpace(M),  
\end{equation}

The \emphDef{curvature} \( \tensor{R}{_i_j_k^l} \) of the connection 
\( \nabla \) is defined by
\begin{equation}
\label{curvature_definition}
	\tensor{R}{_i_j_k^l} Z^k = \nabla_i \nabla_j Z^l - \nabla_j \nabla_i Z^l 
	+ \tensor{T}{_i_j^k} \, \nabla_k Z^l \, .
\end{equation}

Since we will rely heavily on the Penrose notation, this is a good place
to compare it with the standard coordinate free notation. Remember,
the indices are \emph{not} coordinate components of tensors.
For example, in~\eqref{torsion_definition}, $X^i \nabla_i Y^k$ actually 
means $\nabla_X Y$. So, formula~\eqref{torsion_definition}, even though it
looks like the coordinate expression of the torsion tensor, it really
means $T(X,Y)= \nabla_X Y - \nabla_Y X - [X,Y]$, the standard coordinate
free definition of the torsion. This brings us to the interpretation 
of~\eqref{curvature_definition}, which would be a standard formula
had the sub- and superscripts been indices in a coordinate system. Note 
that~\eqref{curvature_definition} does \emph{not state that
$R(X,Y)Z=\nabla_X \nabla_Y Z - \nabla_Y \nabla_X Z + 
\nabla_{T(X,Y)}Z$, which is false}, even though one is tempted 
to interpret it in this manner. To
see how one can recover the standard definition of the curvature from
the Penrose index formula~\eqref{curvature_definition}, we multiply
both sides by $X^i Y^j$ and get (again, the indices and their position 
only reflect what kind of tensor is considered, so the index $l$ in 
the computation below is not a ``free index'' in Penrose notation; 
it only tells us that the result is a vector field and, similarly, 
we need to interpret $\nabla_X Y^j$ as $(\nabla_X Y)^j$ since the 
upper index only indicates that the expression is a vector field):
\begin{align*}
R(X,Y)Z&= X^i Y^j\tensor{R}{_i_j_k^l} Z^k = X^i Y^j\nabla_i \nabla_j Z^l -
X^i Y^j\nabla_j \nabla_i Z^l +  X^i Y^j\tensor{T}{_i_j^k} \, \nabla_k Z^l\\
&= Y^j \nabla_X \nabla_j Z^l - X^i \nabla_Y \nabla_i Z^l + 
T(X,Y)^k \nabla_k Z^l \\
&= \nabla_X(Y^j \nabla_j Z^l) - (\nabla_X Y)^j \nabla_j Z^l -
\nabla_Y(X^i \nabla_i Z^l) + (\nabla_Y X)^i \nabla_i Z^l+ 
\nabla_{T(X,Y)}Z \\
&= \nabla_X \nabla_Y Z - \nabla_{\nabla_X Y} Z -\nabla_Y \nabla_X Z + 
\nabla_{\nabla_Y X} Z + \nabla_{T(X,Y)}Z \\
&= \nabla_X \nabla_Y Z - \nabla_Y \nabla_X Z + 
\nabla_{T(X,Y)-\nabla_X Y + \nabla_Y X}Z  \\
&= \nabla_X \nabla_Y Z - \nabla_Y \nabla_X Z - 
\nabla_{[X,Y]}Z\,,
\end{align*}
which is the definition of the curvature tensor. This simple computation
illustrates the power of the Penrose notation:~\eqref{curvature_definition}
looks like the correct local formula in coordinates for the curvature
tensor, whereas, in reality, it gives an intrinsic expression of the 
curvature tensor and one can recover the classical definition after a simple 
computation. It is in this spirit that all the formulas that appear later 
on should be interpreted; they have Penrose indices, which means that they
are intrinsic, and the index free expressions can be easily obtained
after a computation analogous to the one above. 
\smallskip

\noindent{\bf Lie Group and Lie Algebra Actions.} The 
left (right) action of of a Lie group $G$ on a manifold $M$ is denoted 
by $(g,m) \mapsto g.m$ ($m.g$) for $g \in G$ and $m \in M$. The induced 
left (right) Lie algebra action of $\mathfrak{g}$, the Lie algebra of $G$, 
on $M$ is denoted by $(\xi,m) \mapsto \xi.m$ ($m.\xi$) for $\xi \in 
\mathfrak{g}$ and $m \in M$, where 
\[
\xi.m \defeq \xi^*(m) \defeq 
\left.\frac{d}{dt}\right|_{t=0} \exp(t \xi).m
\]  
is the value of the fundamental vector field (or infinitesimal 
generator) \( \xi^* \) defined by $\xi$ at $m$; analogous notation 
for a right action. Recall 
that for left (right) Lie algebra actions we have $[\xi^*, \eta^*] = 
- [\xi, \eta]^*$ ($[\xi^*, \eta^*] =[\xi, \eta]^*$).

Throughout the paper we think of $\operatorname{U}(1) 
= S^1$ as $\mathbb{R}/\mathbb{Z}$ and write hence the group multiplication
additively.

\smallskip

\noindent{\bf Conventions in Symplectic Geometry.}
Since the sign conventions in symplectic geometry are not uniform,
we specify them at the outset. The canonical one-form on the cotangent 
bundle is in canonical local cotangent bundle coordinates 
$(q^i, p_i)$ equal to $\theta =p_i {\rm d}q^i$ and the symplectic 
form is $\omega = {\rm d}\theta = {\rm d} p_i \wedge {\rm d}q^i$. 
The Hamiltonian vector field $X_h$ of a function $h$ on a general 
symplectic manifold $(M,\omega)$ is defined hence by 
${\rm d}h = - X_h \contr \omega$ and Hamilton's equations 
in Poisson bracket form are $\dot{f}=\{h,f\}$ for any smooth function,
which, in local Darboux coordinates $(q^i, p_i)$ on $M$ 
(\ie, $\omega = {\rm d} p_i \wedge {\rm d}q^i$) are the standard 
Hamilton equations $\frac{dq^i}{dt}  = \frac{\partial h}{\partial p_i}$, 
$\frac{dp_i}{dt} = -\frac{\partial h}{\partial q^i}$. 
We have $[X_f, X_g] = X_{\{f,g\}}$ for any $f,g \in C^{\infty} (M)$.

The trace of a \( 2 \)-form \( \alpha \) with respect to \( \omega \) 
is defined by first raising the second index of \( \alpha \) 
with \( \omega \) and then taking the ordinary trace of the 
resulting endomorphism of \( \TBundle M \), that is, 
\( \tr_\omega (\alpha) \defeq \tensor{\alpha}{_i^i} \).
The following formula
\begin{equation}
	\label{eq:symplectic:formWedgeOmegaNMinus1}
	\alpha \wedge \frac{\omega^{n-1}}{(n-1)!} = 
	\frac{1}{2} \tr_\omega (\alpha)\, \mu_\omega \, ,
\end{equation}
where \( \mu_\omega = \omega^n/n! \) is the volume
form on the \(2n\)-dimensional manifold \( M \) induced by \( \omega \),
will often be used in \cref{sec:kaehler,sec:symplecticConnections}; 
it is checked using a canonical basis in each tangent space.
For all \( 1 \)-forms \( \sigma, \tau \), we have
\begin{equation}
\label{eq:symplectic:symplecticHodge1Form}
\omega(\sigma, \tau) \, \mu_\omega = \sigma \wedge \tau \wedge 
\frac{\omega^{n-1}}{(n-1)!} \, .
\end{equation}

A Lie group action on the symplectic manifold $(M, \omega)$ is called
symplectic or canonical if the diffeomorphism on $M$ defined by each 
$g \in G$ preserves the symplectic form $\omega$ on $M$. This implies 
$\dif (\xi^*\contr \omega) = \difLie_{\xi^*} \omega = 0$, where $\difLie$
denotes the Lie derivative; this condition
is equivalent to the action being symplectic if the Lie group $G$ is 
connected.
We use a weakly nondegenerate pairing $\kappa: \mathfrak{g}^\ast \times  
\mathfrak{g} \rightarrow  \mathbb{R}$ and think of $\mathfrak{g}^\ast$
as the ``dual'' of $\mathfrak{g}$ (even though it is not the functional
analytic dual in infinite dimensions); nondegenerate always means weakly
nondegenerate. The momentum map $J:M \rightarrow  
\mathfrak{g}^\ast$ is defined by the requirement $\xi^\ast = X_{J_\xi}$
for any $\xi \in  \mathfrak{g}$, where $J_\xi (m):= \kappa (J(m), \xi)$
for any $m \in  M$. Thus, $J$ is infinitesimally equivariant if and only
if it is an anti-Poisson map, \ie, $\{ J_\xi , J_\eta \} + 
J_{[ \xi , \eta ]}  = 0$ for any $\xi, \eta \in  \mathfrak{g}$.

\begin{refcontext}[sorting=nyt]{}
	\printbibliography
\end{refcontext}

\end{document}